\documentclass[11pt,a4paper]{article}
\usepackage{amsmath,amssymb,amsthm,amscd,mathrsfs}
\usepackage{indentfirst}
\usepackage[top=25mm, bottom=30mm, left=30mm, right=30mm]{geometry}
\usepackage{color}
\usepackage{comment}

\newtheorem{df}{\bf Definition}[section]
\newtheorem{thm}[df]{\bf Theorem}
\newtheorem{lem}[df]{\bf Lemma}
\newtheorem{cor}[df]{\bf Corollary}
\newtheorem{prop}[df]{\bf Proposition}
\newtheorem{rem}[df]{\bf Remark}
\newtheorem{exa}[df]{\bf Example}

\newtheorem*{claim}{\bf Claim}

\newtheorem{thmA}{\bf Theorem}
\newtheorem{corA}[thmA]{\bf Corollary}

\newcommand{\R}{\mathbb{R}}
\newcommand{\C}{\mathbb{C}}
\newcommand{\Z}{\mathbb{Z}}

\newcommand{\N}{\mathbb{N}}
\newcommand{\B}{\mathbb{B}}
\newcommand{\K}{\mathbb{K}}
\newcommand{\M}{\mathbb{M}}

\newcommand{\rE}{\mathord{\text{\rm E}}}

\newcommand{\ri}{\mathrm{i}}
\newcommand{\actson}{\curvearrowright}

\newcommand{\Ad}{\operatorname{Ad}}
\newcommand{\id}{\text{\rm id}}

\newcommand{\Aut}{\operatorname{Aut}}

\newcommand{\Tr}{\mathord{\text{\rm Tr}}}

\newcommand{\ovt}{\mathbin{\overline{\otimes}}}

\newcommand{\otm}{\otimes_{\rm min}}

\newcommand{\D}{\operatorname{DSG}}

\renewcommand{\Re}{\operatorname{Re}}

\title{\bf Weak relative Dixmier property and Popa's intertwining technique for type III subfactors}
\author{Yusuke Isono\thanks{Research Institute for Mathematical Sciences, Kyoto University, 606-8502, Kyoto, Japan}}
\date{}

\begin{document}

\maketitle

\renewcommand{\thefootnote}{}
\footnote[0]{E-mail: \texttt{isono@kurims.kyoto-u.ac.jp}}
\footnote[0]{2020 \textit{Mathematics Subject Classification.} 46L10, 46L36, 46L40.}
\footnote[0]{\textit{Key words}: Weak relative Dixmier property; Popa's intertwining-by-bimodules; Type $\rm III$ factor; Tomita--Takesaki theory.}
\footnote[0]{YI is supported by JSPS KAKENHI Grant Number 24K06759.}
\renewcommand{\thefootnote}{\arabic{footnote}}

\begin{abstract}
	Let \( A \subset M \) be an inclusion of von Neumann algebras equipped with a faithful normal semifinite operator valued weight \( E \colon M \to A \). We prove that every positive element \( x \in M \) with \( E(x) < \infty \) satisfies the weak Dixmier property relative to \( A \): the \( \sigma \)-weak closure of the convex hull of its unitary orbit under \( \mathcal{U}(A) \) intersects the relative commutant \( A' \cap M \). This extends Marrakchi's result for the case of conditional expectations.
We apply this result to obtain new structural theorems for type III factors, including a reformulation of Popa's intertwining criterion without tracial assumptions, an extension of Ozawa's relative solidity theorem to the type III setting, and a Galois-type correspondence for crossed products by totally disconnected groups. The last result resolves a question posed by Boutonnet and Brothier regarding the structure of intermediate subfactors.
\end{abstract}

\section{Introduction}

Averaging techniques, particularly those involving unitary conjugation, play a fundamental role in the study of von Neumann algebras. A prominent example of this is the \emph{Dixmier property} \cite{Di49}, which asserts that any element in a von Neumann algebra lies in the norm closure of the convex hull of its unitary orbit. In the context of von Neumann algebras, a more natural formulation of this idea is given by the \emph{weak relative Dixmier property}  for inclusions $A\subset M$ \cite{Po99}:  for every element $x \in M$, the $\sigma$-weak closed convex hull of $\{uxu^* \mid u \in \mathcal{U}(A)\}$ intersects the relative commutant $A' \cap M$. Conceptually, this means that elements of $M$ can be made approximately central relative to $A$ by unitary conjugation. For further background and details on this notion and its history, see \cite[Section~0]{Po20}.

Despite its simple definition, unlike in Dixmier's original setting, the weak relative Dixmier property does not hold in general. While it trivially holds when $A=M$ by Dixmier's theorem, there are many situations in which it fails. For example, an inclusion $M\subset \B(H)$ satisfies the property if and only if $M$ is approximately finite-dimensional~\cite{Sc63,Co75}. In the type $\rm III_1$ setting, the property is intimately related to Connes' bicentralizer problem.  Haagerup showed~\cite{Ha85} that the inclusion of the continuous core into a type~$\rm III_1$ factor (via the Takesaki duality) has the weak relative Dixmier property if and only if the factor has trivial bicentralizer. This result completed the proof of the celebrated classification theorem of amenable factors. These examples show that whether the weak relative Dixmier property holds depends strongly on the structure of the inclusion.

The situation is much simpler when $A$ is finite and there exists a faithful normal conditional expectation from $M$ onto $A$. In this setting, the weak relative Dixmier property is well known to hold by standard Hilbert space techniques, and it has been an essential tool in von Neumann algebras. Indeed, Popa employed it in his analysis of maximal abelian subalgebras~\cite{Po81} and in his classification of amenable subfactors~\cite{Po95}, both of which form part of his deep contributions to the theory of von Neumann algebras. It is worth noting that Haagerup's work \cite{Ha85} on the bicentralizer problem was influenced by Popa's earlier analysis of maximal abelian subalgebras.

However, when the subalgebra $A$ is not finite, it had long been unclear whether the property holds. This problem was explicitly raised in~\cite{Po99}. A major breakthrough was made by Marrakchi~\cite{Ma19}, who solved it by proving the weak relative Dixmier property for all inclusions $A \subset M$ admitting a faithful normal conditional expectation. This result covers, in particular, the case where $A$ is of type III.
His proof introduced a novel approach based on compact convex semigroups of completely positive maps and minimal idempotents, making an essential use of Ellis' lemma.  This result was later applied in \cite{Ma23} to the study of the relative bicentralizers. That application demonstrates its importance in the broader theory of von Neumann algebras.

This paper builds on these developments by introducing a more flexible version of the weak relative Dixmier property, formulated for individual positive elements. We focus on positive elements that are ``integrable'' with respect to a given \emph{operator valued weight}, that is, a generalization of conditional expectations that allows infinite values (see Subsection~\ref{Operator valued weights}).  Our main theorem refines the weak relative Dixmier property by providing a localized criterion that still makes sense even when there is no conditional expectation from $M$ onto $A$. This new framework will be applied to several results later in the paper. We note that the applications will concern the basic construction, which is a typical situation where a canonical operator valued weight is available, but no conditional expectation exists.

We now state our main result precisely. See Theorem \ref{thm-Dixmier-expectation} for a more general statement, which includes a version involving ucp maps.

\begin{thmA}\label{thmA}
	Let $A\subset M$ be an inclusion of von Neumann algebras and assume that there exists a faithful normal semifinite operator valued weight $E_A\colon M \to A$. Then for every positive element $x\in M$ with $E_A(x)<\infty$, the $\sigma$-weak closure of the convex hull
	$$\mathcal K(x,A):=\overline{\mathrm{conv}}^{\rm weak}\{ u x u^* \mid u\in \mathcal U(A)\} \subset M$$
intersects $A'\cap M$. 
\end{thmA}

As an immediate consequence, we obtain the following corollary. Haagerup characterized the case when $E_A$ is semifinite on $A'\cap M$ \cite{Ha77b}. The following corollary provides a characterization of the opposite phenomenon in terms of the weak relative Dixmier property. This characterization is also related to Popa's intertwining technique, which we will discuss later (see Remark \ref{rem-intertwining2}).

\begin{corA}\label{corB}
	Keep the setting in Theorem \ref{thmA}. The following conditions are equivalent.
\begin{enumerate}
	\item We have $E_A(x)=\infty$ for all nonzero positive $x\in A'\cap M$. 

	\item For every $x\in M^+$ with $E_A(x)<\infty$, $\mathcal K(x,A)$ contains $0$.

	\item For every $x\in M^+$ with $E_A(x)<\infty$, we have $A'\cap \mathcal K(x,A)=\{0\}$.
\end{enumerate}
\end{corA}

When $A$ is semifinite, Theorem~\ref{thmA} can be proved using standard Hilbert space methods. The main novelty of this paper is therefore that we treat the case where $A$ is of type~III. In the proof, we first follow Marrakchi's approach~\cite{Ma19} and, by making use of Ellis' lemma, construct a suitable ucp map, which can be approximated by $\Ad(u)$ for some $u \in \mathcal U(A)$. Marrakchi proved that this map is faithful on $M$ and used this property in his argument. In our setting, the map is no longer faithful, and a new idea is required. 
We study the behavior of the ucp map on the $\ast$-subalgebra of $M$ on which the operator valued weight $E_A$ is finite, and this part is our contribution. The details will be given in Proposition~\ref{prop-averaging-typeIII}.

\subsection*{Application 1: Popa's intertwining-by-bimodules}

We apply our framework to reformulate Popa's intertwining theory for general type III inclusions. Popa's \emph{deformation/rigidity theory} has played a central role in the study of non-amenable von Neumann algebras over the past two decades; see the surveys \cite{Po06b,Va10,Io17}. One of its main tools is the \emph{intertwining-by-bimodules} technique \cite{Po01,Po03}. In the tracial setting, this method provides an analytic criterion for the embedding of one subalgebra into another, formulated in terms of the convergence of unitary sequences. These characterizations are convenient and widely used.

In the type III setting, however, such techniques are harder to apply due to the lack of traces and the associated Hilbert space structure. Although many partial results have been obtained \cite{CH08,HR10,HV12,Ue12,Is14,HI15,Ue16,BH16,Is19}, a complete extension had remained an open problem. In this paper, we address this gap by proving a new formulation of the intertwining condition that applies to arbitrary subalgebras.

Our Theorem A provides two equivalent conditions: one involving a net of unitaries, and the other involving a completely positive map that satisfies a weak Dixmier type condition (see items (2) and (3) of Theorem C below). Item (2), formulated in terms of nets $(u_i)_i$, is the natural analogue of the finite case, but in the type III setting it often turns out to be less useful. By contrast, item (3) is more effective, since the associated ucp map can be chosen to be normal on $M$. A more detailed explanation of the difference between (2) and (3) will be given right after Theorem \ref{thmC}.

For simplicity, we introduce here the version formulated for unital subalgebras of type III. For a general statement, see Theorem \ref{thm-intertwining1} and Proposition \ref{prop-intertwining1}. 
We say that a (possibly non-unital) inclusion of von Neumann algebras $A\subset M$ is \emph{with expectation} if there exists a faithful normal conditional expectation from $1_AM1_A $ onto $A$. 

\begin{thmA}\label{thmC}
	Let $A,B\subset M$ be unital inclusions of $\sigma$-finite von Neumann algebras with expectation. Assume that $A,B$ are of type $\rm III$. Then the following conditions are equivalent. 
\begin{enumerate}
	\item We have $A \preceq_M B$, see Definition \ref{def-intertwining1}.
	\item There exists no net $(u_i)_{i \in I}$ of unitaries in $\mathcal U(A)$ such that 
	$$E_{B}(b^*u_i a)\rightarrow 0\quad \text{in the strong topology for all}\quad  a,b\in M,$$
where $E_B$ is a fixed faithful normal conditional expectation from $M$ onto $B$.

	\item There exists no ucp map $\Psi$ on the basic construction $\langle M,B\rangle$, which is approximated by convex combinations of $\Ad(u)$ for $u\in \mathcal U(A)$, such that $\Psi|_M$ is normal and 
	$$\Psi(ae_{B} b^*) = 0\quad \text{for all}\quad a,b\in M,$$
where $e_B$ is the Jones projection of $E_B$.

\end{enumerate}
\end{thmA}

In item (3) above, we use the topology of pointwise $\sigma$-weak convergences for ucp maps.

\smallskip

\noindent\textbf{Comment on item (3).}
Suppose that item (2) does not hold. Take a net $(u_i)_i$ and let $\omega$ be a cofinal nonprincipal ultrafilter on $I$. Define $\Psi := \lim_{i\to \omega} \Ad(u_i)$. It is straightforward to check that $\Psi(ae_Bb^*)=0$ for all $a,b\in M$. Moreover, if the net $(u_i)_i$ represents an element of the ultraproduct $M^{\omega}$, then $\Psi$ is normal on $M$, and hence $\Psi$ satisfies the requirements of item (3).

In general, however, $(u_i)_i$ does not give rise to an element of $M^\omega$, so normality may fail. Item (3) asserts that by allowing convex combinations of such adjoint maps, we can nevertheless obtain a normal ucp map. Since $\Psi$ in item (3) is approximated by convex combinations of maps of the form $\Ad(u)$ with $u\in \mathcal U(A)$, this $\Psi$ can serve as a substitute for the original net $(u_i)_i$. We will demonstrate this idea in the proof of Lemma~\ref{lem-uniform-Bernoulli}.

\subsection*{Application 2: Rigidity theorems via Theorem~\ref{thmC}}

Recall that a diffuse von Neumann algebra $M$ is \emph{solid} \cite{Oz03} if for every diffuse von Neumann subalgebra $A\subset M$ with expectation, the relative commutant $A'\cap M$ is amenable. In this case, every non-amenable subfactor $N\subset M$ with expectation is \emph{prime}, meaning that  it does not admit a decomposition $N=P\ovt Q$ into a tensor product of two diffuse von Neumann algebras. Thus, solidity may be regarded as a strong form of indecomposability with respect to tensor product decompositions. More generally, for an inclusion $B\subset M$ of von Neumann algebras with expectation,  we say that  \emph{$M$ is solid relative to $B$} if for any projection $p\in M$ and von Neumann subalgebra $A\subset pMp$ with expectation such that $A\not\preceq_MB$, the relative commutant $A'\cap pMp$ is amenable. 

We present two independent applications of Theorem~\ref{thmC}, each of which yields relatively solid type~III factors.

\smallskip

\noindent\textbf{Application to Popa's spectral gap rigidity.}
Popa's \emph{malleable deformation} and \emph{spectral gap rigidity} principle \cite{Po03,Po06a} has been a powerful tool in the analysis of tracial von Neumann algebras, but its extension to type III settings has proved challenging, particularly for type III subalgebras. Partial generalizations have been obtained in \cite{HI15b,Ma16,HI17}. Using Theorem~\ref{thmC} and making use of Marrakchi's aforementioned results on relative bicentralizers~\cite{Ma23}, we overcome these difficulties and establish a spectral gap argument that applies to type III subalgebras with expectation.

As an application, we obtain the following rigidity theorem for Bernoulli crossed products, extending the main results of~\cite{Po03,CI08,Ma16}.

\begin{thmA}\label{thmD}
	Let $\Gamma$ be a discrete group, $(B_0,\varphi_0)$ an amenable von Neumann algebra with a faithful normal state, and consider the Bernoulli shift action 
	$$ \Gamma \actson \bigotimes_{\Gamma} (B_0,\varphi_0) =:B .$$
Let $\Gamma \actson N$ be any action on an amenable $\sigma$-finite von Neumann algebra and consider the crossed product $M = (N\ovt B)\rtimes \Gamma$ with respect to the diagonal action. Then $M$ is solid relative to $N\rtimes \Gamma$.  
\end{thmA}

\smallskip

\noindent\textbf{Application to Ozawa's relative solidity theorem.}
Ozawa's solidity theorem~\cite{Oz03} asserts that von Neumann algebras arising from \emph{biexact groups} (e.g.\ hyperbolic groups) are solid. In~\cite{Oz04}, he strengthened the result to relative solidity, showing that crossed products by actions on abelian algebras are solid relative to the base algebra. Further generalizations to type III algebras were obtained in \cite{HV12,Is12}. Recently, Ding and Peterson \cite{DP23} introduced the notion of \emph{biexact von Neumann algebras}, providing a unified framework that extends Ozawa's solidity to this general setting. In all these results, however, the target subalgebra was assumed to be finite.

Using Theorem~\ref{thmC}, we remove this restriction and obtain relative solidity for type~III crossed products. The proof follows Ozawa's original strategy, but is carried out in the framework of biexact von Neumann algebras.

\begin{thmA}\label{thmE}
	Let $\alpha \colon \Gamma \actson B$ be an action of a discrete group on an amenable $\sigma$-finite von Neumann algebra. Put $M:=B\rtimes_\alpha \Gamma$. Assume that $\Gamma$ is biexact. Then $M$ is solid relative to $B$. 
\end{thmA}

We will discuss concrete examples of relatively solid type III factors in Subsection~\ref{Examples_prime}.

\subsection*{Application 3: Galois correspondences}

We further investigate the so-called \emph{Galois correspondence} for intermediate subfactors arising from group actions. In the case of discrete groups, the situation is well understood: it was shown by Choda \cite{Ch78} and later by Izumi--Longo--Popa \cite{ILP96} that, for outer actions on $\sigma$-finite factors, every subfactor of the crossed product that contains the base algebra is itself a crossed product by a subgroup. This phenomenon is often regarded as a noncommutative analogue of the classical Galois correspondence.

In contrast, for locally compact groups, the problem becomes significantly more subtle, since conditional expectations onto intermediate subfactors may not exist. Boutonnet and Brothier studied this question for totally disconnected groups and proved that a Galois-type correspondence holds under a strong outerness assumption \cite{BB17}. Their result also requires the presence of a semifinite trace or a suitable invariant state, and provides examples covering various types of group actions. However, their methods do not apply when such structures are absent. Most notably, they do not cover actions on type~$\mathrm{III}_0$ factors.

Using our main theorem, we succeed in removing assumptions of trace or invariant state, and show that the Galois correspondence still holds under the same dynamical conditions. This applies even in the type $\mathrm{III}_0$ case.

\begin{thmA}\label{thmF}
	Let $G\actson B$ be an action of a totally disconnected group on a $\sigma$-finite factor $B$. Assume that $\alpha$ is properly outer relative to a compact open subgroup $K_0\leq G$ whose action is minimal, see Section \ref{Galois correspondences}. Then every subfactor between $B$ and $B\rtimes G$ is of the form that $B\rtimes H$ for some closed subgroup $H\leq G$. 
\end{thmA}

This result answers a question posed by Boutonnet and Brothier \cite[QUESTION 1.5]{BB17} and provides new examples of type~III crossed products exhibiting the intermediate subfactor property. It also demonstrates the strength of our elementwise averaging techniques, which remain effective even in the absence of traces and conditional expectations.

\subsection*{Acknowledgments}
	The author would like to thank Amine Marrakchi and Sorin Popa for useful comments on the first draft of the article.

\tableofcontents

\section{Preliminaries}\label{Preliminaries}

	Let $M$ be a von Neumann algebra and $\varphi$ a faithful normal semifinite weight on $M$. We use two norms: $\|\, \cdot \, \|_\infty$ is the operator norm of $M$, and $\|\, \cdot \, \|_{2,\varphi}$ (or $\| \, \cdot \, \|_{\varphi}$) is the $L^2$-norm by $\varphi$ on $L^2(M,\varphi)$. For $\theta\in \Aut(M)$, put $\theta(\varphi):=\varphi\circ \theta^{-1}$. The \textit{modular operator, conjugation}, and \textit{action} of $\varphi$ are denoted by $\Delta_\varphi$, $J_\varphi$, and $\sigma^\varphi$ respectively. The \emph{continuous core} (with respect to $\varphi$) is the crossed product von Neumann algebra $C_\varphi(M) := M\rtimes_{\sigma^\varphi}\R$, and the \emph{centralizer algebra} $M_\varphi\subset M$ is the fixed point subalgebra of the modular action. See \cite[Chapter VIII]{Ta03} for definitions of all these objects.

Let $G$ be a locally compact group with a fixed left Haar measure $\mu_G$. Let $\alpha \colon G\actson B$ be a continuous action on a von Neumann algebra $B\subset \B(H)$. We denote the crossed product von Neumann algebra by $B\rtimes_\alpha G \subset \B(L^2(G,H))$, which is generated by $\pi_\alpha(x)$ and $\lambda_g$ for $x\in B$ and $g\in G$, where 
	$$ (\pi_\alpha(x)\xi)(s) = \alpha_s^{-1}(x)\xi(s),\quad (\lambda_g\xi)(s)=\xi(g^{-1}s) $$
for all $\xi \in L^2(G,H)$. Here we identify $1\otimes \lambda_g$ on $H\otimes L^2(G)$ with $\lambda_g$ on $L^2(G,H)$.

\subsection{Operator valued weights}
\label{Operator valued weights}

Let $M$ be a von Neumann algebra with the set of positive elements $M^+\subset M$. The \textit{extended positive cone} $\widehat{M}^+$ of $M$ is defined as the set of all the lower semicontinuous functions $m\colon M_*^+\rightarrow [0,\infty]$ satisfying
\begin{itemize}
	\item $m(\varphi+\psi)=m(\varphi)+m(\psi)$ for all $\varphi,\psi\in M_*^+$,
	\item $m(\lambda \varphi)=\lambda m(\varphi)$ for all $\varphi\in M_*^+$ and all $\lambda\geq 0$. 
\end{itemize}
We have a natural inclusion $M^+ \subset \widehat{M}^+$ via the evaluation map. An element $m\in \widehat{M}^+$ is contained in $M^+$ if and only if $m(\varphi)<\infty$ for all $\varphi\in M_*^+$. In this case, we write $m\in M^+$ or $\|m\|_\infty <\infty$.

Let $B\subset M$ be a von Neumann subalgebra. We say that a map $T \colon M^+\rightarrow \widehat{B}^+$ is an \textit{operator valued weight} from $M$ to $B$ \cite{Ha77a,Ha77b} if it satisfies the following three conditions:
\begin{itemize}
	\item $T(\lambda x)=\lambda T(x)$ for all $x\in M^+$ and all $\lambda \geq 0$, 
	\item $T(x+y)=T(x)+T(y)$ for all $x,y \in M^+$,
	\item $T(b^*xb)=b^*T(x)b$ for all $x\in M^+$ and all $b\in B$.
\end{itemize}
Note that when $B=\C$, these conditions reduce to the usual definition of a \textit{weight} on $M$.  For an operator valued weight $T : M^+\rightarrow \widehat{B}^+$, we put 
\begin{align*}
	\mathfrak{n}_{T}&:= \left\{x\in M\mid \|T(x^*x)\|_\infty<+\infty \right\};\\
	\mathfrak{m}_{T}&:=(\mathfrak{n}_{T})^*\mathfrak{n}_{T}= \left\{\sum_{i=1}^nx_i^*y_i \mid n \geq 1, x_i, y_i\in \mathfrak{n}_{T} \text{ for all } 1 \leq i \leq n \right\}.
\end{align*}
Then, $\mathfrak{m}_{T}$ is linearly spanned by the positive part $\mathfrak{m}_{T}^+:= \mathfrak{m}_T \cap M^+$, which coincides with the set of all positive elements $x\in M^+$ such that $\|T(x)\|_\infty <\infty$. 
The operator valued weight $T$ has a unique extension $T\colon \mathfrak{m}_{T}\rightarrow B$ as a $B$-$B$-bimodular linear map. In particular, $T$ extends to a conditional expectation, if $T(1_M)=1_B$. The operator valued weight $T$ is said to be 
\begin{itemize}
\item \textit{faithful} if $T(x)=0$ $\Rightarrow$ $x=0$, ($x\in M^+$),
\item \textit{normal} if  $T(x_i)\nearrow T(x)$ \quad whenever $x_i\nearrow x$, ($x_i,x\in M^+$),
\item \textit{semifinite} if $\mathfrak{m}_{T}$ is $\sigma$-weakly dense in $M$.
\end{itemize}
In this paper, all the operator valued weights are assumed to be faithful, normal and semifinite. We will say that a unital inclusion $B\subset M$ of von Neumann algebras is \textit{with operator valued weight} if there is an operator valued weight $E_B \colon M \to B$. For more on operator valued weights, we refer the reader to \cite{Ha77a,Ha77b}.

\subsection{Weak Dixmier semigroups}\label{Weak Dixmier semigroups}

For a von Neumann algebra $M$, we denote by $\B(M)$ the set of all bounded operators on $M$. It is a dual Banach space, so it has a weak$\ast$ topology. On the unit ball of $\B(M)$, this topology  coincides with  the \textit{point $\sigma$-weak topology}, that is defined by pointwise convergences in the $\sigma$-weak topology (e.g.\ \cite[Theorem 1.3.7]{BO08}). In particular, the unit ball is compact in the point $\sigma$-weak topology. We denote by $\mathrm{CCP}(M)\subset \B(M)$ (resp.\ $\mathrm{UCP}(M)\subset \B(M)$) the set of all contractive completely positive maps (resp.\ unital cp maps) on $M$.

Following \cite{Ma19}, we define the \textit{weak Dixmier semigroup} associated with the inclusion $A\subset M$ as 
	$$\D(A\subset M):= \overline{\mathrm{conv}} \{ \Ad(u) \mid u\in \mathcal U(A)\}\subset \mathrm{UCP}(M),$$
where the closure is taken by the point $\sigma$-weak topology. Put $K:=\D(A\subset M)$ and observe that it is a compact set that is closed under the composition. 
For each fixed $\theta_2\in K$, the composition $K \ni \theta_1\mapsto \theta_1\circ \theta_2\in K$ is continuous. Note that the continuity of $\theta_2$ (for a fixed $\theta_1$) holds if $\theta_1$ is normal. Thus, $K$ forms a \textit{compact (convex) semigroup} (see \cite[Definition 3.4 and Example 3.5]{Ma19}). An element $e\in K$ is called an \textit{idempotent} if $e^2=e$. For two idempotents $e,f\in K$, a partial order $e<f$  is defined by $ef=fe=e$. A \textit{minimal idempotent} is one that is minimal with respect to this partial order. A minimal idempotent always exists by Ellis' lemma.

The following results are due to Marrakchi. Recall that for a ucp map $E\colon M \to M$, the \textit{Choi--Effros product} \cite{CE97} on $E(M)$ is defined by $x\cdot y:=E(xy)$ for all $x,y\in E(M)$. With this product, $E(M)$ becomes a $C^*$-algebra.

\begin{prop}\label{prop-Marrakchi}
	The following assertions hold.
\begin{itemize}

	\item \cite[Remark 3.7]{Ma19} If $e\in K$ is a minimal idempotent and if $x\in K$, then $ex$ and $xe$ are also minimal idempotents.

	\item \cite[Corollary 3.8]{Ma19} If $\Phi,\Psi\in K$ are two minimal idempotents, then 
	$$\Psi|_{\Phi(M)}\colon \Phi(M)\to \Psi(M)$$
 is a $\ast$-isomorphism with inverse $\Phi|_{\Psi(M)}$, where $\Phi(M)$, $\Psi(M)$ are $C^*$-algebras by Choi--Effros products.

\end{itemize}

\end{prop}

The following result is the main theorem of \cite{Ma19}. As we will need a more precise formulation, we include it here.

\begin{thm}[{\cite{Ma19}}]\label{thm-Marrakchi}
	Let $A\subset M$ be an inclusion of von Neumann algebras with expectation. Then there exists a conditional expectation $E\colon M \to A'\cap M$ that is contained in $\D(A\subset M)$. 

If $A$ has no direct summand that is semifinite and properly infinite, then $E$ can be chosen as a faithful and normal one. More precisely, if $A=A_1\oplus A_2$ is the decomposition such that $A_1$ is of type $\rm III$ and $A_2$ is finite, and if $\varphi\in A_*$ is a faithful normal state such that $\varphi|_{A_2}$ is a trace, then there exists $E\colon M \to A'\cap M\in \D(A\subset M)$ such that $\varphi \circ E = \varphi$. 

\end{thm}
\begin{proof}
	The first part is proved in \cite[MAIN THEOREM]{Ma19}. For the second part, if $A$ is finite, then this result is classical and well known (see Lemma \ref{lem-averaging-trace1}). If $A$ is of type III, then it is proved in \cite[Theorem 5.2]{Ma23}, see the proof of this theorem. For the general case, $A$ is decomposed into a direct sum of a finite and a type III subalgebras, and we can reduce the problem into each case, see the argument in Lemma \ref{lem-projection-unitary2}.
\end{proof}

We note that if $A$ is semifinite and properly infinite, then any conditional expectation onto $A'\cap M$ that is contained in $\D(A\subset M)$ vanishes all finite projections in $A$ (e.g.\ \cite[3.2 Approximation Theorem]{SZ99}). In particular, such a conditional expectation can not be normal. For this reason, we will often assume that $A$ has no direct summand that is semifinite and properly infinite.

\section{Weak relative Dixmier property}

Let $M$ be a von Neumann algebra and $G$ a discrete group. We do \textit{not} assume $G$ is countable. Consider a discrete group action $\alpha \colon G\actson M$. We are mainly interested in the following two cases:
\begin{itemize}
	\item let $A\subset M$ be a von Neumann subalgebra, and define $G = \mathcal U(A)$ and $\alpha = \Ad\colon \mathcal U(A)\actson M$; 

	\item $G$ is a locally compact group acting continuously on $M$.

\end{itemize}
In this section, we are only interested in the fixed points under the actions, so we can regard these two actions as discrete group actions.

\subsection{Weak Dixmier subspaces}

	Let $\alpha \colon G \actson M$ be a discrete group action. We begin by extending the notion of weak Dixmier semigroups to our setting. Put 
	$$\D_{\rm alg}(\alpha):= \mathrm{conv} \{ \alpha_g \mid g\in G\}\subset \mathrm{UCP}(M) .$$
\begin{df}\upshape
	The \textit{weak Dixmier semigroup for $\alpha \colon G \actson M$} is the closure in the point $\sigma$-weak topology of $\D_{\rm alg}(\alpha)$. We denote it by $\D(\alpha\colon G \actson M)$ or $\D(\alpha)$. 
\end{df}

Note that $\D(\alpha)$ is a compact (convex) semigroup, as discussed in Subsection \ref{Weak Dixmier semigroups}. For each $x\in M$, we put 
	$$ \mathcal K(x,\alpha):=\overline{\mathrm{conv}}^{\rm weak} \{ \alpha_g (x) \mid g\in G\} \subset M.$$
where the closure is taken with respect to the $\sigma$-weak topology. Observe that $\alpha$ acts on $ \mathcal K(x,\alpha)$ as convex maps, so that the set of fixed points $\mathcal K(x,\alpha)^\alpha$ is naturally defined. 

We introduce the following notion.

\begin{df}\upshape
	
The \textit{weak Dixmier subspace }$\mathcal D(\alpha)$ is defined as the set of all $x\in M$ that satisfy
	$$ \mathcal K(\theta(x),\alpha) ^\alpha \neq \emptyset\quad \text{for all} \quad \theta\in \D(\alpha).$$
\end{df}

We will show in  Corollary \ref{cor-subspace} that $\mathcal D(\alpha)\subset M$ is indeed a subspace. It is straightforward to check that 
	$$M^\alpha \mathcal D(\alpha)M^\alpha \subset \mathcal D(\alpha)\quad\text{and}\quad \theta(\mathcal D(\alpha))\subset \mathcal D(\alpha)\quad\text{for all} \quad\theta\in \D(\alpha).$$
We often use the following elementary facts. These follow easily from the compactness of  $\D(\alpha)$. 

\begin{itemize}
	\item We have $ \mathcal K(x,\alpha) = \D(\alpha)x $ for all $x\in M$.

	\item If $\theta_1,\theta_2 \in \D(\alpha)$, then $\theta_1\circ \theta_2 \in \D(\alpha)$. 

	\item Let $M\subset \widetilde{M}$ be an inclusion of von Neumann algebras with an extended action $\alpha \colon G \actson \widetilde{M}$. If $\theta \in \D(\alpha\colon G\actson M)$, then there exists $\widetilde{\theta}\in \D(\alpha \colon G \actson \widetilde{M})$ such that $\widetilde{\theta}|_M = \theta$. 

\end{itemize}
To see the last statement, observe first that elements in $\D_{\rm alg}(\alpha\colon G\actson M)$ is naturally extended on $\widetilde{M}$, and then use the compactness of $\D(\alpha\colon G\actson \widetilde{M})$.

Let $A\subset M$ be a von Neumann subalgebra and define a group action $\Ad_A \colon \mathcal U(A) \actson M$. In this case, we often use the following natural notation: 
\begin{align*}
	&\D(\Ad_A) = \D(\Ad_A^M) = \D(A\subset M);\\
	&\mathcal D(\Ad_A )=  \mathcal D(\Ad_A^M) = \mathcal D(A\subset M);\\
	& \mathcal K(x,\Ad_A) = \mathcal K(x,A),\quad \mathcal K(x,\Ad_A)^{\Ad_A} = A'\cap \mathcal K(x,A).
\end{align*}

The following theorem is the main observation of this article. We will prove it in Subsection \ref{Proof of Theorem thmA}. Once we prove it, Theorem \ref{thmA} immediately follows from this theorem.

\begin{thm}\label{thm-Dixmier-expectation}
	Let $A\subset M$ be an inclusion of von Neumann algebras with operator valued weight $E_A\colon M\to A$. Then we have $\mathfrak m_{E_A}\subset \mathcal D(A\subset M)$. 
In particular, there exists a ucp map $\Psi \in \D(A\subset M)$ such that 
	$$ \Psi(  \mathfrak m_{E_A} ) \subset A'\cap M.$$
\end{thm}

In the rest of this subsection, we prove several elementary properties of $\D(\alpha)$ and $\mathcal D(\alpha)$.

\subsection*{Schwartz property}

The next lemma is essentially due to Schwartz \cite{Sc63}.

\begin{lem}\label{lem-Schwartz}
	Let $\alpha\colon G\actson M$ be a discrete group action on a von Neumann algebra. Then there exists $\Psi\in \D(\alpha)$ such that 
	$$\Psi(x)\in M^\alpha ,\quad \text{for all}\quad x\in \mathcal D(\alpha).$$
More generally, if there exist a subset $M_0\subset M$ and $\Psi_0\in \D(\alpha)$ such that $\Psi_0(M_0)\subset M^\alpha$ and $\theta(M_0)\subset M_0$ for all $\theta\in \D(\alpha)$, then $M_0\subset \mathcal D(\alpha)$ and the above $\Psi$ can be chosen as satisfying $\Psi|_{M_0}=\Psi_0|_{M_0}$.
\end{lem}
\begin{proof}
Fix $M_0 \subset M$ and $\Psi_0$ as in the statement. Note that we can choose $M_0=0$ and $\Psi_0 = \id$, so such $M_0$ and $\Psi_0$ exist.

If $x\in M_0$ and $\theta\in \D(\alpha)$, then $\theta(x)\in M_0$ and $\Psi_0(\theta(x))\in M^\alpha$ by assumption. Since $\Psi_0(\theta(x)) \in \mathcal K(\theta(x),\alpha)$, we get $\mathcal K(\theta(x),\alpha)^\alpha \neq \emptyset$. This implies $M_0\subset \mathcal D(\alpha)$. 

We will apply Zorn's lemma to the following set:
	$$\D(\alpha,\Psi_0,M_0):= \{ \theta \in \D(\alpha)\mid \theta|_{M_0} = \Psi_0|_{M_0}\}.$$
This is compact in the point $\sigma$-weak topology. For each $\theta_1,\theta_2\in \D(\alpha)$, we define $\theta_1\leq \theta_2$ by the condition
\begin{align*}
	\mathcal K(\theta_2(x),\alpha) \subset \mathcal K(\theta_1(x),\alpha)\quad \text{for all}\quad x\in M.
\end{align*}
Equivalently, $\theta_2(x)\in \mathcal K(\theta_1(x),\alpha)$ for all $x\in M$. 

	Observe that $\D(\alpha,\Psi_0,M_0)$ is nonempty, since it contains $\Psi_0$. Let $\{\theta_{i}\}_{i\in I}$ be a chain in $\D(\alpha,\Psi_0,M_0)$. Then since $\D(\alpha,\Psi_0,M_0)$ is compact, a subnet of $\{\theta_{i}\}_{i\in I}$ converges to some element $\theta_\infty\in\D(\alpha,\Psi_0,M_0)$. 
Fix $x\in M$ and $i\in I$. Then for each $i\leq j$, one has
	$$\theta_j(x)\in \mathcal K(\theta_j(x),\alpha) \subset \mathcal K(\theta_i(x),\alpha),$$
so that $\theta_\infty(x)\in \mathcal K(\theta_i(x),\alpha)$. We get $\theta_i\leq \theta_\infty $. Thus we can apply Zorn's lemma. 

Take a maximal element $\Psi\in \D(\alpha,\Psi_0,M_0)$. We have to show that $\Psi(x)\in M^\alpha$ for all $x\in \mathcal D(\alpha)$. Fix $x_0\in \mathcal D(\alpha)$ and we show $\Psi(x_0)\in M^\alpha$. Since $\mathcal K(\Psi(x_0),\alpha)^\alpha \neq \emptyset$, one can take an element $y_0\in \mathcal K(\Psi(x_0),\alpha)^\alpha$. 
There exist $f_i\in \D_{\rm alg}(\alpha)$ such that $f_i(\Psi(x_0)) \to y_0$ in the $\sigma$-weak topology. Then since $f_i\circ \Psi$ is contained in $\D(\alpha,\Psi_0,M_0)$, one can find a subnet of $\{f_i\circ \Psi\}_i$, which converges to some element $\Psi_\infty\in\D(\alpha,\Psi_0,M_0)$. Then for each $x\in M$ and $i$, one has $ f_i\circ \Psi(x)\in \mathcal K(\Psi(x),\alpha)$, so that $\Psi_\infty(x)\in \mathcal K(\Psi(x),\alpha)$. This implies $\Psi\leq \Psi_\infty $ and $\Psi = \Psi_\infty$ by the maximality. We get $\Psi(x_0)=\Psi_\infty(x_0)=y_0\in M^\alpha$ as desired. 
\end{proof}

\begin{cor}\label{cor-subspace}
	Let $\alpha\colon G\actson M$ be a discrete group action on a von Neumann algebra. The following assertions hold true.
\begin{enumerate}
	\item For each $x\in M$, the following conditions are equivalent.
\begin{enumerate}
	\item[$\rm (i)$] We have $ x\in \mathcal D(\alpha)$. \quad $(\Leftrightarrow$\  $ \mathcal K(x,\alpha)=\D(\alpha)x\subset \mathcal D(\alpha).)$
	\item[$\rm (ii)$] There exists $\Psi\in \D(\alpha)$ such that $\Psi(\mathcal K(x,\alpha)) \subset M^\alpha$.
\end{enumerate}

	\item The weak Dixmier subspace $\mathcal D(\alpha)\subset M$ is a subspace that is closed in the operator norm topology.

\end{enumerate}
\end{cor}
\begin{proof}
	(1) (i)$\Rightarrow$(ii) This follows by Lemma \ref{lem-Schwartz}.

	(ii)$\Rightarrow$(i) Let $\Psi$ be as in the statement. If $\theta\in \D(\alpha)$, then $\theta(x)\in \mathcal K(x,\alpha)$, hence $\Psi(\theta(x))\in M^\alpha$. 
Since $\Psi\in \D(\alpha)$, we get $\Psi(\theta(x))\in \mathcal K(\theta(x),\alpha)^\alpha$ and $x\in \mathcal D(\alpha)$.

	(2) Fix $\Psi$ as in Lemma \ref{lem-Schwartz}. Take any $x,y\in  \mathcal D(\alpha)$ and observe that $\Psi(\mathcal K(x,\alpha)),\Psi(\mathcal K(y,\alpha))\subset M^\alpha$. We apply $\Psi$ to
	$$ \mathcal K(x+y,\alpha) \subset \mathcal K(x,\alpha)+\mathcal K(y,\alpha)$$
and get that $\Psi(\mathcal K(x+y,\alpha)) \subset M^\alpha$. By item (1), this means $x+y\in \mathcal D(\alpha)$. Thus, $\mathcal D(\alpha)$ is a subspace of $M$. 

If $\mathcal D(\alpha)\ni x_n \to x\in M$ is a convergence in the operator norm of $M$. Fix $\theta\in \D(\alpha)$. Then since $\theta(x_n)\in \mathcal K(x_n ,\alpha)\subset \mathcal D(\alpha)$, we get $\Psi(\theta(x_n))\in M^\alpha$. This implies $\Psi(\theta(x))\in M^\alpha$ and hence $\Psi(\theta(x))$ is contained in $\mathcal K(\theta(x),\alpha)^\alpha$. We get $x \in \mathcal D(\alpha)$. Thus, $\mathcal D(\alpha)$ is closed in the operator norm.
\end{proof}

The next lemma will be used later in application.

\begin{lem}\label{lem-Schwartz2}
	Let $\alpha\colon G\actson M$ be a discrete group action on a von Neumann algebra. 
Let $M_0 \subset M$ be a subset satisfying $\theta(M_0)\subset M_0$ for all $\theta\in \D(\alpha)$. If $\mathcal K (x,\alpha)$ contains $0$ for all $x\in M_0$, then there exists $\Psi\in \D(\alpha)$ such that $\Psi(M_0) = 0$.
\end{lem}
\begin{proof}
	Let $x_1,x_2,\ldots,x_n \in M_0$ be any elements. By assumption, take $\theta_1\in \D(\alpha)$ such that $\theta_1(x_1) = 0$. Then since $\theta_1(x_2)$ is in $M_0$ by assumption, there exists $\theta_2 \in \D(\alpha)$ such that $\theta_2(\theta_1(x_2))=0$. We automatically have $\theta_2(\theta_1(x_1))=0$. We repeat this procedure and get $\theta_1,\ldots,\theta_n\in \D(\alpha)$ such that, with $\theta:=\theta_n \circ \cdots \circ \theta_1\in D(\alpha)$, $\theta(x_k)=0$ for all $k=1,\ldots,n$. Thus for any finite subset $F\subset M_0$, there exists $\theta_F\in \D(\alpha)$ such that $\theta_F (F)=0$. Consider a net $\{\theta_F\}_{F}$ and then by the compactness of $\D(\alpha)$, a subnet of $\{\theta_F\}_{F}$ converges to $\Psi \in \D(\alpha)$. Then $\Psi$ does the work.
\end{proof}

\subsection*{Reductions by projections}

\begin{lem}\label{lem-projection-action1}
	Let $\alpha\colon G\actson M$ be a discrete group action on a von Neumann algebra.  Let $p\in M^\alpha$ be a projection and consider the reduced action $\alpha^p\colon G\actson pMp$.
\begin{enumerate}
	\item We have $p \mathcal K(x,\alpha)p=\mathcal K(pxp,\alpha) = \mathcal K(pxp,\alpha^p)$ for all $x\in M$.

	\item The restriction map $\D(\alpha)\ni \theta \mapsto \theta|_{pMp}\in \D(\alpha^p)$ is surjective.

	\item We have $\mathcal D(\alpha^p)=p\mathcal D(\alpha)p \subset \mathcal D(\alpha)$.

\end{enumerate}
\end{lem}
\begin{proof}
	(1) It is obvious that $p \mathcal K(x,\alpha) p\subset  \mathcal K(pxp,\alpha)=\mathcal K( pxp, \alpha^p)$. To see $\mathcal K(pxp,\alpha)\mathcal \subset p \mathcal K(x,\alpha) p$, since  generators of $\mathcal K(pxp,\alpha)$ are contained in $p \mathcal K(x,\alpha) p$, we only need to check that $p \mathcal K(x,\alpha) p$ is a closed set. But this is easy because $ \mathcal K(x,\alpha) $ is a compact set.

	(2) It is easy to see that the restriction map is well defined. Let $\theta\in \D(\alpha^p)$ and take a net $\D_{\rm alg}(\alpha^p)\ni f_i \to \theta$ that converges in the point $\sigma$-weak topology. Each $f_i$ has a natural extension on $M$, which we denote by $\widetilde{f}_i\in \D_{\rm alg}(\alpha)$. Then since $\D(\alpha)$ is compact, up to a subnet, we may assume that $\widetilde{f}_i$ converges to $\widetilde{\theta}\in \D(\alpha)$. By construction, it is easy to see $\widetilde{\theta}|_{pMp}=\theta$, thus the map is surjective.

	(3) Let $x\in \mathcal D(\alpha^p)\subset pMp$ and we show $x\in \mathcal D(\alpha)$. Fix $\widetilde{\theta}\in \D(\alpha )$ and $\theta:=\widetilde{\theta}|_{pMp}$ is contained in $\D(\alpha^p)$. By assumption of $x$ and item (1), we get
	$$\emptyset \neq \mathcal K(\theta(x),\alpha^p)^{\alpha^p} = p \mathcal K( \theta(x),\alpha)^{\alpha} p.$$
This implies $x\in \mathcal D(\alpha)$ and $x\in p\mathcal D(\alpha)p$.

Let $x\in \mathcal D(\alpha)$ and we show $\mathcal K( \theta(pxp),\alpha^p)^{\alpha^p} \neq \emptyset$ for all $\theta\in \D(\alpha^p)$. 
By item (2), there exists $\widetilde{\theta}\in\D(\alpha )$ such that $\widetilde{\theta}|_{pMp}=\theta$. By assumption of $x$, it holds that $\mathcal K(\widetilde{\theta}(x),\alpha)^{\alpha} \neq \emptyset$. Combining it with item (1), we get
	$$\emptyset\neq p \mathcal K( \widetilde{\theta}(x),\alpha)^{\alpha} p= \mathcal K(p \widetilde{\theta}(x)p,\alpha)^{\alpha} = \mathcal K(\theta(pxp),\alpha^p)^{\alpha^p} .$$
This implies $pxp\in \mathcal D(\alpha^p)$, hence $p\mathcal D(\alpha)p\subset \mathcal D(\alpha^p)$.
\end{proof}

\begin{lem}\label{lem-projection-unitary1}
	Let $A\subset M$ be an inclusion of von Neumann algebras and  $p\in A'\cap M$ a projection. Consider $\Ad_A \colon \mathcal U(A)\actson M$ and $\Ad_{Ap}\colon \mathcal U(Ap)\actson pMp$.
\begin{enumerate}
	\item We have $p\mathcal K(x,\Ad_A)p=\mathcal K(pxp,\Ad_{Ap}) = \mathcal K(pxp,\Ad_{Ap})$ for all $x\in M$.

	\item The restriction map $\D(\Ad_A)\ni \theta \mapsto \theta|_{pMp}\in \D(\Ad_{Ap})$ is surjective.

	\item We have $\mathcal D(\Ad_{Ap})=p\mathcal D(\Ad_A)p \subset \mathcal D(\Ad_{A})$.

\end{enumerate}
\end{lem}
\begin{proof}
	One can follow the same proofs as that in Lemma \ref{lem-projection-action1}, once the following claim is proven. The following map is surjective:
	$$ \mathcal U(A)\ni u \mapsto up\in \mathcal U(Ap) .$$
To see the claim, recall that $Az\ni az \mapsto ap\in Ap$ is a $\ast$-isomorphism, where $z\in \mathcal Z(A')=\mathcal Z(A)$ is the central support of $p$ in $A'$. In particular, the map $\mathcal U(Az)\ni uz \mapsto up\in \mathcal U(Ap)$ is a group isomorphism. Since $z$ is in the center of $A$, one has $\mathcal U(A)= \mathcal U(Az)\oplus \mathcal U(Az^{\perp})$, hence the map $\mathcal U(A)\ni u \mapsto uz\in \mathcal U(Az)$ is surjective. Combining these two maps, we get the claim.
\end{proof}

\begin{lem}\label{lem-projection-unitary2}
	Let $A\subset M$ be an inclusion of von Neumann algebras and $z\in \mathcal Z(A)$ a projection. Consider the group action $\Ad_A \colon \mathcal U(A)\actson M$. Then along the natural decomposition $M=zMz\oplus zMz^\perp \oplus z^\perp Mz \oplus z^\perp Mz^\perp$, we have
	$$ \mathcal D(\Ad_A) = \mathcal D(\Ad_{Az}) \oplus  zMz^\perp \oplus z^\perp Mz\oplus  \mathcal D(\Ad_{Az^{\perp}}),$$
where $\Ad_{Ap}\colon \mathcal U(Ap)\actson pMp$ is the reduced action for each projection $p\in A'\cap M$.
\end{lem}
\begin{proof}
Define a finite dimensional von Neumann subalgebera $A_0:=\C z\oplus \C z^\perp\subset A$. Since the unitary group $\mathcal U(A_0)$ is compact, one can define
	$$\Psi\colon M \to A_0'\cap M;\quad  x\mapsto \int_{\mathcal U(A_0)}uxu^* \, du,$$
where $du$ is the Haar measure. By construction, $\Psi$ is contained in $\D(\Ad_{A_0})\subset \D(\Ad_A)$. 
Then for each $x\in zMz^{\perp}$ and $\theta\in \D(\Ad_A)$, one has
	$$\mathcal K(\theta(x),A) \ni \Psi(\theta(x)) = \Psi(z\theta(x)z^\perp)= z\Psi(\theta(x))z^\perp=0.$$
This implies that $A'\cap \mathcal K(\theta(x),A)$ contains $0$, hence $x$ is contained in $\mathcal D(\Ad_A)$. We get $z M z^\perp \subset \mathcal D(\Ad_A)$ and this means $z M z^\perp = z\mathcal D(\Ad_A)z^\perp$. The same reasoning shows that $z^\perp M z = z^\perp\mathcal D(\Ad_A)z$. 
Since $z\mathcal D(\Ad_A)z = \mathcal D(\Ad_{Az})$ by Lemma \ref{lem-projection-unitary1} and the same holds for $z^\perp$, we get the conclusion.
\end{proof}

\subsection{Tracial subalgebras}

In this subsection, we study the case that the subalgebra $A\subset M$ is semifinite. In the case of group actions, this means that the action has an invariant weight.

\begin{lem}\label{lem-averaging-trace1}
	Let $\alpha\colon G\actson M$ be a discrete group action on a von Neumann algebra.  Assume that there exists a faithful normal semifinite weight $\varphi$ such that $\alpha_g(\varphi) = \varphi$ for all $g\in G$. Then it holds that $\mathfrak m_\varphi \subset \mathcal D(\alpha)$.

If further that $\varphi$ is a faithful normal state, then there exists a unique conditional expectation $E\colon M \to M^\alpha$ such that $\varphi \circ E= \varphi$ and 
$E\in \D(\alpha)$.
\end{lem}
\begin{proof}
	Fix $d\in \mathfrak m_\varphi^+$ and take $x\in \mathcal K(d,\alpha)$. We first show  $\varphi(x)\leq \varphi(d)$. To see this, take a net $f_i\in \D_{\rm alg}(\alpha)$ such that $f_i(d)\to x$ in the $\sigma$-weak topology. Since $\varphi$ is lower semi-continuous and since $\varphi$ is $\alpha$-invariant, we get 
	$$ \varphi( x) \leq \liminf_i \varphi(f_i(d)) = \varphi(d).$$
We next show $\mathcal K(d,\alpha)^\alpha \neq \emptyset$. 
As in the proof of \cite[Lemma 4.4]{HI15}, we can embed $\mathcal K(d,\alpha) \to L^2(M,\varphi)$ and the range is closed in the $L^2$-norm topology. 
By the projection theorem, take the unique element $a\in \mathcal K(d,\alpha)$ that attains the minimum distance from $0\in L^2(M,\varphi)$. By the uniqueness and since $\varphi$ is $\alpha$-invariant, we get $\alpha_g(a)=a$ for all $g\in G$. This means $a\in M^\alpha$ and we obtain $\mathcal K(d,\alpha)^\alpha \neq \emptyset$.

We show $d\in \mathcal D(\alpha)$. Take any $\theta\in \D(\alpha)$. Since $\theta(d)$ is in $\mathfrak m_\varphi^+$ by the first part of the proof, we can apply the result in the last argument and get $\mathcal K(\theta(d),\alpha)^\alpha \neq \emptyset$. This means $d\in \mathcal D(\alpha)$.

Suppose that $\varphi$ is a state. Then since $\alpha$ commutes with the modular action $\sigma^\varphi$, the subalgebra $M^\alpha \subset M$ is globally preserved by $\sigma^\varphi$. Hence there exists a unique conditional expectation $E\colon M \to  M^\alpha$ such that $\varphi \circ E = \varphi$. By the uniqueness, we know $\alpha_g \circ E\circ \alpha _g^{-1} = E$ for all $g\in G$. 
Since $M=\mathfrak m_\varphi  \subset \mathcal D(\alpha)$, for each $x\in M$, we have $\mathcal K(x,\alpha)^\alpha \neq \emptyset$, so take any element $a$. Observe that $ E (\mathcal K(x,\alpha)) = \{E(x)\}$ by the normality of $E$, hence we get $a = E(a)=E(x)$. We have proved that $\mathcal K(x,\alpha) = \{E(x)\}$ for all $x\in M$. Now by Lemma \ref{lem-Schwartz}, take $\Psi\in \D(\alpha)$ with $\Psi(M)=\Psi(\mathcal D(\alpha))\subset M^\alpha$. Then for each $x\in M$, we have $\Psi(x)\in  \mathcal K(x,\alpha)^\alpha = \{E(x)\}$, hence we get $\Psi(x)=E(x)$ and $E = \Psi \in \D(\alpha)$. 
\end{proof}

\begin{prop}\label{prop-averaging-trace2}
	Let $A\subset M$ be an inclusion of von Neumann algebras. Assume that there exists a faithful normal semifinite weight $\varphi$ such that $A\subset M_\varphi$. Then we have $\mathfrak m_\varphi \subset \mathcal D(\Ad_A\colon \mathcal U(A)\actson M)$.

This assumption is satisfied in the following two cases:
\begin{itemize}
	\item $\varphi$ is a semifinite trace and $A$ is a general subalgebra (possibly of type $\rm III$);
	\item $A\subset M$ is with an operator valued weight $E_A$, $A$ has a semifinite trace $\Tr_A$, and $\varphi=\Tr_A \circ E_A $.
\end{itemize}
\end{prop}
\begin{proof}
	We consider $\Ad_A\colon \mathcal U(A)\actson M$. Then the condition $A\subset M_\varphi$ means that $\Ad_A$ is $\varphi$-preserving. Hence the proposition follows by Lemma \ref{lem-averaging-trace1}.
\end{proof}

\subsection{Type $\rm III$ subalgebras}

In this subsection, we study the case that the subalgebra $A\subset M$ is of type III. We prepare a lemma.

\begin{lem}\label{lem-operator-bounded}
	Let $\alpha \colon G\actson M$ be a discrete group action on a von Neumann algebra, $A\subset M$ a von Neumann subalgebra with operator valued weight $E_A\colon M\to A$ such that $\alpha_g\circ E_A = E_A \circ \alpha_g$ for all $g\in G$. If $x\in \mathfrak m_{E_A}^+$ and $\theta\in \D(\alpha)$, then $ E_A(\theta(x))\leq \theta(E_A(x))$. 

In particular, we have
\begin{itemize}
	\item $ \mathcal K(x,\alpha)\subset \mathfrak m_{E_A}$ for all $x\in \mathfrak m_{E_A}$; and 
	\item $ \theta(\mathfrak m_{E_A}) \subset \mathfrak m_{E_A}$ for all $\theta\in \D(\alpha)$.
\end{itemize}
\end{lem}
\begin{proof}
	Let $x\in \mathfrak m_{E_A}^+$ and $\theta\in \D(\alpha)$ be as in the statement. Observe that $ E_A(f(x)) =  f (E_A(x))$ for all $f\in \D_{\rm alg}(\alpha)$. Let $\D_{\rm alg}(\alpha )\ni f_i \to \theta \in \D(\alpha)$ be a convergence in the point $\sigma$-weak topology.  Let $\phi\in A_\ast^+$. Then since $\phi\circ E_A$ is lower semi-continuous, we have
\begin{align*}
	(\phi\circ E_A)(\theta(x)) 
	&= (\phi\circ E_A)(\lim_i f_i(x))\\
	&\leq \liminf_{i \in I} (\phi\circ E_A)(f_i(x))
	=\liminf_{i \in I} \phi( f_i(E_A(x)))\\
	&=\lim_{i \in I} \phi( f_i(E_A(x)))= \phi(\theta(E_A(x))).
\end{align*}
Since this holds for all $\phi\in A_\ast^+$, we get $E_A(\theta(x))\leq \theta(E_A(x))$.

	To see the second part, we let $x\in \mathfrak m_{E_A}$ and decompose $x = \sum_{i}c_i x_i$ for some finitely many $c_i\in \C$ and $x_i \in \mathfrak m_{E_A}^+$. Take any $a\in \mathcal K(x,\alpha)$ and find $\theta\in \D(\alpha)$ such that $\theta(x)=a$. As $E_A(\theta(x_i))\leq \theta(E_A(x_i))<\infty$,  $\theta(x)=a$ is contained in $ \mathfrak m_{E_A}$. This argument shows the last two statements.
\end{proof}

The following proposition is the key observation of this article. When $E_A$ is a conditional expectation, the result was established in \cite{Ma19}. The proof in our setting proceeds similarly to \cite{Ma19}, but a different argument is required in the latter part of the proof to deal with operator valued weights.

\begin{prop}\label{prop-averaging-typeIII}
	Let $A\subset M$ be an inclusion of von Neumann algebras with operator valued weight $E_A$. Assume that $A$ is of type $\rm III$. Then $\mathfrak m_{E_A}\subset \mathcal D(\Ad_A\colon \mathcal U(A)\actson M)$.
\end{prop}
\begin{proof}
	By Theorem \ref{thm-Marrakchi} (for the case $A=M$), take a faithful normal conditional expectation $E\colon A \to \mathcal Z(A)$, which is contained in $\mathcal D(\Ad_A\colon \mathcal U(A)\actson A)$. 
Take $\widetilde{E}\in \mathcal D(\Ad_A\colon \mathcal U(A)\actson A)$ such that $\widetilde{E}|_{A} = E$. 
Take a minimal idempotent $\Phi\in \D(\Ad_A\colon \mathcal U(A)\actson M)$. Then since $\Phi\circ \widetilde{E}$ is a minimal idempotent by Proposition \ref{prop-Marrakchi}, up to replacing, we may assume that $\Phi|_{A}=E$. In particular $\Phi$ is a minimal idempotent that is faithful and normal on $A$. 

Next, let $\Psi\in \D(\Ad_A\colon \mathcal U(A)\actson M)$ be any idempotent such that $\Psi|_A = E$. We first show that $\Psi(x^*x)=x^*x$ for all $x\in \Psi(\mathfrak m_{E_A})$. Since $\Psi^2 = \Psi$ and since $\Psi$ is ucp, by the Cauchy--Schwartz inequality, we get
	$$x^*x = \Psi(x^*)\Psi(x) \leq  \Psi(x^*x).$$
Observe that $x\in \Psi(\mathfrak m_{E_A}) \subset \mathfrak m_{E_A}$ by Lemma \ref{lem-operator-bounded}, hence $E_A(x^*x)<\infty $. We apply $E_A$ to this inequality and get that 
	$$ E_A(x^*x)\leq E_A \circ \Psi(x^*x)\leq \Psi\circ E_A(x^*x) <\infty,$$
where we  used Lemma \ref{lem-operator-bounded} again. Since $\Psi|_A=E$, the right hand side coincides with $ E\circ E_A(x^*x) $. We then apply $E$ to the inequalities and get that 
	$$ E\circ E_A(x^*x)\leq E\circ E_A \circ \Psi(x^*x)\leq E\circ E_A(x^*x).$$
This means $E\circ E_A(x^*x)= E\circ E_A \circ \Psi(x^*x)$, and hence $E\circ E_A( \Psi(x^*x) - x^*x)=0$ as $ \Psi(x^*x)$ and $ x^*x$ are contained in $\mathfrak m_{E_A}$. 
Since $\Psi(x^*x) - x^*x\geq 0$ and since $E\circ E_A$ is faithful, we get $\Psi(x^*x) - x^*x=0$ as desired.

We keep $\Psi$ as above. We define a unital $C^*$-subalgebra
	$$ M_0:= \overline{\mathfrak m_{E_A}+A} \subset M,$$
where the closure is taken by the operator norm topology. 
We claim that $\Psi(M_0)\subset M$ is a unital $C^*$-subalgebra, hence the Choi--Effros product of $\Psi$ on $\Psi(M_0)$ coincides with the original product of $M$. 
To see this, if $x,y\in \Psi(\mathfrak m_{E_A})$, then the polarization and the result in the last paragraph show that $ \Psi(xy)=xy $. If $x,y$ are in $\Psi(\mathfrak m_{E_A}+A)$, then it is easy to see that $ \Psi(xy)=xy $, because we can use $\Psi(A) = \mathcal Z(A)$ and $\Psi$ is a $\mathcal Z(A)$-module map. Since $\Psi$ is norm continuous and since $\Psi(M_0)$ is norm closed by $\Psi^2=\Psi$, the claim is proven.

Now we go back to the fixed minimal idempotent $\Phi\in \D(\Ad_A\colon \mathcal U(A)\actson M)$ with $\Phi|_A=E$. By the last paragraph, the Choi--Effros product of $\Phi$ on $\Phi(M_0)$ coincides with the product of $M$. We will show that $\Phi(M_0)\subset A'\cap M$.

The next claim is important.

\begin{claim}
	For each $\theta\in \D(\Ad_A\colon \mathcal U(A)\actson M)$, $\theta\circ \Phi(M_0)\subset M$ is a unital $C^*$-subalgebra on which the Choi--Effros product coincides with the product of $M$. 
Further, the following map is a $\ast$-isomorphism (with respect to the products of $M$)
	$$ \theta\colon \Phi(M_0)\to \theta\circ \Phi(M_0);\quad \Phi(x)\mapsto \theta(\Phi(x)) .$$
\end{claim}
\begin{proof}[{\bf Proof of Claim}]
	By Proposition \ref{prop-Marrakchi}, $\Phi^\theta:=\theta\circ \Phi$ is a minimal idempotent. Since $\Phi^\theta|_A = \theta \circ E = E$, the Choi--Effros product of $\Phi^\theta$ on $\Phi^\theta(M_0)$ coincides with the product of $M$. Since $\Phi$ and $\Phi^\theta$ are minimal, by Proposition \ref{prop-Marrakchi}, the following map is a $\ast$-isomorphism with respect to the Choi--Effros products
	$$ \Phi^\theta\colon \Phi(M)\to \Phi^\theta(M);\quad \Phi(x)\mapsto \Phi^\theta(\Phi(x))=\theta\circ \Phi(x) .$$
Note that the restriction of $\Phi^\theta$ to $\Phi(M_0)$ coincides with the map $\theta$ in the statement of the claim, hence $\theta$ is a bijective map. Since  the Choi--Effros products on $\Phi(M_0)$ and $\Phi^\theta(M_0)$ coincides with the products of $M$, the claim is proven.
\end{proof}

Let $a\in \Phi(M_0)$ and we show that $a\in A'\cap M$. For this, suppose by contradiction that $a\not\in A'\cap M$. Then there exists $u\in \mathcal U(A)$ such that $uau^*\neq a$. We put $\theta := \frac{1}{2}(\Ad(u)+\id)$ and apply the claim. We get the following $\ast$-isomorphism
	$$ \theta\colon \Phi(M_0)\to \theta\circ \Phi(M_0);\quad \Phi(x)\mapsto \theta(\Phi(x))= \frac{1}{2}(u\Phi(x)u^* + \Phi(x)) .$$
Since it is a $\ast$-homomorphism with respect to the products of $M$, if $a=\Phi(x)$, then \begin{align*}
	\frac{1}{2}(ua^*au^* + a^*a)=\theta(a^*a)  = 
	\theta(a)^*\theta(a) = \frac{1}{4}(uau^* + a)^*(uau^* + a).
\end{align*}
We then compute that, by the parallelogram law, 
\begin{align*}
	 \theta(a^*a) 
	&= \frac{1}{2}(|uau^*|^2 + |a|^2|) = \frac{1}{4}(|uau^* + a|^2 + |uau^*-a|^2) \\
	&\geq \frac{1}{4}|uau^* + a|^2 = \theta(a)^*\theta(a) = \theta(a^*a) .
\end{align*}
This means $|uau^*-a|^2=0$, a contradiction. We get $\Phi(M_0)\subset A'\cap M$.

Let $a\in \mathfrak m_{E_A}$. Then for each $\theta\in \D(\Ad_A\colon \mathcal U(A)\actson M)$, one has $\theta(a)\in \mathfrak m_{E_A}\subset M_0$ by Lemma \ref{lem-operator-bounded}. We get $\Phi(\theta(a))\in A'\cap M$, hence it is contained in $\mathcal K(\theta(a),\Ad_A)^{\Ad_A}$. We conclude $a\in \mathcal D(\Ad_A)$. This proves $\mathfrak m_{E_A} \subset \mathcal D(\Ad_A)$, as desired.
\end{proof}

\subsection{Proof of Theorem \ref{thmA}}\label{Proof of Theorem thmA}

Now we prove Theorem \ref{thm-Dixmier-expectation}. Then Theorem A follows immediately.

\begin{proof}[{\bf Proof of Theorem \ref{thm-Dixmier-expectation}}]
	Let $z\in \mathcal Z(A)$ be the unique projection such that $Az$ is semifinite and $Az^\perp$ is of type III. We denote by $E_{Az}\colon zMz\to Az$ the restriction of $E_A$ on $zMz$, and by $E_{Az^\perp}$ similarly. Observe that 
	$$ \mathfrak m_{E_A} = \mathfrak m_{E_{Az}}\oplus  z \mathfrak m_{E_A}z^\perp \oplus z^\perp  \mathfrak m_{E_A}z\oplus  \mathfrak m_{E_{Az^\perp}}.$$
By Lemma \ref{lem-projection-unitary2}, we have the following decomposition 
	$$ \mathcal D(\Ad_A) = \mathcal D(\Ad_{Az}) \oplus  zMz^\perp \oplus z^\perp Mz\oplus  \mathcal D(\Ad_{Az^{\perp}}).$$
By Propositions \ref{prop-averaging-trace2} and \ref{prop-averaging-typeIII}, we have $\mathfrak{m}_{E_{Az}}\subset  \mathcal D(\Ad_{Az}) $ and $\mathfrak m_{E_{Az^{\perp}}}\subset \mathcal D(\Ad_{Az^{\perp}})$, hence we get that $\mathfrak m_{E_A}\subset \mathcal D(\Ad_A)$. 

The last assertion follows from Lemma \ref{lem-Schwartz}.
\end{proof}

\begin{proof}[{\bf Proof of Corollary \ref{corB}}]
	$(1)\Rightarrow(3)$ Let $x\in \mathfrak m_{E_A}^+$ and observe that $A'\cap \mathcal K(x,\Ad_A)$ is not empty by Theorem \ref{thmA}. Then by Lemma \ref{lem-operator-bounded}, every element in $A'\cap \mathcal K(x,\Ad_A)$ is contained in both $\mathfrak m_{E_A}$ and $(A'\cap M)^+$. By assumption, this implies that $A'\cap \mathcal K(x,\Ad_A)=\{0\}$. 

	$(3)\Rightarrow(2)$ This is obvious.

	$(2)\Rightarrow(1)$ Suppose that item (1) does not hold and there exists a nonzero positive $x\in A'\cap M$ such that $E_A(x)<\infty$. Then since $\mathcal K(x,\Ad_A)=\{x\}$, item (2) does not hold. 
\end{proof}
\begin{rem}
	In item (2) and (3) in Corollary \ref{corB}, the conditions actually hold for all $x\in \mathfrak m_{E_A}$. 
To see this, take $x\in \mathfrak m_{E_A}$. Take any $d\in A'\cap \mathcal K(x,\Ad_A)$ and find $\theta\in \D(\Ad_A)$ such that $\theta(x)=d$. By Lemma \ref{lem-Schwartz}, take $\Psi\in \D(\Ad_A)$ such that $\Psi(\mathcal D(\Ad_A))\subset A'\cap M$. By replacing $\Psi\circ \theta$ by $\Psi$, we may assume that $\Psi(x)=d$. 
Then since $x$ is a linear span of elements in $\mathfrak m_{E_A}^+$, we get $\Psi(x)=0$. This means $d=0$.
\end{rem}

\subsection{Amenable groups and subalgebras}

We say that a topological group $G$ has the \textit{fixed point property} if every continuous affine action of $G$ on a compact convex space admits a fixed point. We are only interested in locally compact amenable groups or unitary groups of amenable von Neumann algebras. 
Recall that a von Neumann algebra $M$ is \textit{amenable} \cite{Co75} if it is injective, or equivalently semidiscrete.

\begin{lem}
	If $M$ is amenable, then the unitary group $\mathcal U(M)$ with the strong topology has the fixed point property. 
\end{lem}
\begin{proof}
	Assume first that $M$ has separable predual. Then by \cite{Co75}, there exists an increasing sequence $M_n \subset M$ of finite dimensional subalgebras such that $\bigcup_n M_n \subset M$ is a $\sigma$-weakly dense subalgebra. Then each $\mathcal U(M_n)$ is compact, so it has the fixed point property. It is easy to see that $\mathcal U(M)$ also has the fixed point property. 

	Assume next that $M$ is $\sigma$-finite and fix a faithful normal state $\varphi\in N_\ast$. Then by Takesaki's conditional expectation theorem, it is easy to construct an increasing net $N_i \subset N$ of von Neumann subalgebras with separable predual and with expectation such that $\bigcup_i M_i \subset M$ is a $\sigma$-weakly dense subalgebra. Then since each $M_i$ has the fixed point property, so does $M$. 

	Finally let $M$ be a general amenable von Neumann algebra. Take an increasing net of  $\sigma$-finite projections $p_i$ such that $p_i \to 1$ strongly. Put
	$$ \mathcal U_i:= \mathcal U(p_iMp_i\oplus \C p_i^\perp )\subset \mathcal U(M) .$$
It is not hard to see that it is an increasing net of subgroups in $\mathcal U(M)$ and that $\bigcup_i \mathcal U_i$ is strongly dense in $\mathcal U(M)$. Since each $\mathcal U_i$ has the fixed point, so does $\mathcal U(M)$.
\end{proof}

\begin{lem}\label{lem-averaging-amenable}
	Let $\gamma \colon \widetilde{G}\actson M$ be a discrete group action on a von Neumann algebra. Assume that $\widetilde{G}$ is generated by subgroups $H,G\leq \widetilde{G}$ satisfying  $gHg^{-1} = H$ for all $g\in G$. We denote by $\gamma^H,\gamma^{G}$ restrictions of $\gamma$ on $H,G$ respectively. 

\begin{enumerate}
	\item Assume that $G$ is a topological group with the fixed point property and that $\gamma^G$ is continuous. Then there exists $\Psi\in \D(\gamma)$ such that 
	$$\Psi(x)\in M^\gamma ,\quad \text{for all}\quad x\in \mathcal D(\gamma^H).$$
If a subset $M_0\subset \mathcal D(\gamma^H)$ satisfies that $\theta(M_0) \subset \mathcal D(\gamma^H)$ for all $\theta\in \D(\gamma)$, then we have $M_0\subset \mathcal D(\gamma)$. 

	\item Assume that $M$ has a $\gamma^{G}$-preserving f.n.s.\ weight $\varphi$, and that there exists a subset $A_0\subset \mathfrak m_{\varphi}$ such that
	$$ \theta(A_0) \subset  \mathfrak m_{\varphi},\quad \text{for all}\quad \theta\in \D(\gamma^H).$$
Then there exits $\Psi\in \D(\gamma)$ such that 
	$$\Psi(x)\in M^\gamma ,\quad \text{for all}\quad x\in A_0\cap \mathcal D(\gamma^H).$$
If a subset $M_0\subset \mathfrak m_\varphi\cap \mathcal D(\gamma^H)$ satisfies that $\theta(M_0) \subset \mathfrak m_\varphi\cap \mathcal D(\gamma^H)$ for all $\theta\in \D(\gamma)$, then we have $M_0\subset \mathcal D(\gamma)$.

\end{enumerate}

\end{lem}
\begin{proof}
	By the assumption of $H$ and $G$, it is easy to see that $\gamma_g(M^{\gamma^H}) = M^{\gamma^H}$ for all $g\in \widetilde{G}$. This implies that $\theta(M^{\gamma^H})\subset M^{\gamma^H}$ for all $\theta\in \D(\gamma)$. 
By Lemma \ref{lem-Schwartz}, take $\Psi^H\in \D(\gamma^H)$ and $\Psi^{G}\in \D(\gamma^{G})$ such that
	$$\Psi^H(\mathcal D(\gamma^H))\subset M^{\gamma^H},\quad \Psi^{G}(\mathcal D(\gamma^{G}))\subset M^{\gamma^{G}}.$$
Put $\Psi:=\Psi^{G}\circ \Psi^{H}\in D(\gamma)$, and observe that
	$$ \Psi(\mathcal D(\gamma^{H})) = \Psi^{G}\circ \Psi^{H}(\mathcal D(\gamma^{H}))
\subset \Psi^{G}(M^{\gamma^H})\subset M^{\gamma^H} .$$

	(1) For each $x\in M$, consider a compact convex subset
	$$\mathcal K(x,\gamma^{G})=\overline{\mathrm{conv}}^{\rm weak}\{ \gamma_g(x)\mid g\in G \}\subset M.$$
Since $G$ acts on $\mathcal K(x,\gamma^{G})$, it has a fixed point by assumption. Since this holds for all $x\in M$, we have $\mathcal D(\gamma^{G})=M$, hence Lemma \ref{lem-Schwartz} says that there exists $\Psi^{G}\in\D(\gamma^{G})$ such that $\Psi^{G}(M)\subset M^{\gamma^{G}}$. This is a conditional expectation from $M$ onto $M^{\gamma^{G}}$. If we use this $\Psi^G$, then combining it with the result in the first paragraph, we get
	$$ \Psi(\mathcal D(\gamma^{H})) \subset M^{\gamma^H} \cap M^{\gamma^G} = M^{\gamma}.$$
We get the conclusion. For the last statement, if $x\in M_0$, then for any $\theta\in \D(\gamma)$, we have $\theta(x)\in \mathcal D(\gamma^H)$ and hence $\Psi(\theta(x))\in M^\gamma$. This shows that $\mathcal K(\theta(x),\gamma)^\gamma\neq \emptyset$, hence $x\in \mathcal D(\gamma)$. 

	(2) Lemma \ref{lem-averaging-trace1} says that $\mathfrak m_{\varphi} \subset \mathcal D(\gamma^{G})$. Since $\Psi^H\in \D(\gamma^H)$, we have $\Psi^H(A_0)\subset \mathfrak m_{\varphi}$ by assumption, so
	$$ \Psi(A_0) = \Psi^{G}\circ \Psi^H(A_0)\subset \Psi^{G}(\mathfrak m_{\varphi})\subset M^{\gamma^{G}} .$$
Combining it with the result in the first paragraph, we get $\Psi(A_0\cap \mathcal D(\gamma^H))\subset M^{\gamma^{G}}\cap M^{\gamma^{H}}=M^\gamma$. 
For the last statement, let $\theta\in \D(\gamma)$. Then since $\Psi^H\circ \theta \in \D(\gamma)$, the assumption shows that
	$$\theta(M_0) \subset \mathcal D(\gamma^H),\quad \Psi^H\circ \theta(M_0) \subset \mathfrak m_{\varphi} .$$
These imply that
	$$ \Psi (\theta(M_0) )\subset \Psi(\mathcal D(\gamma^H))\subset M^{\gamma^{H}},\quad \Psi(\theta(M_0)) \subset \Psi^{G}(\mathfrak m_{\varphi})\subset M^{\gamma^{G}} .$$
We get $\Psi(\theta(M_0)) \subset M^{\gamma^{G}}\cap M^{\gamma^{H}}=M^\gamma$ and hence $\mathcal K(\theta(x),\gamma)^\gamma\neq \emptyset$ and $x\in \mathcal D(\gamma)$. 
\end{proof}

\begin{prop}\label{prop-averaging-AFD}
	Let $A\subset M$ be an inclusion of von Neumann algebras and assume that $A$ has a tensor decomposition $A=A_0\ovt N$ with $N$ amenable. Then there exists $\Psi\in \D(\Ad_A)$ such that
	$$\Psi(x)\in A'\cap M \quad \text{for all}\quad x\in \mathcal D(\Ad_{A_0}).$$
We further have
	$$\mathcal D(\Ad\colon \mathcal U(A_0)\actson M)\subset \mathcal D(\Ad\colon \mathcal U(A_0)\times \mathcal U(N)\actson M).$$
\end{prop}
\begin{proof}
	Put $\widetilde{G} = H\times G =\mathcal U(A_0)\times \mathcal U(N)$ and $\gamma = \Ad\colon \widetilde{G}\actson M$. Since $\mathcal U(N)$ has the fixed point property, we can apply Lemma \ref{lem-averaging-amenable}(1). If we put $M_0:=\mathcal D(\alpha^H)$, then the assumption of the last part in Lemma \ref{lem-averaging-amenable}(1) is satisfied. We get the conclusion.
\end{proof}

\begin{prop}\label{prop-averaging-amenable}
	Let $A\subset M$ be an inclusion of von Neumann algebras with operator valued weight $E_A$. Let $\alpha \colon G\actson M$ be a discrete group action such that $\alpha_g(A)=A$ for all $g\in G$. Define a new action by 
	$$ \gamma \colon \widetilde{G}:=\mathcal U(A)\rtimes G \actson M;\quad \gamma_{(u,g)}:= \Ad(u)\circ \alpha_g = \alpha_g \circ \Ad( \alpha_g^{-1}(u)).$$
\begin{enumerate}
	\item If $G$ is a locally compact amenable group, then there exists $\Psi\in \D(\gamma)$ such that
	$$\Psi(x)\in A'\cap M^\alpha,\quad \text{for all}\quad x\in \mathfrak m_{E_A}.$$
If $\alpha_g\circ E_A = E_A\circ \alpha_g$  for all $g\in G$, then we have $\mathfrak{m}_{E_A}\subset \mathcal D(\gamma)$. 

	\item If there exists a faithful normal state $\psi_A\in A_\ast$ that is preserved by $\alpha$, and if $\alpha_g\circ E_A = E_A\circ \alpha_g$ for all $g\in G$, then we have $\mathfrak{m}_{E_A}\subset \mathcal D(\gamma)$.

If we further assume that $E_A(1)=1$ and $A$ does not have any direct summand that is semifinite and properly infinite, then there exists a faithful normal conditional expectation from $M$ onto $M^\gamma=A'\cap M^\alpha$ that is contained in $\D(\gamma)$.

\end{enumerate}
\end{prop}
\begin{proof}
	We first prove a claim.

\begin{claim}
	If $\alpha_g\circ E_A = E_A\circ \alpha_g$ for all $g\in G$, then we have  $\theta(\mathfrak{m}_{E_A})\subset \mathfrak{m}_{E_A}$ for all $\theta\in \D(\gamma)$.
\end{claim}
\begin{proof}
	Since $E_A$ is an operator valued weight to $A$, by assumption, we have that $\gamma_a\circ E_A = E_A \circ \gamma_a$ for all $a\in \mathcal U(A)\rtimes G$. Then we can apply Lemma \ref{lem-operator-bounded}, and get the conclusion.
\end{proof}

	(1) Since $G$ is amenable, we apply Lemma \ref{lem-averaging-amenable}(1) and get $\Psi\in \D(\gamma)$ such that
	$$\Psi(x)\in M^{\gamma}=A'\cap M^\alpha,\quad \text{for all}\quad x\in \mathcal D(\gamma|_{\mathcal U(A)}).$$
Since $\gamma|_{\mathcal U(A)} = \Ad_A$, we have $\mathfrak m_{E_A}\subset D(\gamma|_{\mathcal U(A)})$ by Theorem \ref{thm-Dixmier-expectation}. 

For the last statement, we use Lemma \ref{lem-averaging-amenable}(1) in the case that $M_0=\mathfrak m_{E_A}$. For this, we have to show that $\theta(\mathfrak m_{E_A})\subset \mathfrak m_{E_A}$ for all $\theta\in \D(\gamma)$. This is exactly the claim above.

	(2) Put $\varphi:= \psi_A \circ E_A$ and then it is a f.n.s.\ weight on $M$ preserved by $\alpha$, satisfying $\mathfrak m_{E_A}\subset \mathfrak m_{\varphi}$. We put $M_0=\mathfrak m_{E_A}$ and apply Lemma \ref{lem-averaging-amenable}(2). For this, observe that $\mathfrak m_{E_A}\subset \mathfrak m_{\varphi}\cap \mathcal D(\gamma|_{\mathcal U(A)})$ by Theorem \ref{thm-Dixmier-expectation}. For any $\theta\in \D(\gamma)$, the claim shows that $\theta(\mathfrak m_{E_A})\subset \mathfrak m_{E_A}$, hence 
	$$ \theta(M_0)\subset \mathfrak m_{E_A} \subset m_\varphi\cap \mathcal D(\gamma|_{\mathcal U(A)}) .$$
Then we can apply Lemma \ref{lem-averaging-amenable}(2). 

We verify the last statement. By Theorem \ref{thm-Marrakchi}, there exists a faithful normal conditional expectation $E_{A'\cap M}\colon M\to A'\cap M $, which is contained in $\D(\Ad_A)$. Next, we consider the action $\alpha$ on $A'\cap M$. Since it has a faithful normal $\alpha$-invariant state $\psi_A\circ E_A$, we can apply Lemma \ref{lem-averaging-trace1} and find a faithful normal conditional expectation $E\colon A'\cap M\to A'\cap M^\alpha$, which is contained in $\D(\alpha\colon G\actson A'\cap M)$. We extend $E$ to a map on $\D(\alpha\colon G\actson M)$ by the compactness. Then $E\circ E_{A'\cap M}$ is a faithful normal conditional expectation onto $A'\cap M^\alpha$, which is contained in $\D(\gamma)$. This is the conclusion.
\end{proof}

The next lemma will be used in application.

\begin{lem}\label{lem-operator-valued-weight-extension}
	Let $A\subset M$, $E_A$, $\alpha$ and $\gamma$ be as in Proposition \ref{prop-averaging-amenable} and assume that $\alpha_g \circ E_A = E_A \circ \alpha_g$ for all $g\in G$.
Let $z\in M$ be a central projection such that $\alpha_g(z)=z$ for all $g\in G$.  
Suppose that for any $d\in \mathfrak m_{E_A}^+$, 
	$$\mathcal K(dz,\gamma)= \overline{ \mathrm{co} }^{\rm weak}\{u \alpha_g(d)z u^*\mid u\in \mathcal U(A),\ g\in G\} $$
contains $0$. Then there exists $\Psi\in \D(\gamma)$ such that $\Psi(\mathfrak m_{E_A}z)=0$.
\end{lem}
\begin{proof}
	By the claim in the proof of Proposition \ref{prop-averaging-amenable}, we have $\theta(\mathfrak m_{E_A}^+)\subset \mathfrak m_{E_A}^+$, hence $\theta(\mathfrak m_{E_A}^+z)\subset \mathfrak m_{E_A}^+ z$ for all $\theta\in \D(\gamma)$. 
Then we can use Lemma \ref{lem-Schwartz2} for the case $M_0 = \mathfrak m_{E_A}^+z$, and get $\Psi\in \D(\gamma)$ such that $\Psi(\mathfrak m_{E_A}^+z)=0$. Since $\mathfrak m_{E_A}$ is a linear span of $\mathfrak m_{E_A}^+$, we get the conclusion.
\end{proof}

\section{Popa's intertwining-by-bimodules}

\subsection{Proof of Theorem \ref{thmC}}

We first recall the definition of Popa's intertwining condition. It first appeared in \cite{Po01,Po03}, and the definition here was given in \cite{HI15,Is19}.

\begin{df}\label{def-intertwining1}\upshape
	Let $M$ be a $\sigma$-finite von Neumann algebra and $A,B\subset M$ (possibly non-unital) von Neumann subalgebras with expectation $E_A,E_B$ respectively. We say that $A$ \textit{embeds with expectation into $B$ inside $M$} and write $A \preceq_{M} B$ 
if there exist $(H,f,\pi,w)$:
\begin{itemize}
	\item $H$ is a separable Hilbert space with a fixed minimal projection $e_{1,1}\in \B(H)$;
	\item $f\in B\ovt \B(H)$ is a projection and $\pi\colon A \to f(B\ovt \B(H))f$ is a unital normal $\ast$-homomorphism such that  $\pi(A) \subset f(B \ovt \B(H))f$ is with expectation;

	\item  $w\in (1_A\otimes e_{1,1})(M\ovt \B(H))f$ is a partial isometry (called an \emph{intertwiner}) such that 
	$$w\pi(a)= (a\otimes e_{1,1}) w = (a\otimes 1) w \quad \text{for all}\quad  a\in A.$$

\end{itemize}
\end{df}

We prove the following theorem, from which Theorem \ref{thmC} follows immediately.  As explained in Introduction, it generalizes Popa's original works \cite{Po01,Po03}. Items (4) and $(4')$ are new and serve as substitutes for item (2), that is the analytic criterion involving a net of unitaries. Items (2) and (3) have appeared in previous works, but they are new in the case where $A$ is of type III. The proofs rely on our Theorem \ref{thmA}.

In the following theorem, we use the same notation as in \cite[Theorem 4.3]{HI15}. So let $L^2(M)$ be the standard representation of $M$, and we denote the Jones projection of $\widetilde{B}:=B\oplus \C (1_M-1_B)$ by $e_{\widetilde{B}}$.
Let $\widehat{E}_{\widetilde{B}}$ be a canonical operator valued weight from $\langle M ,\widetilde{B}\rangle$ onto $M$ such that $\widehat{E}_{\widetilde{B}}(xe_{\widetilde{B}}x^*)=xx^*$ for all $x\in M$. See \cite[Section 2 and 4]{HI15} for more about these objects.

\begin{thm}\label{thm-intertwining1}
	Keep the setting as in Definition \ref{def-intertwining1}. Assume that either
\begin{itemize}
	\item[$\rm (a)$] $A$ does not have any direct summand that is semifinite and properly infinite; or 
	\item[$\rm (b)$] $B$ is properly infinite.

\end{itemize}
Then the  following conditions are equivalent. 
\begin{enumerate}
	\item We have $A \preceq_M B$.
	\item There exists no nets $(u_i)_{i}$ of unitaries in $\mathcal U(A)$ such that 
	$$E_B(b^* u_i a ) \rightarrow 0, \quad \text{$\sigma$-strongly for all $a,b\in M1_B$.}$$

	\item There exists $d_0\in \langle M,\widetilde{B}\rangle^+$ such that $\widehat{E}_{\widetilde{B}}(d_0)<\infty$, $d_0 = d_0 J1_BJ$ and 
	$$\mathcal K(d_0,\Ad_A)= \overline{{\rm co}}^{\rm weak} \left \{ u d_0 u^* \mid u \in \mathcal U(A)\right \} \subset (1_A \langle M, \widetilde B\rangle 1_A)^+$$
does not contain $0$. 

	\item There exists no $\Psi \in \D(A\subset 1_A\langle M,\widetilde{B}\rangle 1_A)$ such that $\Psi(ae_{\widetilde{B}} b^*) = 0$  for all $a,b\in M1_B$.

	\item[$\rm (4')$] There exists no $\Psi \in \D(A\subset 1_A\langle M,\widetilde{B}\rangle 1_A)$ such that $\Psi(\mathfrak m_{\widehat{E}_{\widetilde{B}}}J1_BJ) = 0$ and $\Psi|_{1_AM1_A}$ is a conditional expectation onto $A'\cap 1_AM1_A$.

\end{enumerate}
In assumption $\rm (a)$, if $A\not\preceq_M B$, we can choose $\Psi$ in item $(4')$ so that $\Psi|_{1_AM1_A}$ is a faithful normal conditional expectation.
\end{thm}
\begin{rem}\label{rem-intertwining2}\upshape
We will use Theorem~\ref{thmA} to prove the implication $(3)\Rightarrow(1)$. The proof proceeds exactly as in Corollary~\ref{corB}, indicating that the phenomenon observed there is intimately connected to Popa's intertwining techniques.
\end{rem}
\begin{proof}
	Fix faithful states $\varphi\in M_\ast$ such that $1_B \in M_\varphi$ and $\varphi\circ E_B = \varphi $ on $1_BM1_B$. For the reader's convenience, we assume that $A,B\subset M$ are unital subalgebras. This assumption is not essential and the proof works without any change, see the proofs of \cite[Theorem 4.3]{HI15} and \cite[Theorem 3.2]{Is19}.

We note that by \cite[Lemma 2.6]{Is19}, in assumption (a) or (b), the Hilbert space $H$ appeared in $A\preceq_MB$ can be taken finite-dimensional.

	$(1)\Rightarrow(2)$ We can choose $(H,f,\pi,w)$ witnessing $A\preceq_M B$ such that $H$ is finite-dimensional. Suppose that there exists a net $(u_i)_i\subset \mathcal U(A)$ as in (2). Fix a faithful normal state $\omega\in \B(H)_\ast$ and consider the $L^2$-norm by $\varphi\otimes \omega$. Then with $E:= E_{B\ovt \B(H)}$, we have
\begin{align*}
	\|E(w^*w) \|_{2}
	&=\|E(fw^*w) \|_{2}
	= \|fE(w^*w) \|_{2}\\
	&=\|\pi(u_i)E(w^*w) \|_{2}
	=\|E(\pi(u_i)w^*w) \|_{2}
	=\|E(w^* (u_i\otimes 1)w) \|_{2}.
\end{align*}
Then since $H$ is finite dimensional, the last term is bounded by a finite sum of elements of the forms $\|E_B(b^*u_i a)\|_2$ for $a,b\in M$. Hence it converges to 0 and we get $w=0$, a contradiction.

	$(2)\Rightarrow(3)$ There exist $\delta>0$ and a finite set $F\subset M$ such that
	$$\sum_{x,y\in F}\|E_B(b^*ua)\|_{\varphi}^2\geq \delta,  \quad \text{for all }u\in\mathcal{U}(A).$$ 
Put $d_0:=\sum_{b\in F}be_B b^*\in \langle M,B\rangle^+$ and note that $\widehat{E}_B(d_0) = \sum_{b\in F}bb^* <\infty$. For any $u\in \mathcal U(A)$, in $L^2(M,\varphi)$ with the cyclic vector $\xi_\varphi$, we have 
\begin{equation*}
	\sum_{a\in F}\langle ud_0u^* \, a\xi_\varphi, a \xi_\varphi\rangle
	=\sum_{a,b\in F}\langle e_B b^*u^* a\xi_\varphi, b^*u^* a\xi_\varphi\rangle
	=\sum_{a,b\in F }\| E_B( b^*u^* a)\|_{\varphi}^2\geq \delta.
\end{equation*}
This inequality shows that (3) holds.

	$(3)\Rightarrow(1)$ By Theorem \ref{thmA}, there exists $d\in A'\cap \mathcal K(d_0,\Ad_A)$ with $d=dJ1_BJ$. It is a nonzero positive element by assumption, and satisfies $\widehat{E}_B(d)<\infty$ by Lemma \ref{lem-operator-bounded}. We get item (1) by \cite[Theorem 2(ii)]{BH16}.

	$(1)\Rightarrow(4)$ Take $(H,f,\pi,w)$ witnessing $A\preceq_M B$ such that $H$ is finite-dimensional. Suppose that there exists $\Psi$ as in item (4). Then for any $a\in \mathcal U(A)$, 
\begin{align*}
	(a\otimes 1_{H})( w (e_B\otimes 1_{H}) w^*)(a^*\otimes 1_{H})
	=  w \pi(a) (e_B\otimes 1_H)\pi(a^*)w^* = w(e_B\otimes 1_H)w^*.
\end{align*}
Observe that $\Psi\otimes \id_{\B(H)}\in \D(A\otimes \C 1_H \subset \langle M,B\rangle\ovt \B(H))$ as $H$ is finite-dimensional (note that it actually holds for arbitrary $H$). We have 
\begin{align*}
	(\Psi\otimes \id_{\B(H)})( w (e_B\otimes 1_{H}) w^*)
	=  w(e_{B}\otimes 1_{H})w^*.
\end{align*}
Since $H$ is finite dimensional, the left hand side of this equation is a finite sum of elements of the form $	\Psi( a e_B b^*)\otimes c$ for $a,b\in M$ and $c\in \B(H)$. By assumption, the left hand side is zero, hence $w=0$, a contradiction. 

	$(4)\Rightarrow(4')$ This is trivial.

	$(4')\Rightarrow(3)$ Suppose that (3) does not hold and we will construct $\Psi$ as in item $(4')$. Since (3) does not hold, $\mathcal K(d,\Ad_A)$ contains $0$ for all $d\in \mathfrak m_{\widehat{E}_B}^+$, hence by Lemma \ref{lem-operator-valued-weight-extension}, there exists $\Psi_0\in \D(A\subset \langle M,B\rangle )$ such that $\Psi_0(\mathfrak m_{\widehat{E}_B})=0$. We note that if $1_B\neq 1_M$, then $J1_BJ$ is a central projection in $\langle M,\widetilde{B}\rangle$, and one can still use Lemma \ref{lem-operator-valued-weight-extension}.

Since $A\subset M$ is with expectation, by Theorem \ref{thm-Marrakchi}, there exists a conditional expectation $E\in \D(A\subset M)$ onto $A'\cap M$. Take $\widetilde{E}\in \D(A\subset \langle M,B\rangle)$ such that $\widetilde{E}|_M=E$ by the compactness, and put $\Psi:=\Psi_0\circ E$. Then since $E(\mathfrak m_{\widehat{E}_B}) \subset \mathfrak m_{\widehat{E}_B}$ by Lemma \ref{lem-operator-bounded}, we have $\Psi(\mathfrak m_{\widehat{E}_B})=0$. Since $\Psi_0|_{A'\cap M}=\id$, $\Psi|_{M} = E$ is a conditional expectation. Thus $\Psi$ satisfies the condition in item (3). 

Finally, if assumption (a) holds, then the conditional expectation $E$ in the last paragraph can be taken as a faithful and normal conditional expectation, so that $\Psi|_{M} = E$ is faithful and normal. 
\end{proof}

The next proposition gives a general characterization of the condition $A\preceq_MB$. However, it is not useful because we cannot have the normality of $\Psi$, as we explained in Subsection \ref{Weak Dixmier semigroups}.

\begin{prop}\label{prop-intertwining1}
	Keep the setting as in Definition \ref{def-intertwining1}. The following conditions are equivalent. 
\begin{enumerate}
	\item We have $A \preceq_M B$.
	\item There exists no nets $(u_i)_{i}$ of unitaries in $\mathcal U(A)$ such that 
	$$(E_B\otimes \id_{\B(\ell^2)})(b^* (u_i\otimes 1) a ) \rightarrow 0, \quad \text{$\sigma$-strongly for all $a,b\in M1_B\ovt \B(\ell^2)$.}$$

	\item There exists no $\Psi \in \D(A\subset 1_A\langle M,\widetilde{B}\rangle 1_A)$ such that $(\Psi\otimes \id_{\B(\ell^2)})(a(e_{\widetilde{B}}\otimes 1) b^*) = 0$  for all $a,b\in M1_B\ovt \B(\ell^2)$.

\end{enumerate}
\end{prop}
\begin{proof}
	Observe that $A\preceq _M B$ holds if and only if so does $A\otimes \C 1\preceq_{M\ovt \B(\ell^2)}B\ovt \B(\ell^2) $. This indeed follows from definition. Then since $B\ovt \B(\ell^2)$ is properly infinite, we can apply Theorem \ref{thm-intertwining1} and get that all conditions are equivalent.
\end{proof}

\subsection*{Other formulations}

The following corollary is useful.

\begin{cor}\label{cor-intertwining1}
	Let $M$ be a $\sigma$-finite von Neumann algebra and $B_i\subset 1_{B_i}M1_{B_i}$ von Neumann subalgebras with expectations $E_{B_i}$ for $i\in I$. Let $A\subset 1_A M1_A$ be a von Neumann subalgebra with expectation. Assume that $A$ does not have any direct summand that is semifinite and properly infinite.

If $A\not\preceq_M B_i$ for all $i\in I$, then there exists $\Psi\in \D(A\subset 1_A \B(L^2(M))1_A)$ such that $\Psi|_{1_AM1_A} $ is a faithful normal conditional expectation onto $A'\cap 1_AM1_A$, and that
	$$ \Psi (\mathfrak m_{\widetilde{B}_i} J1_{B_i}J) =0,\quad \text{for all}\quad  i\in I.$$
\end{cor}
\begin{proof}
	By Theorem \ref{thm-Marrakchi}, take any $E\in \D(A\subset 1_AM1_A)$ that is a faithful normal conditional expectation onto $A'\cap 1_AM1_A$. By Theorem \ref{thm-intertwining1}(4), take $\Psi_i \in \D(A\subset 1_A \langle M,\widetilde{B}_i\rangle 1_A)$ such that $\Psi_i (\mathfrak m_{\widetilde{B}_i} J1_{B_i}J) =0 $. By the compactness of $\D(A\subset 1_A \B(L^2(M))1_A)$, we can extend $\Psi_i$ and $E$ on $1_A\B(L^2(M))1_A$, so that $\Psi_i,E\in \D(A\subset 1_A \B(L^2(M))1_A)$. As in the last part of the proof in Theorem \ref{thm-intertwining1}, we can replace $\Psi_i\circ E$ by $\Psi_i$, so that we can assume $\Psi_i |_{1_AM1_A} = E$ for all $i\in I$. 

Take a finite subset $F\subset I $ and we construct $\Psi_F\in \D(A\subset 1_A \B(L^2(M))1_A) $ such that $\Psi_F ( \mathfrak m_{B_i} J1_{B_i}J ) = 0$ for all $i\in F$ and $\Psi_F|_{1_AM1_A} = E $. For this, write $F=\{i_1,\ldots,i_n\}\subset I$ and define 
	$$ \Psi_F:=\Psi_{i_1}\circ \cdots \circ \Psi_{i_n}\in \D(A\subset 1_A \B(L^2(M))1_A)  .$$
Then as in the last part of the proof in Theorem \ref{thm-intertwining1}, we have $\Psi_F ( \mathfrak m_{B_i} J1_{B_i}J ) = 0$ for all $i\in F$ and $\Psi_F|_{1_AM1_A} = E $. 

Then consider a net $\{\Psi_F\}_F$ and find a subnet of it that converges to some $\Psi\in \D(A\subset 1_A \B(L^2(M))1_A) $ by the compactness. It is easy to see that $\Psi $ satisfies the desired condition. 
\end{proof}

\begin{rem}\upshape\label{rem-intertwining1}
	In Corollary \ref{cor-intertwining1}, assume further that there exists a large von Neumann algebra $M\subset \widetilde{M}$ with expectation $E_M \colon \widetilde{M}\to M$. Then there exists $\Psi\in  \D(A\subset 1_A \B(L^2(\widetilde{M}))1_A)$ such that $\Psi (e_M \cdot e_M)$ on $1_A \B(L^2(M))1_A$ satisfies the conclusion of Corollary \ref{cor-intertwining1}, and that $\Psi$ is faithful and normal on $1_A\widetilde{M}1_A$. 

To see this, take first $\Psi\in \D(A\subset 1_A \B(L^2(M))1_A)$ by Corollary \ref{cor-intertwining1} and approximate it by elements in $\D_{\rm alg}(A\subset 1_A \B(L^2(M))1_A)$. Then since each element in $\D_{\rm alg}(A\subset 1_A \B(L^2(M))1_A)$ naturally defines an element in $\D_{\rm alg}(A\subset 1_A \B(L^2(\widetilde{M}))1_A)$, by the compactness of $\D(A\subset 1_A \B(L^2(\widetilde{M}))1_A)$, we can find $\widetilde{\Psi}\in \D(A\subset 1_A \B(L^2(\widetilde{M}))1_A)$ such that $\Psi = \widetilde{\Psi} (e_M \, \cdot \, e_M) $ on $1_A \B(L^2(M))1_A$. Then as in the proof above, we compose $\widetilde{\Psi}\circ E$, where $E\in \D(A\subset 1_A \B(L^2(\widetilde{M}))1_A)$ satisfies that $E|_{1_A \widetilde{M}1_A}$ is a faithful normal conditional expectation onto $A'\cap 1_A\widetilde{M}1_A$, which is the desired element.
\end{rem}

Let $M$ be a von Neumann algebra with a faithful normal semifinite weight $\varphi$. Then $L^2(M,\varphi)$ is the completion of $\mathfrak n_\varphi$ by $\|\, \cdot\, \|_\varphi$, and we use the notation 
	$$ \Lambda_\varphi \colon \mathfrak n_\varphi \to L^2(M,\varphi) $$
for the canonical embedding.

\begin{lem}\label{lem-intertwining1}
Keep the setting as in Theorem \ref{thm-intertwining1}, thus we assume condition $\rm (a)$ or $\rm (b)$ in that theorem. Let $\varphi$ be a faithful normal semifinite weight $\varphi$ on $M$ such that $1_B \in M_\varphi$ and $\varphi \circ E_B = \varphi$ on $1_B M 1_B$. Take subsets $X \subset M$ and $Y \subset \mathfrak{n}_\varphi$ such that:
\begin{itemize}
\item $\mathop{\mathrm{span}} X^* \subset M$ is dense in the strong topology;
\item $\mathop{\mathrm{span}} Y \subset \mathfrak{n}_\varphi$ is dense in the $L^2$-norm $\|\, \cdot\, \|_\varphi$.
\end{itemize}
\begin{enumerate}
	\item Let $f = \sum_{i=1}^n t_i \Ad(u_i^*) \in \mathrm{DSG}_{\rm alg}(A \subset 1_A \langle M, \widetilde{B} \rangle 1_A)$, where $\sum_{i=1}^nt_i=1$ with $t_i\geq 0$ and $u_i\in \mathcal U(A)$. Then for each $x \in M$ and $y \in \mathfrak{n}_\varphi$,
	$$\langle f(x e_{\widetilde{B}}1_B x^*)\Lambda_\varphi(y), \Lambda_\varphi(y)\rangle= \sum_{i=1}^n t_i \| E_{B} (1_Bx^*u_i y1_B)\|_{\varphi}^2.$$

	\item Suppose there exists $\Psi \in \mathrm{DSG}(A \subset 1_A \langle M, \widetilde{B} \rangle 1_A)$ such that $\Psi|_{1_AM1_A}$ is normal and that
\[
\langle \Psi( x e_B x^* ) \Lambda_\varphi(y), \Lambda_\varphi(y) \rangle = 0, \quad \text{for all } x \in X,\, y \in Y.
\]
Then we have $A \not\preceq_M B$.

\end{enumerate}
\end{lem}
\begin{proof}
(1) This is just a computation. We set $e_B = e_{\widetilde{B}} 1_B = e_{\widetilde{B}} J 1_B J$. Then,
\begin{align*}
	\langle f(x e_{B} x^*)\Lambda_\varphi(y), \Lambda_\varphi(y) \rangle
	&= \sum_{i=1}^n t_i \langle u_i^* x e_{B} x^* u_i \Lambda_\varphi(y), \Lambda_\varphi(y) \rangle 
	= \sum_{i=1}^n t_i \| E_{B}(1_B x^* u_i y 1_B) \|_{\varphi}^2.
\end{align*}

(2) By definition, there exists a net $\{f_\lambda\}_\lambda$ in $\D_{\rm alg}(A\subset 1_A \langle M,\widetilde{B}\rangle 1_A)$ that converges to $\Psi$ in the $\sigma$-weak topology. Take one such $f_\lambda$ and write $f_\lambda = \sum_{i=1}^n t_i \Ad(u_i^*)$ for some $t_i\geq 0$ and $u_i \in \mathcal U(A)$. Let $x \in \mathop{\mathrm{span}} X$ and $y \in \mathop{\mathrm{span}} Y$ and we write $x = \sum_{k=1}^{m_x} x_k$ with $x_k \in X$ and $y = \sum_{\ell=1}^{m_y} y_\ell$ with $y_\ell \in Y$. By using the inequality $\left\| \sum_{k=1}^m a_k \right\|_\varphi^2 \leq 2^{m-1} \sum_{k=1}^m \|a_k\|_\varphi^2$ that follows by the parallelogram law, and by using item (1), we have:
\begin{align*}
	\langle f_\lambda(x e_{B} x^*)\Lambda_\varphi(y), \Lambda_\varphi(y) \rangle
	&= \sum_{i=1}^n t_i \| E_{B}(1_B x^* u_i y 1_B) \|_{\varphi}^2 \\
	&\leq 2^{m_x-1} 2^{m_y-1} \sum_{k=1}^{m_x} \sum_{\ell=1}^{m_y} \sum_{i=1}^n t_i \| E_{B}(1_B x_k^* u_i y_\ell 1_B) \|_{\varphi}^2 \\
	&= 2^{m_x + m_y - 2} \sum_{k=1}^{m_x} \sum_{\ell=1}^{m_y} \langle f_\lambda(x_k e_{B} x_k^*)\Lambda_\varphi(y_\ell), \Lambda_\varphi(y_\ell) \rangle.
\end{align*}
We take the limit with respect to $\lambda$ and obtain
\[
0 \leq \langle \Psi(x e_{B} x^*)\Lambda_\varphi(y), \Lambda_\varphi(y) \rangle \leq 2^{m_x + m_y - 2} \sum_{k=1}^{m_x} \sum_{\ell=1}^{m_y} \langle \Psi(x_k e_{B} x_k^*)\Lambda_\varphi(y_\ell), \Lambda_\varphi(y_\ell) \rangle = 0.
\]
Thus, $\langle \Psi(x e_{B} x^*)\Lambda_\varphi(y), \Lambda_\varphi(y) \rangle =0$ for all $x \in \mathop{\mathrm{span}} X$ and $y \in \mathop{\mathrm{span}} Y$, and by the polarization identity we conclude $\Psi(x_1 e_{B} x_2^*) = 0$ for all $x_1, x_2 \in \mathop{\mathrm{span}} X$. 

Let $x \in M$ and $x_2\in \mathop{\mathrm{span}} X$ be arbitrary elements and take a net  $(x_i)_i$ in $\mathop{\mathrm{span}} X$ such that $x_i^* \to x^*$ $\sigma$-strongly. Take a normal state $\omega$, then by the Cauchy-Schwarz inequality, we have:
\[
|\omega \circ \Psi((x_i - x) e_{B} x_2^*)| \leq \| (x_i^* - x^*) 1_A\|_{\omega \circ \Psi} \| e_{B} x_2^* 1_A\|_{\omega \circ \Psi} \to 0,\quad \text{as }i\to \infty,
\]
where we used the fact that $\Psi$ is normal on $1_AM1_A$. Since $\omega$ is any normal state, this shows that $\Psi((x_i - x) e_{B} x_2^*) \to 0$ in the $\sigma$-weak topology. We conclude that $\Psi(x e_{B} x_2^*) = 0$. A similar computation shows that $\Psi(x e_{B} y^*) = 0$ for all $x,y\in M$ and this implies $A\not\preceq_M B$ by Theorem \ref{thm-intertwining1}(4). 
\end{proof}

\subsection*{Remark on Theorem \ref{thm-intertwining1}(2)}

	As in \cite[Theorem 4.3(5)]{HI15}, it is natural to ask the following conditions are also equivalent to conditions in Theorem \ref{thm-intertwining1}: let $X\subset M$ be a $\sigma$-weakly dense subset, then
\begin{enumerate}
	\item[$\rm (2')$] the same condition as in item $(2)$, but the convergence holds only for $a,b\in 1_A X 1_B$;

	\item[$\rm (2'')$] the same condition as in item $(2)$, but the convergence holds only for $a\in 1_A M1_B$ and $b\in 1_A X 1_B$.

\end{enumerate}
By the standard triangle inequality argument, it is easy to see that $(2)$ and $(2'')$ are equivalent. Here we show that item $(2')$ is \textit{not} equivalent to item (2).

We note that Lemma \ref{lem-intertwining1} is a substitute of item $(2')$ in terms of the weak Dixmier property, which works by the normality of the associated ucp map.

\begin{prop}
	Let $M$ be a $\sigma$-finite type $\rm III_\lambda$ factor for $0<\lambda <1$ with the discrete decomposition $B\rtimes \Z = M$. Then the inclusion of continuous cores $C(B)\subset C(M)$ satisfy that $C(M)\preceq_{C(M)}C(B)$, hence item $(2)$ does not hold, while it has a net $(u_i)_i$ in $\mathcal U(C(M))$ such that $(2')$ does hold for 
	$$ X = \mathrm{span}\{ x\lambda_t \mid x\in M,\ t\in \R\}\subset C(M). $$
In particular, the convergence condition in item $(2')$ can not be extended on the one for $M$.
\end{prop}
\begin{proof}
	Let $E_B$ be the unique faithful normal conditional expectation and $\varphi\in M_\ast$ a faithful normal state such that $\varphi \circ E_B=\varphi$. By \cite[Theorem 4.9]{HI17}, we have $M\not \preceq_M B$ and $C_\varphi(M)\preceq_{C_\varphi(M)}C_\varphi(B)$. Note that $M,B,C_\varphi(M)$ and $C_\varphi(B)$ are all properly infinite. 

Since $M$ of type III and $B$ is semifinite, we also have $(M,\sigma^\varphi)\not\preceq_{M}(B,\sigma^\varphi)$, see Definition \ref{def-intertwining2} below. Hence by Theorem \ref{thm-intertwining2} below, there exists a net $(u_i)_{i}$ in $\mathcal U(A)$ and $(g_i)_i$ in $\R$ such that 
	$$E_B(\sigma^\varphi_{g_i}(b^*) u_i a ) \rightarrow 0, \quad \text{$\sigma$-strongly for all $a,b\in M$.}$$
We show that the net $w_i:= \lambda_{g_i^{-1}}u_i $ does the job. Indeed, for any $a,b\in M$ and $s,s'\in \R$, we have 
\begin{align*}
	 E_{C_\varphi(B)}(\lambda_s b^* w_i a\lambda_{s'}) 
	= \lambda_{sg_i^{-1}} E_{B}(\sigma^\varphi_{g_i}(b^*) u_i a)\lambda_{s'}.
\end{align*}
Since $\lambda_{sg_i^{-1}}$ is a unitary, this term converges to 0 strongly as $i\to \infty$. 
\end{proof}

\subsection{Intertwining theorem with group actions}

Let $M$ be a $\sigma$-finite von Neumann algebra and $A,B\subset M$ (possibly non-unital) von Neumann subalgebras with expectations $E_A,E_B$. We define the following group actions. Let $G$ be a locally compact second countable group, and consider continuous actions $\alpha$ and $\beta$ of $G$ on $M$ such that 
\begin{itemize}
	\item $\alpha_g(A) =A$ and $\beta_g(B) = B$ for all $g\in G$;
	\item $\alpha_g\circ E_A = E_A \circ \alpha_g$ on $1_AM1_A$ and $\beta_g\circ E_B = E_B\circ \beta_g$ on $1_BM1_B$ for all $g\in G$;
	\item $\alpha$ and $\beta$ are cocycle conjugate, that is,  there exists 
a $\beta$-cocycle $\omega \colon G \to M$ such that $\alpha_g=\Ad(\omega_g) \circ \beta_g $ for all $g\in G$.

\end{itemize}
Recall that a \emph{generalized $\beta$-cocycle with support projection $p\in M$} is a strongly continuous map $u\colon G \to pM$ that satisfies $u_gu_g^* =p$, $u_g^*u_g=\beta_g(p)$, and $u_{gh} = u_g \beta_g(u_h)$ for all $g,h\in G$. When $p=1$, we simply say that it is a $\beta$-cocycle. We are mainly interested in the case that $G=\R$, $\alpha = \sigma^\psi$, $\beta=\sigma^\varphi$.

The following definition comes from \cite[Lemma 2.6 and Definition 3.1]{Is19}. The original concept was found in some proofs in \cite{Po05a}.

\begin{df}\label{def-intertwining2}\upshape
	We say that $(A,\alpha)$ \textit{embeds with expectation into $(B,\beta)$ inside $M$} and write $(A,\alpha) \preceq_{M} (B,\beta)$ if there exist $(H,f,\pi,w)$ witnessing $A\preceq_M B$ and a generalized cocycle $(u_g)_{g\in G}$ for $\beta\otimes \id_H$ with values in $B\ovt \B(H)$ and with support projection $f$ such that 
\begin{itemize}
	\item[$\rm(i)$] $wu_g =  (\omega_g\otimes 1_H )(\beta_g\otimes \id_H) (w)$ for all $g\in G$;
	\item[$\rm(ii)$]  $u_g(\beta_g\otimes \id_H)(\pi(a))u_g^* = \pi(\alpha_g(a))$ for all $g\in G$ and $a\in A$.
\end{itemize}
In this case, we say that \textit{$(H,f,\pi,w)$ and $(u_g)_{g\in G}$ witness $(A,\alpha) \preceq_{M} (B,\beta)$}. 
\end{df}

As explained in \cite[Definition 3.1]{Is19}, if $w$ in the above definition is not a partial isometry, we can exchange it with a partial isometry by the polar decomposition. 

\subsection*{Basic theory}

Keep the setting as in Definition \ref{def-intertwining2}.

\begin{lem}\label{lem-intertwining-action1}
	Let $(H,f,\pi,w)$ and $\{u_g\}_g$ be witnessing $(A,\alpha) \preceq_{M} (B,\beta)$. We do not assume that $w$ is a partial isometry.
\begin{enumerate}

	\item If there is a projection $z\in \mathcal Z(B)^\beta$ such that $w(z\otimes 1)\neq 0$, then putting $z^H:=z\otimes 1\in \mathcal Z(B)^\beta\ovt \B(H)$, $(H,fz^H,\pi(\,\cdot\, )z^H,wz^H)$ and $\{u_gz^H\}_g$ witness $(A,\alpha) \preceq_{M} (Bz,\beta)$. 

	\item If there is a projection $e\in A^\alpha$ such that $\pi(e)w\neq 0$, then $(H,\pi(e),\pi|_{eAe},w\pi(e))$ and $\{\pi(e)u_g\}_g$ witness $(eAe,\alpha) \preceq_{M} (B,\beta)$.

	\item If there is a projection $f'\in B\ovt \B(H)$ that is equivalent to $f$, then with a partial isometry $v\in  B\ovt \B(H)$ with $v^*v =f$ and $vv^* =f'$, $(H,f',\Ad(v)\circ \pi,wv^*)$ and $\{vu_g\beta_g^H(v^*)\}_g$ witness $(A,\alpha) \preceq_{M} (B,\beta)$.

\end{enumerate}
\end{lem}
\begin{proof}
	(1) This is easy because $z\otimes 1$ commutes with $\pi(A)$ and $\beta\otimes \id$. 

	(2) Observe that $\pi\colon eAe \to \pi(e)(B\ovt \B(H))\pi(e)$ is a unital $\ast$-homomorphism, $w\pi(e) = (1\otimes e_{11})w\pi(e)$, and $(H,fz^H,\pi(\,\cdot\, )z^H,wz^H)$ witness $eAe \preceq_{M} B$. 
By item (ii) in the definition, for all $a\in A$ and $g\in G$,
	$$\pi(\alpha_g(eae))=\pi(e\alpha_g(eae)e)=\pi(e)u_g \beta_g^H(\pi(eae))u_g^*\pi(e) .$$
We get item (ii) for $eAe$ and $g\in G$. Since $\alpha_g(e)=e$, this implies for $g\in G$,
	$$\pi(e) = \pi(\alpha_g(e)) = u_g\beta_g^H(\pi(e))u_g^*.$$
This means that $\pi(e)u_g$ has the left support $\pi(e)$ and the right support $\beta_g^H(\pi(e))$. Now it is easy to check that $\{\pi(e)u_g\}_g$ is a generalized cocycle and that item (i) holds. 

	(3) It is easy to see that $(H,f',\Ad(v)\circ \pi,wv^*)$ witness $A \preceq_{M} B$. Put  $\tilde{u}_g:= vu_g \beta^H_g(v^*)$ for $g\in G$, and observe that it is a partial isometry that has the left support $f'$ and the right support $\beta_g^H(f')$. It is easy to compute that 
$\tilde{u}_{gh} = \tilde{u}_g \beta^H_g(\tilde{u}_h)$ for $g,h\in G$. It is also easy to check item (i) and (ii). 
\end{proof}

\begin{lem}\label{lem-intertwining-action2}
	The following assertions hold.
\begin{enumerate}
	\item Put $\B:=\B(\ell^2)$, and extend $\alpha,\beta $ on $M\ovt \B$ as trivial actions on $\B$. Then $(A,\alpha) \preceq_{M} (B,\beta)$ holds if and only if $(A\otimes \C ,\alpha\otimes \id) \preceq_{M\ovt \B} (B\ovt \B,\beta\otimes \id)$. 

	\item If there exist projections $z\in \mathcal Z(A)^\alpha$ and $q\in B^\beta$ such that $(Az,\alpha) \preceq_{M} (qBq,\beta)$, then $(A,\alpha) \preceq_{M} (B,\beta)$.

	\item Let $z_i\in \mathcal Z(A)^\alpha$ and $w_j\in \mathcal Z(B)^\beta$ be projections with $1_A = \sum_{i\in I} z_i$ and $1_B = \sum_{j\in J} w_j$. Then $(A,\alpha) \preceq_{M} (B,\beta)$ holds if and only if there exist some $i,j$ such that $(Az_i,\alpha) \preceq_{M} (Bw_j,\beta)$.
\end{enumerate}
\end{lem}
\begin{proof}
	(1) Let $(H,f,\pi,w)$ and $\{u_g\}_g$ be witnessing $(A,\alpha) \preceq_{M} (B,\beta)$. Since $H$ is separable, we may assume $\B(H)\subset \B$, hence $(H,f,\pi,w)$ and $\{u_g\}_g$ witness $(A\otimes \C ,\alpha\otimes \id) \preceq_{M\ovt \B} (B\ovt \B,\beta\otimes \id)$. 

Conversely if $(H,f,\pi,w)$ and $\{u_g\}_g$ witness $(A\otimes \C ,\alpha\otimes \id) \preceq_{M\ovt \B} (B\ovt \B,\beta\otimes \id)$. Then using $H\ovt \ell^2$, we get $(A,\alpha) \preceq_{M} (B,\beta)$.

	(2) This is obvious, one can compose the map $A \mapsto Az$. 

	(3) If $(Az_i,\alpha) \preceq_{M} (Bw_j,\beta)$, we have $(A,\alpha) \preceq_{M} (B,\beta)$ by item (2). To see the converse, let $(H,f,\pi,w)$ and $\{u_g\}_g$ be witness $(A,\alpha) \preceq_{M} (B,\beta)$. Since $w = (1\otimes e_{11})wf$, there exist $z_i,w_j$ such that 
	$$(z_i\otimes 1)w(w_j\otimes 1)\neq 0 .$$
We can then apply Lemma \ref{lem-intertwining-action1}(1) and (2) repeatedly. 
\end{proof}

\begin{lem}\label{lem-intertwining-finite-dimensional1}
	Assume that either $A$ is of type $\rm III$, or $B$ is properly infinite. If $(A,\alpha) \preceq_M (B,\beta)$ holds, then there exist $(H,f,\pi,w)$ and $\{u_g\}_g$ witnessing the condition such that $H$ is finite-dimensional.
\end{lem}
\begin{rem}\upshape
	Let $A=A_1\oplus A_2$, where $A_1$ is of type $\rm III$ and $A_2$ is finite (possibly $A_1=0$ or $A_2=0$), and assume that $A_2$ has a $\alpha$-invariant faithful normal state $\psi_2 \in (A_2)_\ast$. In this setting, we can also choose $H$ to be finite-dimensional. This will be proved in the last part in the proof of Theorem \ref{thm-intertwining2}.
\end{rem}
\begin{proof}
	Assume that $B$ is properly infinite. Take $(H,f, \pi,w)$ and $\{u_g\}_g$. Then for the minimal projection $e_{11}$ in $\B(H)$, since $1\otimes e_{11}$ is infinite in $B\ovt \B(H)$ and since $f\in B\ovt \B(H)$ is a $\sigma$-finite projection, we have $f\sim f' \leq 1\otimes e_{11}$ in $B\ovt \B(H)$. We apply Lemma \ref{lem-intertwining-action1}(3) and we may assume $f\leq 1\otimes e_{11}$. Then we indeed get that $(\C e_{11},f, \pi,w)$ and $\{u_g\}_g$ witness the condition. 

Assume next that $A$ is of type III. 
We  decompose $B = B_1 \oplus B_2$ as $B_1$ is of type III and $B_2$ is semifinite. Then by Lemma \ref{lem-intertwining-action2}(3), we get $(A,\alpha) \preceq_M (B_k,\beta)$ for some $k\in \{1,2\}$. But since $A$ is of type III, we indeed get $k=1$. Then since $B_1$ is properly infinite, we can use the result in the previous paragraph.
\end{proof}

\subsection*{Intertwining theorem with group actions}

The following theorem is an appropriate generalization of results in \cite{Po05a,HSV16} and \cite[Theorem 3.2]{Is19}. 

\begin{thm}\label{thm-intertwining2}
	Keep the setting as in Definition \ref{def-intertwining2}. Assume that either
\begin{itemize}
	\item[$\rm (a)$] $A=A_1\oplus A_2$, where $A_1$ is of type $\rm III$,  and $A_2$ is finite with $\alpha$-invariant faithful normal state $\psi_2 \in (A_2)_\ast$ (possibly $A_1=0$ or $A_2=0$); or
	\item[$\rm (b)$] $B$ is properly infinite.
\end{itemize}
Further, assume that either
\begin{itemize}
	\item[$\rm (a')$] there exists an $\alpha$-invariant faithful normal state on $A$; or
	\item[$\rm (b')$]  $G$ is amenable.
\end{itemize}
Put $\widetilde{G}:=\mathcal U(A)\rtimes_\alpha G$ and define an action (as a discrete group) by
	$$ \gamma \colon \widetilde{G}  \actson 1_A \langle M,\widetilde{B}\rangle 1_A;\quad \gamma_{(u,g)}(T) = u \widehat{\alpha}_g(T) u^*.$$
Then the following conditions are equivalent. 
\begin{enumerate}
	\item We have $(A,\alpha) \preceq_M (B,\beta)$.
	\item There exists no nets $(u_i)_{i}$ of unitaries in $\mathcal U(A)$ and $(g_i)_i$ in $G$ such that 
	$$E_B(\beta_{g_i}(b^*)\omega_{g_i}^* u_i a ) \rightarrow 0, \quad \text{$\sigma$-strongly for all $a,b\in M1_B$.}$$

	\item There exists $d_0\in \langle M,\widetilde{B}\rangle^+$ such that $\widehat{E}_{\widetilde{B}}(d_0)<\infty$, $d_0 = d_0 J1_BJ$ and 
	$$ \mathcal K(d_0,\gamma)=\overline{{\rm co}}^{\rm weak} \left \{ u \widehat{\alpha}_g(d_0) u^* \mid u \in \mathcal U(A),\ g\in G\right \} \subset (1_A \langle M, \widetilde B\rangle 1_A)^+$$
does not contain $0$. Here $\widehat{\alpha}_g := \Ad(\omega_g)\circ \widehat{\beta}_g$, $\widehat{\beta}_g := \Ad(u_g^\beta)$, and $u_g^\beta \in \mathcal U(L^2(M))$ is the canonical implementation of $\beta_g$ for all $g\in G$.

	\item There exists no $\Psi \in \D(\gamma)$ such that $\Psi(ae_{\widetilde{B}}b^*) = 0$ for all $a,b\in M1_B$.

	\item[$(4')$] There exists no $\Psi \in \D(\gamma)$ such that $\Psi(\mathfrak m_{\widehat{E}_{\widetilde{B}}} J1_BJ) = 0$  and $\Psi|_{1_AM1_A}$ is a conditional expectation onto $A'\cap 1_A M^\alpha 1_A$.

\end{enumerate}
In assumption $\rm (a)$ and $\rm (a')$, if $(A,\alpha) \not\preceq_M (B,\beta)$, then we can choose $\Psi$ in item $(4')$ so that $\Psi|_{1_AM1_A}$ is a faithful normal conditional expectation. 
\end{thm}
\begin{rem}\upshape
	We note that assumption (a) or (b) is used to show that, if $(A,\alpha) \preceq_M (B,\beta)$, then we can choose $(H,f,\pi,w)$ and $\{u_g\}_g$ such that $H$ is finite-dimensional. 
Assumption $\rm (a')$ or $\rm (b')$ is used to prove $(4)\Rightarrow(3)$ and $(3)\Rightarrow(1)$.
\end{rem}
\begin{proof}
	Fix faithful states $\varphi\in M_\ast$ such that $1_B \in M_\varphi$ and $\varphi\circ E_B = \varphi $ on $1_BM1_B$.  As in the proof of Theorem \ref{thm-intertwining1}, we assume that $A,B\subset M$ are unital subalgebras. We first accept the following fact: in assumption (a) or (b), if $(A,\alpha) \preceq_M (B,\beta)$, then we can choose $(H,f,\pi,w)$ and $\{u_g\}_g$ such that $H$ is finite-dimensional. We prove this fact in the last part of the proof.

	For implications of $(1)\Rightarrow(2)$ and $(2)\Rightarrow(3)$, proofs are really the same as ones in Theorem \ref{thm-intertwining1}. Indeed, one can compute in $(1)\Rightarrow(2)$ that 
\begin{align*}
	\|E(w^*w) \|_{2}
	=\|E((\beta_{g_i}\otimes \id_H)(w^*) (\omega_{g_i}^* u_i\otimes e_{1,1})w) \|_{2},
\end{align*}
and in $(2)\Rightarrow(3)$ that 
\begin{align*}
	\sum_{y\in\mathcal{F}}\langle w^*\widehat{\alpha}_g(d_0)w \, a\xi_\varphi,a\xi_\varphi\rangle_{\varphi}
	&=\sum_{a,b\in F}\|\rE_B(\beta_g(b^*)\omega_g wa)\|_{\varphi}^2\geq \delta.
\end{align*}
Then exactly same arguments work, so we omit them. See \cite[Theorem 3.2]{Is19} for similar computations.

	$(3)\Rightarrow(1)$ Put $E:=E_A \circ \widehat{E}_{B}$ and observe that $\widehat{E}_{B} \circ \widehat{\gamma}_g = \widehat{\gamma}_g \circ \widehat{E}_B$ and $E\circ \widehat{\gamma}_g = \widehat{\gamma}_g\circ E$ for all $g\in \widetilde{G}$. 
By assumption $\rm (a')$ or $\rm (b')$, we can use Proposition \ref{prop-averaging-amenable} for the action $\gamma$ and $E$. We get $ \mathfrak m_{E}\subset \mathcal D(\gamma)$, and hence $\mathcal K(d_0, \gamma)^{\gamma}\neq \emptyset$, as $d_0$ is contained in $\mathfrak m_E$. 
Take any $d\in \mathcal K(d_0, \gamma)^{\gamma}$ that is nonzero by assumption. Observe that $d$ is contained in $A'\cap \langle M,B\rangle^{\widehat{\alpha}}$. 
Next we apply Lemma \ref{lem-operator-bounded} for $\gamma$ and $\widehat{E}_B$, and get that $d\in \mathcal K(d_0,\gamma) \subset \mathfrak m_{\widehat{E}_B}$. We conclude $\widehat{E}_B(d)<\infty$, and hence we get item (1) by \cite[Theorem 3.2(4)]{Is19}.

	$(1)\Rightarrow(4)$ We can take $(\pi,H,f,w)$ and $\{u_g\}_g$ such that $H$ is finite-dimensional. If there exists $\Psi$ as in (4), Then for each $a\in \mathcal U(A)$ and $g\in G$, 
\begin{align*}
	(a\otimes 1_{H})\widehat{\alpha}_g( w (e_B\otimes 1_{H}) w^*)(a^*\otimes 1_{H})
	&=  	(a\omega_g^H\otimes 1_{H})( \beta_g^H(w) (e_B\otimes 1_{H}) \beta_g^H(w^*)(\omega_g^H)^*a^*\otimes 1_{H})\\
	&=  	(a\otimes 1_{H})( wu_g (e_B\otimes 1_{H}) u_g^* w^*(a^*\otimes 1_{H})\\
	&=  	 w \pi(a)u_g(e_B\otimes 1_{H}) u_g^* \pi(a^*)w^*
	 = w(e_B\otimes 1_{H})w^*.
\end{align*}
This implies that 
\begin{align*}
	(\Psi\otimes \id_{\B(H)})( w (e_B\otimes 1_{H}) w^*)
	=  w(e_{B}\otimes 1_{H})w^*.
\end{align*}
Then since $H$ is finite-dimensional, by assumption of $\Psi$, the left hand side of the equation is zero. This implies that $w=0$, a contradiction.

	$(4)\Rightarrow(4')$ This is trivial.

	$(4')\Rightarrow(3)$ Assume that (3) does not hold and we construct $\Psi$ as in item $(4')$. Since (3) does not hold, for each $d\in \mathfrak m_{\widehat{E}_B}^+$, $\mathcal K(d,\gamma)$ contains $0$. 
We apply Lemma \ref{lem-operator-valued-weight-extension} and find $\Psi_0\in \D(\gamma)$ such that $\Psi_0(\mathfrak m_{\widehat{E}_{B}})=0$. 

Since $A\subset M$ is with expectation $E_A$ that commutes with $\alpha$, by assumption $(a')$ or $(b')$, we can use Proposition \ref{prop-averaging-amenable} to the inclusion $A\subset^{E_A} M$ with $\alpha$, and get 
	$$M=\mathfrak m_{E_A}\subset \mathcal D(\gamma\colon \widetilde{G}\actson M).$$
By Lemma \ref{lem-Schwartz}, take a conditional expectation $E\in \D(\gamma\colon \widetilde{G}\actson M)$ onto $M^\gamma = A'\cap M^\alpha$. We extend $E$ to a map $\widetilde{E}$ contained in $\D(\gamma \colon \widetilde{G}\actson \langle M,B\rangle)$. 
Put $\Psi:=\Psi_0\circ \widetilde{E}\in \D(\gamma)$ and we show that it satisfies the conclusion. 
Indeed, by Lemma \ref{lem-operator-bounded}, we have $\widetilde{E}(\mathfrak m_{\widehat{E}_{B}})\subset \mathfrak m_{\widehat{E}_{B}}$ and hence 
$\Psi(\mathfrak m_{\widehat{E}_{B}})=0$. Since $\Psi_0=\id$ on $A'\cap M^\alpha$, $\Psi|_{M} = E$ is a conditional expectation. Thus we get the conclusion.

Suppose that assumption (a) and $\rm (a')$ holds. Then by Proposition \ref{prop-averaging-amenable}, the conditional expectation $E\in \D(\gamma\colon \widetilde{G}\actson M)$ above can be chosen as a faithful and normal one.

\bigskip

Finally we prove that if $(A,\alpha)\preceq_M (B,\beta)$ and if assumption (a) or (b) holds, then we can choose $(H,f,\pi,w)$ and $\{u_g\}_g$ such that $H$ is finite-dimensional. The case of assumption (b) is already proved in Lemma \ref{lem-intertwining-finite-dimensional1}, so we consider the case of assumption (a). 

By Lemma \ref{lem-intertwining-action2}, we may assume that $A$ is finite or of type III, and that $B$ is finite or properly infinite. Then the case $B$ is properly infinite was already proved, hence we assume $B$ is finite. Then $A$ must be finite, hence we can assume $A,B$ are finite. We then prove this case in Proposition \ref{prop-intertwining-action-finite} below. 
\end{proof}

Since we have already proved Theorem \ref{thm-intertwining2} in the case that $B$ is properly infinite, we immediately obtain the following proposition.

\begin{prop}\label{prop-intertwining2}
	Keep the setting as in Definition \ref{def-intertwining2}. Assume that either
\begin{itemize}
	\item $G$ is amenable; or $A$ has a $\alpha$-invariant faithful normal state $\psi\in A_\ast$.
\end{itemize}
Then the following conditions are equivalent.
\begin{enumerate}
	\item $(A,\alpha) \preceq_M (B,\beta)$.

	\item There exist no net $(u_i)_{i}$ in $\mathcal U(A)$ and $(g_i)_i$ in $G$ such that
	$$(E_{B}\otimes \id_{\B(\ell^2)})( (\beta_{g_i}\otimes \id)(b^*)(\omega_{g_i}^* u_i \otimes 1)  a ) \rightarrow 0, \quad \text{$\sigma$-strongly for all $a,b\in M1_B\ovt \B(\ell^2)$.}$$

	\item There exists no $\Psi \in \D(A\subset 1_A\langle M,\widetilde{B}\rangle 1_A)$ such that $(\Psi\otimes \id_{\B(\ell^2)})(a (e_{\widetilde{B}}\otimes 1) b^*J1_BJ) = 0$ for all $a,b\in M1_B\ovt \B(\ell^2)$.

\end{enumerate}
\end{prop}
\begin{proof}
	By Lemma \ref{lem-intertwining-action2}, (1) is equivalent to $(A\otimes \C,\alpha\otimes \id) \preceq_{M\ovt \B(\ell^2)} (B\ovt \B(\ell^2),\beta\otimes \id)$. Then $B\ovt \B(\ell^2)$ is properly infinite and we can apply Theorem \ref{thm-intertwining2}. 
\end{proof}

\subsection*{Technical lemmas when $A$ is finite with $\alpha$-invariant state}

\begin{lem}\label{lem-intertwining-action-finite1}
	Keep the setting as in Definition \ref{def-intertwining2} and assume that $A$ is finite with $\alpha$-invariant faithful normal state $\psi_A$. Fix a $\ast$-strongly dense subset $X\subset (M)_1$. The following conditions are equivalent.
\begin{enumerate}
	\item There exists no nets $(u_i)_{i}$ in $\mathcal U(A)$ and $(g_i)_i$ in $G$ such that 
	$$E_{B}( \beta_{g_i}(b^*)\omega_{g_i}^* u_ia ) \rightarrow 0, \quad \text{$\sigma$-strongly for all $a,b\in 1_AM1_B$.}$$

	\item The same condition as item $(1)$, but the convergence holds for all  $a,b\in X1_B$.

\end{enumerate}
\end{lem}
\begin{proof}
	The implication $(2)\Rightarrow(1)$ is trivial. Assume that there exist $(u_i)_{i}$ and $(g_i)_i$ that satisfies the convergence for all $a,b\in X1_B$. By the triangle inequality, it is easy to see that the convergence for all $a\in M1_B$ and $b\in X1_B$. 

Take $a$ as an analytic element for $\varphi$, hence $a\xi_\varphi = \xi_\varphi a'$ for some $a'\in M$. Then for $b\in M1_B$, $g\in G$ and $u\in \mathcal U(A)$, we have
\begin{align*}
	 \|E_{B}( \beta_{g}(b^*)\omega_{g}^* ua )\|_{2,\varphi} 
	&\leq  \| \beta_{g}(b^*)\omega_{g}^* ua \|_{2,\varphi} 
	=  \| \omega_{g}^*\alpha_{g}(b^*) ua \|_{2,\varphi} \\
	&\leq \| \alpha_{g}(b^*) ua \|_{2,\varphi} 
	\leq \|a'\|_\infty \| \alpha_{g}(b^*) u \|_{2,\varphi}.
\end{align*}
Now fix any element $b\in M1_B$ with $\|b\|_\infty \leq 1$, and take a $\ast$-strong approximation $X\ni b_\lambda \to b\in M$. By the above inequality, we have only to show
	$$\sup_{g\in G,\ u\in \mathcal U(A)}\| \alpha_{g}(b^*-b_\lambda^*) u \|_{2,\varphi}  \to 0.$$
Observe that all elements in the norm inside this supremum is contained in a bounded set $(M)_2$. Since $M$ is $\sigma$-finite, all $L^2$-norms by faithful normal states on $(M)_2$ are equivalent, so there are constants $C,D>0$ such that
	$$ \|T \|_{2,\varphi} \leq C\|T\|_{2,\omega}\leq D\|T\|_{2,\psi},\quad T\in (M)_2 ,$$
where $\psi:= \psi_A \circ E_A$ and $\omega:=\tau_A \circ E_A$, and $\tau_A$ is a trace on $A$. Then $\omega$ has the trace condition on $A$, and $\psi$ is invariant by $\alpha$, so for all $u\in \mathcal U(A)$ and $g\in G$,
\begin{align*}
	\|\alpha_{g}(b^*-b_\lambda^*) u \|_{2,\varphi} 
	&\leq C\|\alpha_{g}(b^*-b_\lambda^*) u\|_{2,\omega}
	= C\|\alpha_{g}(b^*-b_\lambda^*) \|_{2,\omega}\\
	&\leq D\|\alpha_{g}(b^*-b_\lambda^*)\|_{2,\psi}
	=D\|b^*-b_\lambda^*\|_{2,\psi}.
\end{align*}
Since $b_\lambda \to b$ is a $\ast$-strong convergence, we get
	$$\sup_{g\in G,\ u\in \mathcal U(A)}\| \alpha_{g}(b^*-b_\lambda^*) u \|_{2,\varphi} \leq D\|b^*-b_\lambda^*\|_{2,\psi} \to 0.$$
This is the desired condition.
\end{proof}

\begin{prop}\label{prop-intertwining-action-finite}
	Keep the setting as in Definition \ref{def-intertwining2}. Assume that $A,B$ are finite and that there exists a $\alpha$-invariant faithful normal state $\psi_A\in A_\ast$. Let $\mathrm{ctr}_B$ be the center valued trace on B. Recall that  $\langle M,\widetilde{B}\rangle J1_BJ$ admits an extended center trace $\mathrm{ctr}_{\langle M,\widetilde{B}\rangle}$ such that 
	$$\mathrm{ctr}_{\langle M,\widetilde{B}\rangle}((x^*e_{\widetilde{B}}x) J1_BJ) = J\mathrm{ctr}_{B}\circ E_B(1_Bxx^*1_B)J, \quad x\in M,$$
(e.g.\ \cite[Subsection 2]{Is19}). Then the following conditions are equivalent.
\begin{enumerate}
	\item We have $(A,\alpha)\preceq_M (B,\beta)$.

	\item There exists a nonzero positive element $d\in A'\cap 1_A\langle M,\widetilde{B}\rangle^{\widehat{\alpha}} 1_A$ such that
	$$d=dJ1_BJ,\quad \widehat{E}_{\widetilde{B}}(d)<\infty,\quad\text{and}\quad \mathrm{ctr}_{\langle M,\widetilde{B}\rangle}(d)<\infty.$$

\end{enumerate}
If one of these conditions hold, then there exists $(H,f,\pi,w)$ and $\{u_g\}_g$ that witness $(A,\alpha)\preceq_M (B,\beta)$ and that $H$ is finite-dimensional.
\end{prop}
\begin{proof}
	$(1)\Rightarrow(2)$ We write
	$$B_\infty:=B\ovt \B(\ell^2)\subset M_\infty := M\ovt \B(\ell^2),\quad \beta^\infty:=\beta\otimes \id.$$
Then by Proposition \ref{prop-intertwining2}, there exist no nets 
$(u_i)_{i}$ in $\mathcal U(A)$ and $(g_i)_i$ in $G$ such that
	$$E_{B_\infty}( \beta_{g_i}^\infty(b^*)(\omega_{g_i}^* u_i \otimes 1)  a ) \rightarrow 0, \quad \text{$\sigma$-strongly for all $a,b\in M1_B\ovt \B(\ell^2)$.}$$
Let $Q_n \in \B(\ell^2)$ be an increasing sequence of finite rank projections converging to $1$ strongly, and put $q_n=1_B\otimes Q_n \in B_\infty$. Then by Lemma \ref{lem-intertwining-action-finite1}, there exists no nets $(u_i)_{i}$ in $\mathcal U(A)$ and $(g_i)_i$ in $G$ such that the above convergence holds for all $a,b\in X1_{B_\infty}$, where $X$ is the unit ball of the $\sigma$-weakly dense subalgebra $\bigcup_{n}q_n B_\infty q_n$. Thus, there exists $\delta>0$ and a finite subset $F\subset X 1_{B_\infty}$ such that
	$$\sum_{a,b\in F}\|E_{B_\infty}( \beta^\infty_{g}(b^*)(\omega_{g}^* u \otimes 1)  a )\|_{\varphi}^2\geq\delta,  \quad \text{for all } g\in G,\ u\in\mathcal{U}(A).$$ 
Since $F$ is a finite set, there exists a finite rank projection $Q\in \B(\ell^2)$ with $q=1\otimes Q$ such that $a = qaq$ for all $a\in F$. Then this inequality exactly means that
	$$\sum_{a,b\in F}\|E_{B_\infty}(q \beta_{g}^\infty(b^*)(\omega_{g}^* u \otimes Q)  a q)\|_{\varphi}^2\geq\delta,  \quad \text{for all } g\in G,\ u\in\mathcal{U}(A).$$
Then this is a condition for inclusions $A\otimes \C Q, B\ovt Q\B(\ell^2)Q\subset M\ovt Q\B(\ell^2)Q$. Since $Q$ is finite rank, if there exists a net $(u_i)_i$ and $g_i$ in item (2) of Theorem \ref{thm-intertwining2}, then we get a contradiction to the above inequality. This shows that Theorem \ref{thm-intertwining2}(2) holds. 

Once we get this condition, we can follow the proof of Theorem \ref{thm-intertwining2} $(2)\Rightarrow(3)$. So we define $d_0= \sum_{b\in F}b e_{\widetilde{B}} b^*$ for some finite subset $F\subset M1_B$, and the closed convex hull
	$$ \mathcal K(d_0,\gamma)=  \overline{{\rm co}}^{\rm weak} \left \{ u^* \widehat{\alpha}_g(d_0) u \mid u \in \mathcal U(A),\ g\in G\right \} \subset (1_A \langle M, \widetilde B\rangle 1_A)^+$$
does not contain $0$. Then find a nonzero $d\in  \mathcal K(d_0,\gamma)^{\gamma}$ such that $d=dJ1_BJ$ and $\widehat{E}_{\widetilde{B}}(d)<\infty$. Thus, to finish the proof, we have only to show that $\mathrm{ctr}_{\langle M,\widetilde{B}\rangle}(d)<\infty$. 
To see this, we can use the same computation as in the proof of (3)$\Rightarrow$(4) of  \cite[Theorem 3.2]{Is19}, and we get
\begin{align*}
	\|\mathrm{ctr}(d J1_BJ)\|_\infty \leq  \left\|\mathrm{ctr}_B\circ E_B\left(\sum_{b\in F}  1_Bb^*b1_B\right)\right\|_\infty<\infty.
\end{align*}
This shows item (2) holds. 

	$(2)\Rightarrow(1)$ This is really the same as the proof of the last part of (3)$\Rightarrow$(4) of \cite[Theorem 3.2]{Is19}. So we omit it.
\end{proof}

\section{Malleable deformation and spectral gap rigidity}

\subsection{Lemmas on ultraproducts}

In this subsection, we collect some necessary lemmas on ultraproduct von Neumann algebras. For the ultraproduct of von Neumann algebras, we refer to \cite{Oc85,AH12}. In this paper, we denote by $M^\omega$ the \emph{Ocneanu ultraproduct} of a $\sigma$-finite von Neumann algebra $M$. Let $E_\omega$ be the canonical conditional expectation associated with the inclusion $M \subset M^\omega$, and define $\varphi^\omega := \varphi \circ E_\omega$ for $\varphi \in M_\ast$.

The following lemma is essentially shown in \cite{BMO19,IM19}.

\begin{lem}\label{lem-convex-expectation-ultraproduct}
Let $M$ be a von Neumann algebra, $\psi \in M_\ast$ a faithful normal state, and $p \in M_\psi$ a projection. Let $A \subset pMp$ be a von Neumann algebra with expectation $E_A$ such that $\psi \circ E_A|_{pMp} = \psi$. Suppose there exist an automorphism $\theta \in \Aut(M)$ and a constant $\delta > 0$ such that
\[
\sup_{u \in \mathcal U(A^\omega_{\psi^\omega})} \| \theta^\omega(u) - u \|_{\psi^\omega} < \delta.
\]
Then, any $\psi$-preserving conditional expectation from $M$ onto $(A^\omega)_{\psi^\omega}' \cap M$ can be approximated pointwise in the $\sigma$-weak topology by convex combinations of $\Ad(u)$ with $u \in \mathcal U(A)$ satisfying $\| \theta(u) - u \|_{\psi} < \delta$.
\end{lem}

\begin{proof}
Choose $\delta' < \delta$ such that 
\[
\sup_{u \in \mathcal U(A^\omega_{\psi^\omega})} \| \theta^\omega(u) - u \|_{\psi^\omega} < \delta'.
\]
For each $\varepsilon > 0$, set
\[
\mathcal U_\varepsilon := \{ u \in \mathcal U(A) \mid \| \psi u - u \psi \| < \varepsilon,\ \| \theta(u) - u \|_{\psi} \leq \delta' \}.
\]
Then the same argument as in the proof of \cite[Proposition 4.2]{IM19} applies. We omit the proof.
\end{proof}

\begin{lem}\label{lem-uniform-convex-ultraproduct}
Let $M$ be a von Neumann algebra, $\psi \in M_\ast$ a faithful normal state, and $p \in M_\psi$ a projection. Let $A \subset pMp$ be a von Neumann algebra with expectation $E_A$ such that $\psi \circ E_A|_{pMp} = \psi$. Let $\theta \in \Aut(M)$ and $0 < \alpha < 1/2$, and define
\[
\iota\colon M \ni a \mapsto 
\begin{bmatrix}
a & 0 \\
0 & \theta(a)
\end{bmatrix}
\in M \otimes \M_2 =: M_2,\quad 
\psi_2 := \left(\frac{1}{2}\, \psi\right) \oplus \left( \alpha \psi \circ \theta^{-1} \right) \in (M_2)_\ast.
\]
Set $q := \iota(p)$ and $B := \iota(A)$. Suppose that
\[
\sup_{u \in \mathcal U(A^\omega_{\psi^\omega})} \| \theta^\omega(u) - u \|_{2, \psi^\omega}^2 < 1/4,\quad \| (\psi \circ \theta - \psi)|_A \| < 1/2,
\]
and that
\[
(B^\omega)_{\psi_2^\omega}' \cap q M_2 q = B' \cap q M_2 q.
\]
Then there exists $0 \neq w \in \theta(p) M p$ such that $\theta(a) w = w a$ for all $a \in A$.
\end{lem}

\begin{proof}
Let $E_1\colon q M_2 q \to B' \cap q M_2 q$ be the $\psi_2$-preserving conditional expectation, and define
\[
\mathcal U_\theta := \{ a \in \mathcal U(A) \mid \| \theta(a) - a \|_{2, \psi}^2 < 1/2 \}.
\]
We claim that $E_1$ can be approximated by convex combinations of $\Ad(\iota(a))$ with $a \in \mathcal U_\theta$ in the point $\sigma$-weak topology .

Let $E_2$ be the $\psi_2$-preserving conditional expectation from $M_2$ onto $(B^\omega)_{\psi_2^\omega}' \cap M_2$. We have $E_1 = E_2$ by assumption, so it suffices to approximate $E_2$. Note that $\psi \circ \iota = (1/2 + \alpha) \psi$. Using the first inequality in the assumption and $1/2 + \alpha < 1$, we have
\[
\sup_{u \in \mathcal U(B^\omega_{\psi_2^\omega})} \| (\theta \otimes \id)^\omega(u) - u \|_{2, \psi_2^\omega}^2 < 1/4.
\]
By Lemma~\ref{lem-convex-expectation-ultraproduct}, $E_2$ is approximated by convex combinations of unitaries $u \in \mathcal U(B)$ satisfying $\| (\theta \otimes \id)(u) - u \|_{2, \psi_2}^2 < 1/4$. Since $B = \iota(A)$ and $\psi \circ \iota = (1/2 + \alpha) \psi \geq (1/2) \psi$, with a computation similar to above, the claim follows.

Let
$d_0 := 
\begin{bmatrix}
0 & 0 \\
1 & 0
\end{bmatrix} \in \widetilde{M}_2$. 
By the previous claim, it follows that
$E_1(d_0) = 
\begin{bmatrix}
0 & 0 \\
w & 0
\end{bmatrix}$
for some $w$ that is approximated in the $\sigma$-weak topology by convex combinations of $\theta(a^*) a$ with $a \in \mathcal U_\theta$. For any $a \in \mathcal U_\theta$, we compute
\begin{align*}
	\frac{1}{2}
	&\geq \| \theta(a) - a \|_{2, \psi}^2 
= \| \theta(a) \|_{2, \psi}^2 + \| a \|_{2, \psi}^2 - 2 \Re \psi(\theta(a^*) a) \\
	&\geq \| a \|_{2, \psi}^2 - \| (\psi \circ \theta - \psi)|_A \| \| a^* a \|_\infty + \| a \|_{2, \psi}^2 - 2 \Re \psi(\theta(a^*) a) \\
	&\geq 1 - \frac{1}{2} + 1 - 2 \Re \psi(\theta(a^*) a),
\end{align*}
which gives $1/2 \leq \Re \psi(\theta(a^*) a)$. Hence, $1/2 \leq \Re \psi(w)$ and $w \neq 0$. Finally, since $E_1(d_0) \in B' \cap M_2$, it commutes with $B$, so that $\theta(a) w = w a$ for all $a \in \mathcal U(A)$.
\end{proof}

\begin{lem}\label{lem-spectral-gap}
Let $M \subset \widetilde{M}$ be an inclusion of von Neumann algebras with expectation. Assume the following conditions.
\begin{itemize}
\item There exists an $M$-bimodule $\mathcal H = {}_M \mathcal H_M$ such that, as $M$-bimodules,
\[
{}_{M} L^2(\widetilde{M}) \ominus L^2(M)_{M} \prec {}_{M} \mathcal H_{M}.
\]
\item There exists a projection $p \in M$ and a von Neumann subalgebra $P \subset pMp$ with expectation such that
\[
P' \cap p \widetilde{M}^\omega p \not\subset p M^\omega p.
\]
\end{itemize}
Then, there exist a projection $q \in P' \cap p M p$ and a ucp map $\Psi \colon \rho_{\mathcal H}(M^{\mathrm{op}})' \to q M q$ such that $\Psi|_M$ is normal and $\Psi(x) = xq$ for all $x \in P$.
\end{lem}

\begin{proof}
The case where $\mathcal H_M$ is the coarse bimodule is shown in \cite[Lemma 4.1]{Ma16}. The general case follows by the same argument.
\end{proof}

\subsection*{Lemma on relative bicentralizers}

Let $N \subset M$ be an inclusion of von Neumann algebras with expectation, and let $\psi \in N_\ast$ be a faithful normal state. The \emph{relative bicentralizer} is defined by
\[
\mathrm{B}(N \subset M, \psi) := (N^\omega)_{\psi^\omega}' \cap M.
\]
The following lemma is derived from \cite{Ma23}. The case that $A'\cap M \subset A$ was already considered (see, e.g., \cite[Example~3]{Ma23}), but we need the result for arbitrary inclusions.
The proof is essentially the same as in \cite{Ma23}, with only minor modifications.

\begin{lem}\label{lem-relative-bicentralizer-tensor}
Let $A \subset M$ be an inclusion of $\sigma$-finite von Neumann algebras with expectation, and let $N_1, N_2$ be $\sigma$-finite type $\rm III_1$ factors. Set $N := N_1 \overline{\otimes} N_2$. Then for any faithful normal state $\psi$ on $A \overline{\otimes} N$, we have
\[
\mathrm{B}(A \overline{\otimes} N \subset M \overline{\otimes} N, \psi) = (A' \cap M) \otimes \mathbb{C}1_N.
\]
\end{lem}

\begin{proof}
We write $A \overline{\otimes} N = (A \overline{\otimes} N_1) \overline{\otimes} N_2$. By Lemma~\ref{lem-tensor-III_1} below, $A \overline{\otimes} N_1$ is of type $\rm III_1$. Hence, by \cite[Theorem 9.17]{Ma23}, the inclusion $A \overline{\otimes} N \subset M \overline{\otimes} N$ satisfies the relative bicentralizer property: for any faithful normal state $\psi \in (A \overline{\otimes} N)_\ast$,
\[
\mathrm{B}(A \overline{\otimes} N \subset M \overline{\otimes} N, \psi)
= \mathrm{b}(A \overline{\otimes} N \subset M \overline{\otimes} N, \psi),
\]
where the right-hand side is the \emph{algebraic bicentralizer} as in \cite[Section 6]{Ma23}. Applying Lemma~\ref{lem-tensor-III_1} and taking commutants with $A$, we have
\[
(A \overline{\otimes} N)' \cap C_\psi(M \overline{\otimes} N) = (A' \cap M) \overline{\otimes} \mathbb{C}1_N.
\]
Therefore,
\begin{align*}
\mathrm{b}(A \overline{\otimes} N \subset M \overline{\otimes} N, \psi)
&= \left[ ((A \overline{\otimes} N)' \cap C_\psi(M \overline{\otimes} N)) \vee L_\psi \mathbb{R} \right] \cap (M \overline{\otimes} N) \\
&= \left[ ((A' \cap M) \overline{\otimes} \mathbb{C}1_N) \vee L_\psi \mathbb{R} \right] \cap (M \overline{\otimes} N) \\
&= [ C_\psi(A' \cap M) \overline{\otimes} \mathbb{C}1_N ] \cap (M \overline{\otimes} N) \\
&= (A' \cap M) \overline{\otimes} \mathbb{C}1_N.
\end{align*}
This proves the claim.
\end{proof}

\begin{lem}\label{lem-tensor-III_1}
Let $M$ be a von Neumann algebra and $N$ a type $\rm III_1$ factor. Then for any faithful normal semifinite weight $\psi$ on $N \overline{\otimes} M$, we have
\[
(N \overline{\otimes} \mathbb{C}1_M)' \cap C_\psi(N \overline{\otimes} M) = \mathbb{C}1_N \overline{\otimes} M.
\]
In particular, $N \overline{\otimes} M$ is a type $\rm III_1$ von Neumann algebra, i.e.,
\[
(N \overline{\otimes} M)' \cap C_\psi(N \overline{\otimes} M) = \mathcal{Z}(N \overline{\otimes} M).
\]
\end{lem}
\begin{proof}

The following proof is given by Marrakchi, which is much simpler than the previous proof in the first draft of the article. 

The inclusion $N \overline{\otimes} M \subset C_\psi(N \overline{\otimes} M)$ is independent of the choice of $\psi$, so it suffices to prove the lemma for a fixed $\psi$. Let $\psi_N$ be a dominant weight on $N$ and $\psi_M$ a faithful normal semifinite weight on $M$, and let $\psi := \psi_N \otimes \psi_M$. Consider dual actions $\theta^M\colon \R\actson C_{\varphi_M}(M)$, $\theta^N\colon \R\actson C_{\varphi_N}(N)$ and define 
	$$ \theta^{(2)}:=\theta^N\otimes \theta^M\colon  \R\times \R\actson C_{\varphi_N}(N)\ovt C_{\varphi_M}(M).$$
Then observe that $C_\psi(N\ovt M)$ is canonically identified with the fixed point algebra of $\theta^{(2)}|_{H}$, where $H:=\{ (t,-t)\in \R \mid t\in \R\}$. Since $N\ovt \C 1_M$ is fixed by $\theta^{(2)}$ and since $N'\cap C_{\psi_N}(N)=\C 1_N$ (because $\psi_N$ is a dominant weight), we have 
	$$ (N \overline{\otimes} \mathbb{C}1_M)' \cap C_\psi(N \overline{\otimes} M) = [ \C 1_N \ovt C_{\varphi_M}(M) ]^{\theta^{(2)}|_H} = \C 1_N \ovt M. $$
This completes the proof.
\end{proof}

\subsection{Spectral gap rigidity}
\label{Spectral gap rigidity}

Let $M \subset \widetilde{M}$ be an inclusion of von Neumann algebras with expectation $E_M$. 
A \emph{(symmetric) malleable deformation} \cite{Po03} is a pair consisting of an $\R$-action $\theta\colon \R \curvearrowright \widetilde{M}$ and an automorphism $\beta \in \Aut(\widetilde{M})$ satisfying
\[
\theta_t \circ \beta = \beta \circ \theta_{-t}, \quad \beta^2 = \id, \quad \beta|_M = \id_M.
\]
As in \cite{Ma16}, we do not assume the existence of a state or weight that is invariant under $\theta_t$. This formulation is more convenient in the context of type $\rm III$ von Neumann algebras. For this reason, the main reference for this section is \cite{Ma16}.

Our goal is to prove the following theorem. The main idea of the proof first appeared in \cite{Po06a}. Our version generalizes \cite[Theorem 4.2]{Ma16} by removing the assumption that the subalgebra $Q$ is finite.

\begin{thm}\label{thm-malleable-nonfinite}
Let $(\theta_t,\beta)$ be a malleable deformation of the inclusion $M \subset^{E_M} \widetilde{M}$. Let $p \in M$ be a projection and $Q \subset pMp$ a von Neumann subalgebra with expectation. Put $P := Q' \cap pMp$. Take $\sigma$-finite type $\rm III_1$ factors $N_1, N_2$ and set $N := N_1 \overline{\otimes} N_2$. Then, one of the following conditions holds:
\begin{enumerate}
    \item $Q \preceq_{\widetilde{M}} \theta_1(Q)$; or

    \item $(P \overline{\otimes} \mathbb{C}1_N)' \cap p(\widetilde{M} \overline{\otimes} N)^\omega p \not\subset p(M \overline{\otimes} N)^\omega p$.
\end{enumerate}
\end{thm}

We fix the setting of the theorem and prove some lemmas. Let $\psi \in M_\ast$ be a faithful normal state such that $p \in M_\psi$ and $\psi \circ E_Q|_{pMp} = \psi$.

\begin{lem}\label{lem-spectral-gap2}
In the setting of Theorem~\ref{thm-malleable-nonfinite}, assume that $P' \cap p\widetilde{M}^\omega p \subset pM^\omega p$.
Then for any sequence $t_n \in \R$ with $t_n \to 0$ and any $(a_n)_\omega \in Q^\omega$, we have
$(\theta_{t_n}(a_n))_\omega = (a_n)_\omega$ in $\widetilde{M}^\omega$.
In particular,
\[
\lim_{t \to 0} \sup_{a \in ((Q^\omega)_{\psi^\omega})_1} \| \theta_t^\omega(a) - a \|_{\psi^\omega} = 0.
\]
\end{lem}

\begin{proof}
The first part is proven in Step 1 of the proof of \cite[Theorem 4.2]{Ma16}, so we focus on the second part. Suppose the conclusion does not hold. Then there exists a sequence $t_n \to 0$ and a bounded sequence $(a_n)_n \subset (Q^\omega)_{\psi^\omega}$ such that 
$\lim_n \| \theta^\omega_{t_n}(a_n) - a_n \|_{2,\psi} \not = 0$.
Passing to a subsequence, we may assume there exists $\delta > 0$ such that for all $n$,
\[
\| \theta^\omega_{t_n}(a_n) - a_n \|_{2,\psi} \geq \delta > 0.
\]
Write $a_n = (a_{n,k})_k$ with $\|a_{n,k}\|_\infty \leq 1$. Since $a_n \in (Q^\omega)_{\psi^\omega}$, we can construct a sequence of integers $k_n$ such that
\[
\|a_{n,k_n} \psi - \psi a_{n,k_n}\| < \frac{1}{n},\quad \| \theta_{t_n}(a_{n,k_n}) - a_{n,k_n} \|_{2,\psi} \geq \frac{\delta}{2}.
\]
Then the bounded sequence $(a_{n,k_n})_n$ defines an element in $(Q^\omega)_{\psi^\omega}$, while
\[
\| (\theta_{t_n}(a_{n,k_n}))_\omega - (a_{n,k_n})_\omega \| \geq \frac{\delta}{2},
\]
which contradicts the first part. Hence the second part follows.
\end{proof}

We define the sequence by $t_1 = 0$ and $t_n = 2^{-n}$ for $n \geq 2$. Let $\{e_{i,j}\}_{i,j \in \mathbb{N}}$ be the standard matrix units in $\mathbb{B}(\ell^2)$. Consider the embedding
\[
\iota\colon \widetilde{M} \to \widetilde{M}_\infty := \widetilde{M} \overline{\otimes} \mathbb{B}(\ell^2);\quad a \mapsto \sum_{n \in \mathbb{N}} \theta_{t_n}(a) \otimes e_{n,n}.
\]
Define a faithful normal state $\psi_\infty \in (\widetilde{M}_\infty)_\ast$ by
\[
\psi_\infty(a \otimes e_{i,j}) := \delta_{i,j} \frac{1}{2^i} \psi \circ \theta_{t_i}^{-1}(a),\quad a \in \widetilde{M},\ i,j \geq 1,
\]
so that $\psi_\infty \circ \iota = \psi$. Let $q := \iota(p)$ and $Q_\infty := \iota(Q)$. Then the inclusion $Q_\infty \subset q \widetilde{M}_\infty q$ admits a conditional expectation and is invariant under $\sigma^{\psi_\infty}$.

\begin{lem}\label{lem-malleable-nonfinite}
	In this setting, assume that
\[
\mathrm{B}(Q_\infty \subset q \widetilde{M}_\infty q, \psi_\infty) = Q_\infty' \cap q \widetilde{M}_\infty q.
\]
Then one of the following conditions holds:
\begin{enumerate}
    \item there exists a nonzero element $w \in \theta_1(p) \widetilde{M} p$ such that
    \[
    \theta_1(x) w = w x,\quad \text{for all } x \in Q.
    \]
    \item $P' \cap p \widetilde{M}^\omega p \not\subset p M^\omega p$.
\end{enumerate}
\end{lem}
\begin{proof}

Assume (2) does not hold. By Lemma~\ref{lem-spectral-gap2}, for sufficiently large $n$ we have
\[
\sup_{u \in \mathcal{U}((Q^\omega)_{\psi^\omega})} \| \theta_{t_n}(u) - u \|_{2,\psi}^2 < \frac{1}{4},\quad \| \psi \circ \theta_{t_n} - \psi \| < \frac{1}{2}.
\]
Fix such $n$ and define a projection $q_n := 1 \otimes e_{1,1} + 1 \otimes e_{n,n} \in \widetilde{M}_\infty$. Let $\iota_n$ be the composition of $\iota$ with compression by $q_n$:
\[
\iota_n\colon \widetilde{M} \to q_n(\widetilde{M} \overline{\otimes} \mathbb{B}(\ell^2))q_n \simeq \widetilde{M} \overline{\otimes} \mathbb{M}_2.
\]
This embedding fits the setting of Lemma~\ref{lem-uniform-convex-ultraproduct}. With the assumption of the relative bicentralizer, the hypothesis of Lemma~\ref{lem-uniform-convex-ultraproduct} is satisfied (with $\alpha = 1/2^n$ and $\theta = \theta_{t_n}$). There exists $w \in \theta_{t_n}(q) \widetilde{M} q$ such that $\theta_{t_n}(x) w = w x$ for all $x \in Q$. Since $P' \cap p \widetilde{M} p \subset p M p$, Step 3 of the proof of \cite[Theorem 4.2]{Ma16} implies that $\theta_{t_n}$ can be replaced with $\theta_1$. We get item (1) and this proves the lemma.
\end{proof}

\begin{proof}[{\bf Proof of Theorem~\ref{thm-malleable-nonfinite}}]
Let
\[
\mathcal{M} := M \overline{\otimes} N \subset \widetilde{M} \overline{\otimes} N =: \widetilde{\mathcal{M}},\quad
\mathcal{Q} := Q \overline{\otimes} N \subset p\mathcal{M}p,\quad
\mathcal{P} := P \overline{\otimes} \mathbb{C}1_N.
\]
Extend the malleable deformation as $\Theta_t := \theta_t \otimes \id_N$ acting trivially on $N$. Define the embedding
\[
\iota_\Theta \colon \widetilde{\mathcal{M}} \to \widetilde{\mathcal{M}}_\infty := \widetilde{\mathcal{M}} \overline{\otimes} \mathbb{B}(\ell^2)
\]
similarly as above. Since $\theta_t$ acts trivially on $N$, we may write
\[
\iota_\Theta = \id_N \otimes \iota_\theta \colon N \overline{\otimes} \widetilde{M} \to N \overline{\otimes} \widetilde{M}_\infty.
\]
Therefore, Lemma~\ref{lem-relative-bicentralizer-tensor} applies to the inclusion $\iota_\Theta(\mathcal{Q}) \subset q \widetilde{\mathcal{M}}_\infty q$, yielding
\[
\mathrm{B}(\mathcal{Q} \subset q \widetilde{\mathcal{M}}_\infty q, \psi) = \mathcal{Q}' \cap q \widetilde{\mathcal{M}}_\infty q
\]
for any faithful normal state $\psi$ on $\mathcal{Q}$. Hence, Lemma~\ref{lem-malleable-nonfinite} applies, and one of the following holds:
\begin{enumerate}
    \item there exists a nonzero $w \in \theta_1(p) \widetilde{\mathcal{M}} p$ such that
$\theta_1(x) w = w x$ for all  $x \in \mathcal{Q}$.

    \item $\mathcal{P}' \cap p \widetilde{\mathcal{M}}^\omega p \not\subset p \mathcal{M}^\omega p$.
\end{enumerate}
If (1) holds, then by the existence of $w$, we have $Q \overline{\otimes} N \preceq_{\widetilde{M} \overline{\otimes} N} \theta_1(Q) \overline{\otimes} N$. By \cite[Lemma 4.3]{Is17}, this implies $Q \preceq_{\widetilde{M}} \theta_1(Q)$. This finishes the prove.
\end{proof}

\subsection{Proof of Theorem \ref{thmD}}

We prove Theorem \ref{thmD}. We indeed prove the following more general theorem, which removes the finiteness assumption on subalgebras from \cite[Section 5]{Ma16} and \cite[Theorem 7.1]{IM19}. The ideas of the proof first appeared in \cite{Po03,Po06a,CI08}.

For inclusion of von Neumann algebras $A,B\subset M$ with $A$ possibly non-unital, recall that \emph{$A$ is amenable relative to $B$ in $M$} and write $A\lessdot_MB$ if there exists a condtional expectation from $1_A\langle M,B\rangle1_A$ onto $A$ that is normal on $1_AM1_A$.

\begin{thm}\label{thm-Bernoulli}
	Let $\Gamma \actson I$ be an action of a discrete group on a set with finite stabilizers, and let $(B_0,\varphi_0)$ be a von Neumann algebra with a faithful normal state. Consider the generalized Bernoulli action
	$$ \alpha\colon \Gamma \actson \bigotimes_{I} (B_0,\varphi_0) =: (B,\varphi_B). $$
	Let $\Gamma \actson N$ be an action on a $\sigma$-finite von Neumann algebra, and consider the crossed product $M = (N\overline{\otimes} B)\rtimes \Gamma$ with respect to the diagonal action. Then for every projection $p \in M$ and every von Neumann subalgebra $A \subset pMp$ with expectation, we have either
	\begin{enumerate}
		\item $A \preceq_M N \rtimes \Gamma$, or
		\item $A'\cap pMp \lessdot_M N \overline{\otimes} B$.
	\end{enumerate}
\end{thm}

\begin{proof}
	As in \cite{Ma16,IM19}, we consider the malleable deformation $(\theta_t, \beta)$ and the inclusion $M \subset^{E_M} \widetilde{M}$ defined in \cite{CI08}. Let $R_\infty$ be an amenable type $\mathrm{III}_1$ factor, and extend the malleable deformation to $R_\infty \overline{\otimes} M$ so that it acts trivially on $R_\infty$.

	Next, take the maximal projection $z \in \mathcal{Z}(P)$ such that $Pz \lessdot_M N \overline{\otimes} B$ (cf.\ \cite[Lemma 3.4]{HI17}). Suppose that (1) and (2) do not hold, and we deduce a contradiction. We have that $A \not\preceq_{M} N \rtimes_{\alpha} \Gamma$ and $z \neq p$. Replacing $p - z$ by $p$, we may assume that $P$ has no direct summand that is amenable relative to $N \overline{\otimes} B$.

	Noting the isomorphism $R_\infty \simeq R_\infty \overline{\otimes} R_\infty$, we can apply Theorem~\ref{thm-malleable-nonfinite}. Set $\mathcal{P} := \C 1_{R_\infty} \overline{\otimes} (A'\cap pMp)$, then one of the following holds:
	\begin{enumerate}
		\item[$\mathrm{(a)}$] $A \preceq_{\widetilde{M}} \theta_1(A)$;
		\item[$\mathrm{(b)}$] $\mathcal{P}' \cap p(R_\infty \overline{\otimes} \widetilde{M})^\omega p \not\subset p(R_\infty \overline{\otimes} M)^\omega p$.
	\end{enumerate}

	\bigskip
	\noindent
	\textbf{Case (a).} Using Lemma~\ref{lem-uniform-Bernoulli} below, there exists a finite subset $F$ such that $A \preceq_M N \overline{\otimes} B^F$. Since $A \not\preceq_M N$ by assumption, it follows from \cite[Lemma 7.2]{IM19} that there exists a projection $q \in A' \cap pMp$ such that $\mathcal{N}_{qMq}(Aq)'' \preceq_M N \overline{\otimes} B$. By \cite[Proposition 2.2]{IM19}, there exists a central projection $w \in \mathcal{Z}(\mathcal{N}_{qMq}(Aq)'')$ such that $\mathcal{N}_{qMq}(Aq)'' w \lessdot_M N \overline{\otimes} B$. This implies $qPqw \lessdot_M N \overline{\otimes} B$. By \cite[Lemma 3.4(iv)]{HI17}, this contradicts the assumption that $P$ has no direct summand that is amenable relative to $N \overline{\otimes} B$.

	\bigskip
	\noindent
	\textbf{Case (b).} Let $\mathcal{N} := R_\infty \overline{\otimes} N$, $\mathcal{M} := R_\infty \overline{\otimes} M \subset \widetilde{\mathcal{M}} := R_\infty \overline{\otimes} \widetilde{M}$. As proved in \cite{CI08} and \cite[Theorem 5.2]{Ma16}, and as explained in the proof of \cite[Theorem 7.1]{IM19} the following weak containment holds:
	$$ L^2(\widetilde{\mathcal{M}}) \ominus L^2(\mathcal{M}) \prec \bigoplus_{x \in J} L^2(\mathcal{M}) \otimes_{\mathcal{N} \overline{\otimes} B_0^{\Delta_x}} L^2(\mathcal{M}) $$
as $\mathcal{M}$-bimodules. On the right-hand side, the bimodule has a left diagonal action of $\langle \mathcal{M}, \mathcal{N} \overline{\otimes} B \rangle$ that commutes with the right action. Combining this with (b), we can apply Lemma~\ref{lem-spectral-gap}, obtaining a projection $q \in \mathcal{P}' \cap p\mathcal{M}p$ and a ucp map
	$$ \Psi\colon \langle \mathcal{M}, \mathcal{N} \overline{\otimes} B \rangle \to q \mathcal{M} q $$
	such that $\Psi|_{\mathcal{M}}$ is normal, and for any $x \in \mathcal{P}$, we have $\Psi(x) = xq$. Since $\mathcal{P}q \subset q\mathcal{M}q$ is with expectation, we may replace the image with $\mathcal{P}q$. Let $z \in \mathcal{Z}(\mathcal{P}) = \C 1_{R_\infty} \overline{\otimes} \mathcal{Z}(A' \cap pMp)$ be the central support projection of $q$ in $\mathcal{P}'$. Then the map $xq \mapsto xz$ gives an isomorphism from $\mathcal{P}q$ to $\mathcal{P}z$, and hence we obtain a map from $\langle \mathcal{M}, \mathcal{N} \overline{\otimes} B \rangle$ to $\mathcal{P}z$. Finally, noting $\langle \mathcal{M}, \mathcal{N} \overline{\otimes} B \rangle = R_\infty \overline{\otimes} \langle M, N \overline{\otimes} B \rangle$, we restrict the domain to $\langle M, N \overline{\otimes} B \rangle$, and get that $Pz$ is amenable relative to $N\ovt B$ in $M$. This is a contradiction.
\end{proof}

\begin{cor}\label{cor-Bernoulli}
	Keep the setting as in Theorem \ref{thm-Bernoulli} and assume that $B,N$ are amenable. If $N\rtimes \Gamma$ is solid relative to $N$, then $M = (N\ovt B)\rtimes \Gamma$ is solid relative to $N$. 
\end{cor}
\begin{proof}
	Let $A\subset pMp$ be von Neumann subalgebra with expectation where $p\in M$ is a projection. Assume that $A\not\preceq_MN$ and that $A'\cap pMp$ is not amenable. We will deduce a contradiction. Let $z\in \mathcal Z(A'\cap pMp)$ be the maximum projection such that $(A'\cap pMp)z$ is amenable. Since $z\neq p$, up to replacing $p-z$ by $p$, we may assume that $A'\cap pMp$ has no amenable direct summand.

By Theorem \ref{thm-Bernoulli}, we have $A\preceq_M N\rtimes \Gamma$. 
By \cite[Lemma 2.6(1)]{Is19}, with $B:=N\rtimes \Gamma$, take projections $e\in A$, $f\in B$, a partial isometry $v\in eMf$, and a unital normal $\ast$-homomorphism $\theta\colon eAe\to fBf$ such that $eAe\subset fBf$ is with expectation and that $v\theta(a) = av$ for all $a \in eAe$. We may assume that the support projection of $E_B(v^*v)$ is $f$. We then follow the proof of \cite[Lemma 7.3(2)]{IM19}, and get $v^*v\in \theta(eAe)'\cap fMf \subset fBf$. Up to replacing $f$ with $v^*v$, putting $q:=vv^*\in (eAe)'\cap eMe$, we get 
	$$ v^* \mathcal N_{qMq}(eAeq)'' v \subset fMf. $$
Then $\Ad(v^*)$ restricts to a unital embedding 
	$$\widetilde{\theta}\colon (eAeq)\vee ((eAeq)'\cap qMq) \to f(N\rtimes \Gamma)f.$$
Observe that  $\widetilde{\theta}(eAeq)\not\preceq_{N\rtimes \Gamma} N$, because $v\in M$ and $A\not\preceq_MN$. By the relative solidity of $N \subset N\rtimes \Gamma$, we have that $\widetilde{\theta}(eAeq)'\cap f(N\rtimes \Gamma)f$ is amenable. This implies that $\widetilde{\theta}((eAeq)'\cap qMq)$ is amenable. Since $(eAeq)'\cap qMq = q( A'\cap pMp )q$, this further implies that $A'\cap pMp$ has no amenable direct summand. This is a contradiction.
\end{proof}

The following lemma was used in the proof of Theorem \ref{thm-Bernoulli}. The proof relies on our characterization of Popa's intertwining condition.

\begin{lem}\label{lem-uniform-Bernoulli}
In the setting of the proof of Theorem~\ref{thm-Bernoulli}, suppose that $A \preceq_{\widetilde{M}} \theta_1(M)$. Then one of the following conditions holds:
\begin{itemize}
  \item $A \preceq_M N \rtimes \Gamma$; or
  \item There exists a finite subset $F \subset I$ such that $A \preceq_M N \overline{\otimes} B_0^F$.
\end{itemize}
\end{lem}
\begin{proof}
By \cite[Remark 4.2(4)]{HI15}, we may assume that $A$ is a direct sum of a finite von Neumann algebra and a type $\mathrm{III}$ algebra. 
Let $\varphi_N$ be a faithful normal state on $N$ and define a state $\varphi := \varphi_N \otimes \varphi_B$ on $N \overline{\otimes} B$. Extend it canonically to $M$, and extend it on $\widetilde{M}$ by $E_M$. Let $E_F$ and $E_{N \rtimes \Gamma}$ be the $\varphi$-preserving conditional expectations onto $N \overline{\otimes} B_0^F$ and $N \rtimes \Gamma$, respectively. Let $e_F, e_{N \rtimes \Gamma}, e_M \in \mathbb{B}(L^2(\widetilde{M}))$ be the corresponding Jones projections. Note that $e_Me_F=e_F$ and $e_Me_{N \rtimes \Gamma}=e_{N \rtimes \Gamma} $.

Assume that both conditions do not hold. Then by Remark \ref{rem-intertwining1}, there exists $\Psi \in \mathrm{DSG}(A \subset p \mathbb{B}(L^2(\widetilde{M})) p)$ such that $\Psi|_{\widetilde{M}}$ is normal and
\[
\Psi(a e_{N \rtimes \Gamma} b^*) = 0 = \Psi(a e_F b^*)
\]
for all $a, b \in M$ and all finite subsets $F \subset I$. It suffices to show that
\[
\Psi(a e_{\theta_1(M)} b^*) = 0, \quad \text{for all } a, b \in \widetilde{M}.
\]
We will use Lemma~\ref{lem-intertwining1}. For this, we define the following sets:
\begin{align*}
  X_0 &:= \left\{ x = \bigotimes_{s \in K_x} x_s \in \widetilde{B} \mid x_s \in B_0 \ast L\mathbb{Z},\ \text{each } x_s \text{ is a reduced word} \right\}, \\\\
  X &:= \left\{ x_0z \mid x_0 \in X_0,\ z\in N\rtimes\Gamma \right\}, \\\\
  Y_0 &:= \left\{ y = \bigotimes_{s \in K_y} y_s \in X_0 \mid \text{each } y_s \text{ is a reduced word in } (B_0)_{\mathrm{a}} \ast L\mathbb{Z} \right\}, \\\\
  Y &:= \left\{ y_0 w \mid  y_0 \in Y_0,\ w\in (N\rtimes \Gamma)_a \right\},
\end{align*}
where $(B_0)_{\mathrm{a}}$ and $(N\rtimes \Gamma)_{\mathrm{a}}$ denote the sets of $\varphi$-analytic elements, and where $K_x, K_y \subset I$ are finite subsets depending on $x,y$ respectively.  
The span of $X$ is dense in $\widetilde{M}$, and the span of $Y$ is $L^2$-dense in $L^2(\widetilde{M}, \varphi)$. Therefore, by Lemma~\ref{lem-intertwining1}(2), it suffices to show that for all $x \in X$ and $y \in Y$,
\[
(*) \quad \langle \Psi(x e_{\theta_1(M)} x^*) \Lambda_\varphi(y), \Lambda_\varphi(y) \rangle = 0.
\]
Fix $x \in X$ and $y \in Y$, and write $x = x_0z$, $y = y_0w$. Let $(f_\mu)$ be a net converging pointwise $\sigma$-weakly to $\Psi$, where each $f_\mu$ has the form
\[
f_\mu = \sum_{i=1}^{n_\mu} t_{\mu,i} \operatorname{Ad}(u_{\mu,i}^*), \quad \text{with } u_{\mu,i} \in \mathcal{U}(A),\ t_{\mu,i} > 0,\ \sum_{i=1}^{n_\mu} t_{\mu,i} = 1.
\]
By Lemma\ref{lem-intertwining1}(1), we compute that
\begin{align*}
	 \langle f_\mu( x e_{\theta_1(M)} x^* )\Lambda_\varphi(y),\Lambda_\varphi(y) \rangle
	&= \sum_{i=1}^{n_\mu} t_{\mu,i} \| E_{\theta_1(M)} (z^* x_0^*u_{\mu,i} y_0w  )\|_{\varphi}^2\\
	&\leq \| z^* \|_\infty^2 \| E_{\theta_1(M)} (x_0^*u_{\mu,i} y_0)  \|_{\varphi}^2\|\sigma^\varphi_{\ri/2}(w) \|_\infty^2\\
	&= \| z^* \|_\infty^2 \|\sigma^\varphi_{\ri/2}(w) \|_\infty^2  \langle f_\mu( x_0 e_{\theta_1(M)} x_0^* )\Lambda_\varphi(y_0),\Lambda_\varphi(y_0) \rangle.
\end{align*}
Since $f_\lambda$ converges to $\Psi$, we may assume that $z = w = 1$ in order to prove $(*)$. Thus, we may write $x = \bigotimes_{s \in K_x} x_s \in X_0$ and $y = \bigotimes_{s \in K_y} y_s \in Y_0$, where each $x_s, y_s$ is a reduced word in the free product $B_0 \ast L\mathbb{Z}$. We now divide the proof into two cases.

\bigskip
\noindent
{\bf Case 1: All reduced words $x_s, y_s$ belong to $B_0 \theta_1(B_0)$.}

In this case, each reduced word can be written in the form $x_s = x_{1,s} \theta_1(x_{2,s})$ and $y_s = y_{1,s} \theta_1(y_{2,s})$. Hence, we can decompose
\[
x = \bigotimes_{s \in K_x} x_{1,s} \theta_1(x_{2,s}) = \left( \bigotimes_{s \in K_x} x_{1,s} \right) \left( \bigotimes_{s \in K_x} \theta_1(x_{2,s}) \right) =: x_1 \theta_1(x_2).
\]
Similarly, write $y = y_1 \theta_1(y_2)$. Then, by Lemma~\ref{lem-intertwining1}(1), for each $f_\mu$, we have
\begin{align*}
  \langle f_\mu( x e_{\theta_1(M)} x^* )\Lambda_\varphi(y),\Lambda_\varphi(y) \rangle
  &= \sum_{i=1}^{n_\mu} t_{\mu,i} \left\| E_{\theta_1(M)}( \theta_1(x_2^*) x_1^* u_{\mu,i} y_1 \theta_1(y_2) ) \right\|_{\varphi}^2 \\
  &\leq \| \theta_1(x_2^*) \|_\infty^2 \| \sigma^\varphi_{\ri/2}(\theta_1(y_2)) \|_\infty^2 \sum_{i=1}^{n_\mu} t_{\mu,i} \left\| E_{\theta_1(M)}( x_1^* u_{\mu,i} y_1 ) \right\|_{\varphi}^2.
\end{align*}
Since $x_1^* u_{\mu,i} y_1 \in M$, and since $M$ and $\theta_1(M)$ are freely independent over $N \rtimes \Gamma$, we may replace $E_{\theta_1(M)}$ with $E_{N \rtimes \Gamma}$. Taking the limit $f_\mu \to \Psi$, we obtain
\[
\langle \Psi( x e_{\theta_1(M)} x^* )\Lambda_\varphi(y),\Lambda_\varphi(y) \rangle
\leq \| \theta_1(x_2^*) \|_\infty^2 \| \sigma^\varphi_{\ri/2}(\theta_1(y_2)) \|_\infty^2 \langle \Psi( x_1 e_{N \rtimes \Gamma} x_1^* )\Lambda_\varphi(y_1),\Lambda_\varphi(y_1) \rangle = 0.
\]
Therefore, $(*)$ holds in this case.

\bigskip
\noindent
{\bf Case 2: Some $x_s$ or $y_s$ does not belong to $\theta_1(B_0) B_0$.}

In this case, by consulting the proof of \cite[Lemma 5.1]{Ma16}, we know that for any $a \in M$, the following holds:
\begin{align*}
	E_{\theta_1(M)}(x^* ay)
	&= \sum_{\gamma \in \Gamma,\  \gamma^{-1} K_x\cap K_y\neq \emptyset}\lambda_{\gamma}^{\Gamma} E_{\theta_1(N\ovt B)}(\alpha_{\gamma^{-1}}(x^*) E_{ \gamma^{-1} K_x\cap K_y}( \lambda_\gamma^* a) y) .
\end{align*}
By the assumption of finite stabilizers, the right-hand side is a finite sum. We have
\begin{align*}
	\| E_{\theta_1(M)} (x^*u_{\mu,i} y)\|_{\varphi}^2
	&= \sum_{\gamma \in \Gamma,\ \gamma^{-1}K_x \cap K_y\neq \emptyset} \left\| E_{\theta_1(N\ovt B)}(\alpha_{\gamma^{-1}}(x^*) E_{ \gamma^{-1} K_x\cap K_y}( \lambda_\gamma^* u_{\mu,i}) y)\right\|_\varphi^2 \\
	&\leq \|x\|_\infty^2 \|\sigma^\varphi_{\ri/2}(y)\|_\infty^2 \sum_{\gamma \in \Gamma,\ \gamma^{-1}K_x \cap K_y\neq \emptyset} \| E_{ \gamma^{-1} K_x\cap K_y}( \lambda_\gamma^* u_{\mu,i})\|_{\varphi}^2.
\end{align*}
Therefore, by Lemma~\ref{lem-intertwining1}(1), we obtain
\begin{align*}
	\langle f_\mu( x e_{\theta_1(M)} x^* )\Lambda_\varphi(y),\Lambda_\varphi(y) \rangle
	&\leq \|x\|_\infty^2 \|\sigma^\varphi_{\ri/2}(y)\|_\infty^2 \sum_{\gamma \in \Gamma,\ K_x \cap \gamma K_y\neq \emptyset} \langle f_\mu( \lambda_\gamma e_{\gamma^{-1}K_x \cup K_y} \lambda_\gamma ^* )\Lambda_\varphi( 1 ),\Lambda_\varphi(1 ) \rangle.
\end{align*}
Now, since $f_\mu \to \Psi$, the right-hand side converges to $0$, and thus so does the left-hand side. Hence, $(*)$ also holds in this case. This completes the proof.
\end{proof}

\section{Ozawa's relative solidity theorem}

\subsection{Biexact groups and von Neumann algebras}

In this subsection, we recall biexactness of groups and von Neumann algebras. We refer the reader to \cite[Section 15]{BO08} and the recent developments in \cite{DKEP22,DP23}.

\subsection*{Biexact groups}

Let $\Gamma $ be a discrete group and $I \subset \ell^\infty(\Gamma)$ a closed ideal containing $c_0(\Gamma)$. We say that $I$ is a \textit{boundary piece} if it is globally invariant under the left and right translations on $\Gamma$. Then the \textit{small-at-infinity compactification of $\Gamma$} (relative to $I$) is the spectrum of the $C^*$-algebra 
	$$ \mathbb S_I(\Gamma):= \{f\in \ell^\infty(\Gamma)\mid f-f^t \in I,\quad \text{for all }t\in \Gamma \} ,$$
where $f^t$ is defined by $f^t(s):=f(st)$ for $s\in \Gamma$. We say that $\Gamma$ is \textit{biexact relative to $I$} \cite{BO08} if the action $\Gamma \actson \mathbb S_I(\Gamma)/I$ is topologically amenable. 
When $I=c_0(\Gamma)$, we simply say that $\Gamma$ is biexact. 
For example, all amenable groups and hyperbolic groups are biexact, see \cite[Section 15]{BO08}. 

	Let $\mathcal G$ be a family of subgroups of $\Gamma$. We say that a subset $\Omega\subset \Gamma$ is \textit{small relative to $\mathcal G$} if it is contained in a finite union of $s\Lambda t$, where $s,t\in \Gamma $ and $\Lambda \in \mathcal G$. Let $c_0(\Gamma;\mathcal G)\subset \ell^\infty(\Gamma)$ be the $C^*$-algebra generated by functions whose supports are small relative to $\mathcal G$. Equivalently, $c_0(\Gamma;\mathcal G)\subset \ell^\infty(\Gamma)$ is the closed ideal generated by $e_{s\Lambda t}$, where $s,t\in \Gamma$, $\Lambda\in \mathcal G$, and $e_{s\Lambda t}\in \ell^\infty(\Gamma)$ is the characteristic function on $s\Lambda t$. We say that $\Gamma$ is \textit{biexact relative to $\mathcal G$} if it is biexact relative to the boundary piece $c_0(\Gamma;\mathcal G)$.

\subsection*{Biexact von Neumann algebras}

Let $X\subset \B(H)$ be a norm closed subspace, and $M,N\subset \B(H)$ (unital) von Neumann subalgebras. We say that $X$ is a \textit{normal operator $M$-$N$-bimodule} if $MXN\subset X$. We simply say it is a \textit{right $N$-module} if $M=\C$, and a \textit{left $M$-module} if $N =\C$.

We say that $\varphi \in X^*$ is \textit{$M$-$N$-normal} if for each $T\in X$,
	$$ M\ni x\mapsto \varphi(xT),\quad N\ni y\mapsto \varphi(Ty) $$
are normal functionals. We denote by $X^{M\sharp N}\subset X^*$ the set of all $M$-$N$-normal functionals. When $M=\C$ (resp.\ $N=\C$), this normality is called \textit{left $N$-normal} (resp.\ \text{right $M$-normal}) and we use the notation $X^{\sharp N}$ (resp.\ $X^{M\sharp }$). Then by definition, we have $X^{M\sharp N} = X^{M\sharp } \cap X^{\sharp N}$ and it is easy to see that $X^{M\sharp N},X^{M\sharp },X^{\sharp N}$ are norm closed subspaces in $X^*$.

Let $\widetilde{M}\subset M'$ and $\widetilde{N}\subset N'$ be von Neumann subalgberas and assume that $X$ is a normal operator $M$-$N$-bimodule and $\widetilde{M}$-$\widetilde{N}$-bimodule, that is, $M\widetilde{M}XN\widetilde{N}\subset X$. In this setting, we define
	$$ X^{(M,\widetilde{M})\sharp (N,\widetilde{N})}:= X^{M\sharp } \cap  X^{\widetilde{M}\sharp } \cap X^{\sharp N}\cap X^{\sharp \widetilde{N}}  ,$$
which is a norm closed subspace in $X^*$. We call the weak topology in $X$ given by $X^{(M,\widetilde{M})\sharp (N,\widetilde{N})}$ (namely, $\sigma(X,X^{(M,\widetilde{M})\sharp (N,\widetilde{N})})$) the \textit{weak $(M,\widetilde{M})$-$(N,\widetilde{N})$-topology}. We simply say the \textit{weak $M$-$N$-topology} if $\widetilde{M}=\C =\widetilde{N}$.

Let $M \subset \B(L^2(M))$ be a von Neumann algebra with the standard representation. We say that a hereditary $C^*$-subalgebra $\mathbb X \subset \B(L^2(M))$ is a \textit{$M$-boundary piece} if $M\cap \mathrm M(\mathbb X) \subset M $ and $M'\cap \mathrm M(\mathbb X) \subset M'$ are $\sigma$-weakly dense, where $\mathrm M(\mathbb X)$ is the multiplier algebra of $\mathbb X$. Consider the weak $(M,M')$-$(M,M')$-topology on $\B(L^2(M))$, and we define $\mathbb K_{\mathbb X}^{\infty,1}(M)$ as the closure of $\mathbb X$ by this topology. Since $\mathbb X$ is a $M$-boundary piece, it is easy to see that $\mathbb K_{\mathbb X}^{\infty,1}(M)\subset \B(L^2(M))$ is a normal operator $M$-$M$-bimodule and $M'$-$M'$-bimodule. 
We define the \textit{small-at-infinity boundary of $M$} (relative to $\mathbb X$) as 
	$$ \mathbb S_{\mathbb X}(M):=\{ T\in \B(L^2(M))\mid [T,x]\in \mathbb K_{\mathbb X}^{\infty,1}(M), \ \text{for all } x\in M' \} .$$
It is an operator system and is a normal $M$-$M$-bimodule containing $M$. It is closed in the weak $M$-$M$-topology (of $\B(L^2(M))$).

\begin{df}[{\cite[Definition 6.1]{DP23}}]\upshape
	Let $M$ be a von Neumann algebra and $\mathbb X \subset \B(L^2(M))$ a $M$-boundary piece. We say that \emph{$M$ is biexact relative to $\mathbb X$} if the inclusion map $M \to \mathbb S_{\mathbb X}(M)$ is $M$-nuclear. This means that there exist nets of ccp maps $\phi_i \colon M \to \M_{n(i)}$ and $\psi_i\colon \M_{n(i)}\to \mathbb S_{\mathbb X}(M)$ such that for each $x\in M$, $\psi_i\circ \phi_i(x)$ converges to $x$ in the weak $M$-$M$-topology. 
\end{df}

When $\mathbb X = \mathbb K(L^2(M))$, we simply say that $M$ is biexact. It was proved in \cite[Theorem 6.2]{DP23} that a discrete group $\Gamma$ is biexact if and only if $L\Gamma$ is biexact.

We prove a lemma. This will be used in the proof of Theorem \ref{thm-biexact} below.

\begin{lem}\label{lem-biexact-ccp}
	Let $M\subset \B(L^2(M))$ be an inclusion of von Neumann algebras and let 
$\theta\colon \B(L^2(M))\to \B(L^2(M))$ be a completely positive map that is normal on $M$ and $M'$.

\begin{enumerate}
	\item The map $\theta$ is continuous from the weak $(M,M')$-$(M,M')$-topology to the $\sigma$-weak topology.

	\item If $\theta$ is $M'$-bimodular in the sense that $\theta(xTy) = x\theta(T)y$ for all $T\in \B(L^2(M))$ and $x,y\in M'$, and if $\theta(\mathbb X)=0$ for a $M$-boundary piece $\mathbb X\subset \B(L^2(M))$, then we have $\theta( \mathbb S_{\mathbb X}(M) ) \subset M $.

\end{enumerate}
\end{lem}
\begin{proof}
	(1) Fix $\psi\in \B(H)_\ast^+$. We have only to show that $\psi\circ \theta$ is contained in $X^{(M,M')\sharp (M,M')}$. Fix $T\in \B(L^2(M))$, then for any $a\in M$, by the Cauchy--Schwartz inequality, 
\begin{align*}
	|\psi\circ \theta (aT)|
	\leq \|a^*\|_{\psi\circ \theta} \|T\|_{\psi\circ \theta},\quad |\psi\circ \theta (Ta)|
	\leq \|T^*\|_{\psi\circ \theta}\|a\|_{\psi\circ \theta} .
\end{align*}
Since $\psi\circ \theta$ is normal on $M$, these inequalities show that $\psi\circ \theta (aT) $ and $\psi\circ \theta (Ta)$ are normal on $a\in M$, thus $\psi\circ \theta\in X^{M\sharp }\cap X^{\sharp M}$. A similar argument shows that $\psi\circ \theta\in X^{M'\sharp }\cap X^{\sharp M'}$ and we get $\psi\circ \theta \in X^{(M,M')\sharp (M,M')}$.

	(2) By the continuity proved in item (1) and by $\theta(\mathbb X)=0$, we get $\theta(\mathbb K_{\mathbb X}^{\infty,1}(M))=0$. For any $x'\in M'$ and $T\in \mathbb S_{\mathbb X}(M)$, we have
	$$x'\theta(T) = \theta(x'T)=\theta(Tx' - [T,x']) = \theta(Tx')=\theta(T)x',$$
where we used $\theta([T,x'])=0$ (because $ [T,x']\in \mathbb K_{\mathbb X}^{\infty,1}(M)$). Hence $\theta(T)$ is contained in $M'' = M$. This finishes the proof.
\end{proof}

\subsection{Proof of Theorem \ref{thmE}}

We prove Theorem \ref{thmE}. We will indeed prove a more general theorem, see Theorem \ref{thm-biexact-group}.

Let $M $ be a $\sigma$-finite von Neumann algebra and $M_0\subset M$ a $\sigma$-weakly dense unital $C^*$-subalgebra. 
Consider a family $\mathcal F:= \{ (B_i,E_i) \}_{i\in I}$, where each $B_i \subset 1_{B_i}M1_{B_i}$ is a von Neumann subalgebra with a faithful normal conditional expectation $E_i$ such that for all $i\in I$,
	$$ 1_{B_i}\in M_0 \quad \text{and}\quad E_i( 1_{B_i}M_01_{B_i} )\subset M_0.$$
Put $\widetilde{B}_i:=  B_i \oplus \C (1_M-1_{B_i})$ and extend $E_i$ to $\widetilde{E}_i\colon M\to \widetilde{B}_i$ as a faithful normal conditional expectation. 
We denote the Jones projection by $e_{\widetilde{B}_i}$ and put $e_{B_i}:= e_{\widetilde{B}_i}1_{B_i}=e_{\widetilde{B}_i} J1_{B_i}J$. Note that $e_{\widetilde{B}_i}$ is defined on the standard representation $L^2(M)$ in the following way. Fix a faithful normal state $\varphi_i\in M_*$ with $\varphi_i \circ \widetilde{E}_i = \varphi_i$ and take the implementing vector $\xi_{\varphi_i}\in L^2(M)$ in the positive cone. Then define $e_{\widetilde{B}_i} x \xi_{\varphi_i} = \widetilde{E}_i(x)\xi_{\varphi_i}$. It does not depend on the choice of $\varphi_i$, see \cite[Appendix A]{HI15}.

We define a $M$-boundary piece $\mathbb X_{\mathcal F}\subset \B(L^2(M))$ as the hereditary $C^*$-subalgebra generated by $e_{B_{i}} xJyJ  $ for all $x,y\in M_0$ and $i\in I$. 
Observe that $\mathbb X_{\mathcal F}\subset \B(L^2(M))$ coincides with the norm closure of linear spans of 
	$$  x_1Jy_1J e_{B_{i_1}} T e_{B_{i_2}} x_2^*Jy_2^*J ,\quad x_k,y_k\in M_0,\ i_k\in I,\ k=1,2,\quad T\in \B(L^2(M)).$$

\begin{rem}\upshape
	The $M$-boundary piece $\mathbb X_{\mathcal F}\subset \B(L^2(M))$ does depend on the choice of the dense subalgebra $M_0 \subset M$. However, the closure $\mathbb K^{1,\infty}_{\mathbb X_{\mathcal F}}(M)$ does not depend on $M_0$. Indeed, for the generators  $x_1Jy_1J e_{B_{i_1}} T e_{B_{i_2}} x_2^*Jy_2^*J$ of $\mathbb X_{\mathcal F}$ above, one can approximate each $x_k,y_k$ from $M_0$ to $M$ in the strong$\ast$-topology by bounded nets. Then this provides approximations by the weak $(M,M')$-$(M,M')$-topology.
\end{rem}

The next theorem generalizes \cite[Proposition 6.13]{DP23}, see also \cite[Theorem 6.7 and 6.9]{DP23}.

\begin{thm}\label{thm-biexact}
	Keep the setting and assume that $M$ is biexact relative to $\mathbb X_{\mathcal F}$. 
Then for any projection $p\in M$ and any von Neumann subalgebra $A\subset pMp$ with expectation, either one of the following conditions holds.
\begin{enumerate}
	\item[$\rm (i)$] We have $A\preceq_M B_i$ for some $i\in I$.

	\item[$\rm (ii)$] The relative commutant $A'\cap pMp$ is amenable.

\end{enumerate}
\end{thm}
\begin{proof}
	Assume first that $A$ does not have any direct summand that is semifinite and properly infinite. Assume that (1) does not hold. Then by Corollary \ref{cor-intertwining1},  take $\Psi\in \D(A\subset p \B(L^2(M))p)$ such that  $\Psi(\mathfrak m_{\widetilde{E}_{\widetilde{B}_i}}J1_{B_i}J)=0$ for all $i\in I$ and that $\Psi|_{pMp}$ is a faithful normal conditional expectation onto $A'\cap pMp$. It holds that 
	$$\Psi(Jy_1Jx_1 e_{\widetilde{B}_{i}} x_2^* Jy_2^*J) = Jy_1J\Psi(x_1 e_{\widetilde{B}_{i}} x_2^* ) Jy_2^*J=0,\quad \text{for all }x_1,x_2,y_1,y_2\in M1_{B_i},\ i\in I.$$
By the Cauchy--Schwartz inequality, we get $\Psi(\mathbb X_{\mathcal F}) = 0$. 

Since $\Psi|_M$ and $\Psi|_{M'}=\id_{M'} $ are normal and since $\Psi$ is a $M'$-bimodule map,  we can apply Lemma \ref{lem-biexact-ccp}(2) and get  $\Psi(\mathbb S_{\mathbb X_{\mathcal F}}(M))\subset M$. Note that $\Psi$ is continuous from the weak $M$-$M$-topology to the $\sigma$-weak topology by Lemma \ref{lem-biexact-ccp}(1), as  the weak $M$-$M$-topology is stronger than the weak $(M,M')$-$(M,M')$-topology.

Now we assume that $M$ is biexact relative to $\mathbb X_{\mathcal F}$. Then the composed map
	$$ M \hookrightarrow  \mathbb S_{\mathbb X_{\mathcal F}}(M) \to^{\text{compression by $p$}}  p\mathbb S_{\mathbb X_{\mathcal F}}(M)p  \to^{\Psi} pMp $$
is weakly nuclear, because $\Psi$ is continuous from the weak $M$-$M$-topology to the $\sigma$-weak topology. The restriction of this map on $A'\cap pMp$ is the identity map, hence by composing a normal conditional expectation from $pMp$ onto $A'\cap pMp$, we conclude that $A'\cap pMp$ is semidiscrete. We get the conclusion in this case.

We next prove the general case. For a general $A$, one can find an increasing sequence of  projections $p_n \in A$, converging to $p=1_A$ strongly, such that the central support of $p_n$ in $A$ is $p$ and that $p_nAp_n$ does not have a direct summand that is semifinite and properly infinite. Suppose that (i) does not hold. Then by \cite[Remark 4.2(4)]{HI15}, we have that $p_nAp_n\not\preceq_M B_i$ for all $i$. Hence we get that $(p_nAp_n)' \cap p_nM p_n = p_n(A'\cap pMp)p_n$ is amenable. Then we get that $A'\cap pMp$ is amenable. This finishes the proof. 
\end{proof}

\subsection*{Rigidity of crossed products by biexact groups}

We use Theorem \ref{thm-biexact} to study crossed product von Neumann algebras. 

Let $\Gamma $ be a discrete group, $(B,\varphi)$ a von Neumann algebra with a faithful normal state, $\alpha \colon \Gamma\actson B$ an action. We denote by $J_\varphi$ the modular conjugation on $L^2(B,\varphi)$. Then the modular conjugation $J$ on the standard representation $L^2(B\rtimes G) = L^2(G,L^2(B,\varphi)) \simeq L^2(B,\varphi)\otimes \ell^2(\Gamma)$ is given by
\begin{align*}
	&J(1\otimes \lambda_s) J = v_s\otimes \rho_s\quad (s\in \Gamma);\\
	&J\pi_\alpha(x)J = J_\varphi x J_\varphi\otimes 1\quad (x\in B),
\end{align*}
where $\rho_s$ is the right regular representation given by $\rho_s\delta_t = \delta_{ts^{-1}}$ for all $s,t\in \Gamma$. 

The following theorem generalizes results in \cite{Oz04,HV12,Is12}. It removes the finiteness assumption on subalgebras.

\begin{thm}\label{thm-biexact-group}
	Let $\alpha \colon \Gamma \actson B$ be an action of a discrete group on a $\sigma$-finite von Neumann algebra. Put $M:=B\rtimes_\alpha \Gamma$. Assume that $B$ is amenable and that $\Gamma$ is bi-exact relative to $\mathcal G$, where $\mathcal G$ is a family of subgroups in $\Gamma$. 

Then for any projection $p\in M$ and any von Neumann subalgebra $A\subset pMp$ with expectation, we have either one of the following conditions.
\begin{itemize}
	\item[$(\rm{i})$] There exists $\Lambda \in \mathcal G$ such that $A\preceq_{M} B\rtimes_\alpha \Lambda$.
	\item[$(\rm{ii})$] The relative commutant $A'\cap pMp$ is amenable.

\end{itemize}
\end{thm}
\begin{proof}
	Let $\mathbb X\subset \B(L^2(M))$ be the hereditary $C^*$-subalgebra generated by $\B(L^2(B))\otm c_0(\Gamma;\mathcal G)$. Observe that 
	$$\mathbb X= \overline{c_0(\Gamma;\mathcal G)\B(L^2(M))c_0(\Gamma;\mathcal G)}, $$
and that $M_0:=B\rtimes_{\rm red}\Gamma$ and $JM_0J$ are contained in the multiplier algebra of $\mathbb X$. Then $\mathbb X$ is a $M$-boundary piece and by \cite[Proposition 8.3]{DP23}, $M$ is biexact relative to $\mathbb X$.

Let $E_\Lambda\colon M\to B\rtimes\Lambda$ be canonical faithful normal conditional expectations for $\Lambda\in \mathcal G$, and put $\mathcal F:=\{( B\rtimes \Lambda,E_\Lambda )\}_{\Lambda\in \mathcal G}$. Consider the $M$-boundary piece $\mathbb X_{\mathcal F}$ with the dense algebra $M_0\subset M$. Then it is straightforward to show that $\mathbb X = \mathbb X_{\mathcal F}$. Thus, we get that $M$ is biexact relative to $\mathbb X_{\mathcal F}$, and the theorem follows by Theorem \ref{thm-biexact}.
\end{proof}

\subsection*{Unique prime factorization results}

We next use Theorem \ref{thm-biexact} to study tensor product von Neumann algebras. The following theorem generalizes results in \cite{OP03,HI15}.

\begin{thm}\label{thm-biexact-UPF}
	Let $M_1,\ldots,M_m$ be $\sigma$-finite and biexact von Neumann algebras and put  $M:=M_1\ovt \cdots \ovt M_m$. 
Then for any projection $p\in M$ and any von Neumann subalgebra $A\subset pMp$ with expectation, we have either one of the following conditions.
\begin{itemize}
	\item[$(\rm{i})$] There exists $i\in \{1,\ldots,m\}$ such that $A\preceq_{M} M_i^c$, where $M_i^c:= \left(\overline{\bigotimes}_{j\neq i} M_j \right)\ovt \C 1_{M_i}$.

	\item[$(\rm{ii})$] The relative commutant $A'\cap pMp$ is amenable.

\end{itemize}
\end{thm}
\begin{proof}
	Define $\mathbb X$ as the hereditary $C^*$-algebra generated by 
	$$ \K(L^2(M_i))\otm \B(L^2(M_i^c)) \subset \B(L^2(M)),\quad i=1,\ldots,m.$$
Then it is easy to see that $\mathbb X$ is the norm closure of all the linear spans of 
	$$\K(L^2(M_i))\B(L^2(M))\K(L^2(M_j))\quad \text{for } i,j=1,\ldots,m,$$
where we identify $\K(L^2(M_i))\simeq \K(L^2(M_i))\otm \C1_{M_i^c}$ for all $i$. 
Define a $\sigma$-weakly dense $C^*$-subalgebra by
	$$M_0:=M_1\otm \cdots \otm M_m\subset M .$$
Then $M_0$ and $JM_0J$ are contained in the multiplier algebra of $\mathbb X$, hence $\mathbb X$ is a $M$-boundary piece. Then by (the proof of) \cite[Proposition 6.14]{DP23}, $M$ is biexact relative to $\mathbb X$.

For each $i$, fix a faithful normal state $\varphi_i \in (M_i)_\ast$ and definite a faithful normal conditional expectation $E_i:= \id\otimes\varphi_i\colon M \to M_i^c $. 
Put $\mathcal F:= \{ (M_i^c ,E_i) \}_{i=1}^m$ and define $\mathbb X_{\mathcal F}$ with the dense algebra $M_0 \subset M$. 
Then it is straightforward to show that $\mathbb X = \mathbb X_{\mathcal F}$. Thus, we get that $M$ is biexact relative to $\mathbb X_{\mathcal F}$, and the theorem follows by Theorem \ref{thm-biexact}.
\end{proof}

The next corollary follows by the same argument as in \cite[Theorem 7]{OP03}. Note that when $M$ is of type III, the case where $M=N$ was already proved in \cite{Is14,HI15,DP23}, hence the case where $M$ is of type III and $N$ is a subfactor is new. Therefore the assumption that each $M_i$ is properly infinite is not essential. 

\begin{cor}\label{cor-biexact-UPF}
	For each $1 \leq i \leq m\in \N$, let $M_i$ be a non-amenable, biexact, and properly infinite factor. Put $M:=M_1\ovt \cdots \ovt M_m$. Let $N=N_1\ovt \cdots \ovt N_n\subset M$ be a tensor product factor with expectation such that at most one of $N_j$ is amenable. Then $m\leq n$. If in addition $m=n$ and all $N_j$ are non-amenable, then $N'\cap M$ is atomic, and up to permutation and unitary conjugacy and reduction, we get $N_i \subset M_i$ for all $i$. 
\end{cor}

\subsection{Examples}
\label{Examples_prime}

In this subsection, we construct concrete examples of group actions, which provides relatively solid factors.

\begin{lem}\label{lem-relative-solid-prime}
	Let $B\subset M$ be an inclusion of $\sigma$-finite von Neumann algebras with expectation. Assume that $B$ is amenable and $B\subset M$ is relatively solid. Then every non-amenable subfactor of $M$ with expectation is prime or McDuff. 
\end{lem}
\begin{proof}
	Let $N\subset M$ be a non-amenable subfactor with expectation and assume that $N$ is not prime. 
Let $N = P\ovt Q$ be a tensor decomposition, and we may assume that $Q$ is not amenable. Then since $P'\cap M$ is not amenable as well, we get $P\preceq_M B$ and $P$ is amenable. Hence $N$ is a McDuff factor.
\end{proof}

\begin{exa}\upshape
	Let $\Gamma$ be a discrete biexact group and $H_\R$ a real Hilbert space. Consider  orthogonal representations $\pi\colon \Gamma \actson H_\R$ and $U\colon \R\actson H_\R$, and assume that they commute. Let $\Gamma(U_t,H_\R)''$ be the von Neumann algebra arising from the CAR functor, see \cite{HT70}. Since $\pi$ and $U$ commute, we have an associated action $\alpha\colon \Gamma\actson \Gamma(U_t,H_\R)''$, which can be defined via the extended representation of $\pi$ on the Fock space. 
Then since $\Gamma(U_t,H_\R)''$ is amenable, by Theorem \ref{thm-biexact-group}, the inclusion $\Gamma(U_t,H_\R)''\subset \Gamma(U_t,H_\R)''\rtimes \Gamma$ is relatively solid.

Assume further that $\pi$ has stable spectral gap, that is, $\pi\otimes \rho$ has no almost invariant vectors for all orthogonal representations $\rho$ of $\Gamma$. Then the extension of $\pi$ on the Fock space has no almost invariant vectors, hence it holds that the action $\alpha$ is strongly ergodic. This implies that the crossed product $\Gamma(U_t,H_\R)''\rtimes \Gamma$ is a full factor, see \cite{HI15b}. 
In particular, it is a prime factor by Lemma \ref{lem-relative-solid-prime}. Since $\Gamma(U_t,H_\R)''$ can be of type $\rm III_{\lambda}$ $(0<\lambda\leq 1)$, we are able to construct plenty of new examples of  prime type III factors.
\end{exa}

\begin{exa}\upshape
	Let $\Gamma$ be a discrete non-amenable group acting on a countable set $I$ with finite stabilizers. Let $(B_0,\varphi_0)$ be an amenable von Neumann algebra with a faithful normal state. 
Then the generalized Bernoulli action $\Gamma \actson \overline{\bigotimes}_{I}(B_0,\varphi_0)=:(B,\varphi)$ satisfies the assumptions on Theorem \ref{thm-Bernoulli} and the inclusion $L\Gamma \subset B\rtimes \Gamma$ is relatively solid. 

Assume further that $\Gamma$ is biexact and $\Gamma \actson N$ is any action on an amenable von Neumann algebra. Then by Theorem \ref{thm-biexact-group} and Corollary \ref{cor-Bernoulli}, the inclusion $N\subset (N\ovt B)\rtimes \Gamma$ is relatively solid. 
\end{exa}

We next construct group actions on semifinite von Neumann algebras whose crossed products are of type III. We prove a lemma.

\begin{lem}\label{lem-dual-action}
	Consider the following setting.
\begin{itemize}
	\item Let $\theta \colon \R\actson (N,\Tr_N)$ be a centrally ergodic action on a type $\rm II_\infty$ von Neumann algebra with trace. I is trace scaling in the sense that $\theta_t (\Tr_N) = e^t \Tr_N$ for all $t\in \R$. 

	\item Let $\beta\colon \Gamma \actson (B,\tau_B)$ be an outer action of a discrete group on a $\rm II_1$ factor with trace.

	\item Let $\pi\colon \Gamma \to \R$ be a group homomorphism.

\end{itemize}
We consider following  actions 
\begin{align*}
	& \alpha \colon \Gamma \actson N\ovt B ;\quad \alpha_g = \theta_{\pi(g)} \otimes \beta_g\quad \text{for }g\in \Gamma;\\
	& \widehat{\alpha} \colon \Gamma \actson \mathcal Z(N)\ovt L^\infty(\R);\quad \widehat{\alpha}_g = \theta_{\pi(g)} \otimes \tau_{-\pi(g)}\quad \text{for }g\in \Gamma,
\end{align*}
where $\tau_{t}$ is given by $\tau_t(f)(s)=f(s-t)$ for all $f\in L^\infty(\R)$ and $s,t\in \R$. 
Then the crossed product $M=(N\ovt B)\rtimes_\alpha \Gamma$ is a diffuse factor, whose flow of weight coincides with 
	$$\id\otimes \tau\colon \R \actson [\mathcal Z(N)\ovt L^\infty(\R)]^{\widehat{\alpha}}.$$
In particular, the following assertions hold.
\begin{enumerate}

	\item If $N$ is a type $\rm II_\infty$ factor, then $M$ is of
\begin{itemize}

	\item a type $\rm II_\infty$ factor if $\pi(\Gamma)=\{0\}$;

	\item a type $\rm III_\lambda$ factor if $\pi(\Gamma) = T \Z$, where $T>0$ and $\lambda=e^{-T}$;

	\item a type $\rm III_1$ factor if $\pi(\Gamma)\leq \R$ is a dense subgroup.

\end{itemize}

	\item If $\pi(\Gamma)\leq \R$ is a dense subgroup and $\mathcal Z(N)$ is diffuse, then the flow of weight coincides with the restriction of $\theta$ on $\mathcal Z(N)$.

\end{enumerate}
\end{lem}
\begin{proof}
	Since $\beta$ is properly outer, so is the diagonal action $\alpha$. Since $\alpha$ is centrally ergodic, $M$ is a factor. Since $N\ovt B \subset M$ is with expectation and $N\ovt B$ is diffuse and properly infinite, it holds that $M$ is also diffuse and properly infinite.

	Let $\Tr:=\Tr_N\otimes \tau_B$ and observe that $\alpha$ is $\Tr$-scaling in the sense that $\alpha_g(\Tr) = e^{\pi(g)}\Tr$ for $g\in \Gamma$. Let $\varphi:=\widehat{\Tr}$ be the dual weight on $M$. Then the modular action $\sigma^{\widehat{\varphi}}$ satisfies 
	$$\sigma^{\widehat{\varphi}}_t|_{N\ovt B} = \sigma^\varphi_t,\quad \sigma^{\widehat{\varphi}}(\lambda^{\Gamma}_g) =\lambda_g^\Gamma [D\Tr\circ \alpha_g : D\Tr]_t = e^{-\ri t\pi(g) }\lambda_g^\Gamma \quad \text{for } g\in \Gamma, \ t\in \R,$$
where $[D\Tr\circ \alpha_g : D\Tr]_t$ is the Connes cocycle and $\lambda_g^\Gamma$ is the canonical unitary in $L\Gamma$ corresponding to $g\in\Gamma$. Then by flipping $\ell^2(\Gamma)$ and $L^2(\R)$, there is the following canonical $\ast$-isomorphism:
	$$C_{\widehat{\varphi}}(M)= [(N\ovt B)\rtimes_\alpha \Gamma]\rtimes_{\sigma^{\widehat{\varphi}}} \R \simeq  [N\ovt B\ovt L\R]\rtimes_{\widehat{\alpha}}\Gamma $$
where the action $\widehat{\alpha}\colon \Gamma \actson N\ovt B\ovt L\R$ is given by 
	$$\widehat{\alpha}_g|_{N\ovt B} = \alpha_g,\quad \widehat{\alpha}_g(\lambda^{\R}_t) =  e^{\ri t\pi(g) }\lambda_t^\R \quad \text{for } g\in \Gamma, \ t\in \R,$$
where  $\lambda_t^\R$ is the canonical unitary in $L\R$. Under the natural isomorphism $L\R \ni \lambda_t^\R \mapsto e^{\ri t \cdot}\in L^\infty(\R)$ by the Pontryagin duality, $\widehat{\alpha}_g$ on $L\R$ coincides with $\tau_{\pi(g)}$. 
Since $\beta$ is properly outer, $\widehat{\alpha}$ is properly outer as well, hence it holds that
	$$ \mathcal Z(C_\varphi(M)) \simeq \mathcal Z(N\ovt B \ovt L^\infty(\R))^{\widehat{\alpha}} = [\mathcal Z(N)\ovt L^\infty(\R)]^{\widehat{\alpha}}.$$
Then the dual action on $\mathcal Z(C_\varphi(M)) $ coincides with the translation $\tau$ on $L^\infty(\R)$ on the right hand side, so the flow of weight of $M$ has the desired identification.

	(1) Since $N$ is a factor, the flow of weight of $M$ is the translation on $L^\infty(\R)^{\widehat{\alpha}}$. By seeing the range $\pi(\Gamma)\leq \R$, the conclusion easily follows.

	(2) Let $\gamma\colon \R \actson \mathcal Z(N)\ovt L^\infty(\R)$ be given by $\gamma_t = \theta_t \otimes \tau_{-t}$ for $t\in \R$. Then since $\widehat{\alpha}_g = \theta_{\pi(g)}\otimes \tau_{\pi(g)}$ for $g\in \Gamma$ and since $\pi(\Gamma)\leq \R$ is dense, the fixed points of the action $\widehat{\alpha}\colon \Gamma \actson \mathcal Z(N)\ovt L^\infty(\R)$ coincides with the ones of $\gamma$. Then since $\tau_t$ is the translation, by Fell's absorption trick, the flow of weight of $M$ coincides with the $\theta\otimes \id \colon \R\actson \mathcal Z(N)\ovt \C 1_{L^\infty(\R)}$. 
\end{proof}

\begin{exa}\upshape
	Let $N_0$ be the unique amenable type $\rm III_1$ factor and consider the trace scaling action $\theta \colon \R\actson N$ on the continuous core $N=C(N_0)$. Let $\mathbb F_m \to \R$ be any homomorphism for some $m\geq 2$, and $\mathbb F_m \actson N$ any outer action on the amenable type $\rm II_1$ factor (e.g.\ Bernoulli shift action). Then the crossed product $M=(N\ovt B)\rtimes_\alpha \mathbb F_m$ in Lemma \ref{lem-dual-action}(1) provides examples of actions that satisfy the assumptions on Corollary \ref{thm-biexact-group}. Then the crossed product can be a factor of type $\rm II_\infty$ or $\rm III_\lambda$ $(0<\lambda \leq 1)$.

	Let $\gamma \colon \R \actson L^\infty(X)$ be any ergodic continuous action such that it is not conjugate to translations of $\R \actson \R/\Lambda$ for all closed subgroups $\Lambda \leq \R$.  Then there exists a unique amenable type $\rm III_0$ factor whose flow of weight is $\gamma$ \cite{Kr75, Co75}. We denote by $\theta \colon \R\actson N$ the dual action on the continuous core of the factor, so that the restriction of $\theta$ on the center of $N$ is $\gamma$. 
Then using the same construction as in the last paragraph and by Lemma \ref{lem-dual-action}(2), we can construct examples of type $\rm III_0$ factors whose flow of weight is $\gamma$. 

We note that all these examples are McDuff. Indeed by construction all above $N$ has a decomposition $N = R\ovt N_{00}$, where $R$ is the amenable $\rm II_1$ factor, such that $\theta$ acts trivially on $R$. Hence $M$ also has a similar tensor decomposition. By Corollary \ref{thm-biexact-group}, if $M = P\ovt Q$ and $Q$ is of type III, then since $Q\not\preceq_M N\ovt B$ (because $N\ovt B$ is semifinite), $P$ satisfies $P\preceq_M N\ovt B$. This shows that $P$ is amenable and semifinite. Hence $M$ does not have any tensor decomposition with an amenable type III factor, while it dose admit a tensor decomposition with the amenable $\rm II_1$ factor. 
\end{exa}

\section{Galois correspondences}
\label{Galois correspondences}

Let $G$ be a totally disconnected group and $K_0\leq G$ a compact open subgroup. We say that an action $\alpha\colon G \actson B$ on a diffuse factor is \textit{properly outer relative to $K_0$} if $\alpha$ satisfies the following property: for any $g\in G$, if there exists nonzero $b\in B$ such that $\alpha_g(x)b=bx$ for all $x\in B^{K_0}$, then $g$ is contained in $K_0$. In this case, for each open subgroup $K\leq K_0$, we have $(B^K)'\cap (B\rtimes G) = LK$  \cite[Proposition 4.8]{BB17}.

\subsection{Proof of Theorem \ref{thmF}}

Let $G\actson^\alpha B$ be an action of a locally compact group on a von Neumann algebra $B$, and put $M=B\rtimes G$. Let $H\leq G$ be an open subgroup and assume that $B^H\subset B\rtimes H$ is with expectation $E_{B^H}$. We extend it to $M$ by composing the canonical conditional expectation $E_{B\rtimes H}\colon M \to B\rtimes H$. Let $q\in B^H$ be a projection such that $qB^Hq$ has no direct summand that is semifinite and properly infinite. 
Let $\varphi_{qB^Hq}$ be a faithful normal state on $qB^Hq$ such that it is a trace on the finite direct summand. Put $\varphi:= \varphi_{qB^Hq}\circ E_{B^H}$. Then by Theorem \ref{thm-Marrakchi}, the unique conditional expectation 
	$$E:=E_{(qB^Hq)'\cap qMq} \colon qMq\to  (qB^Hq)'\cap qMq ,$$
which preserves  $\varphi$, is contained in $\D(qB^Hq \subset qMq)$.

\begin{lem}\label{lem-Galois}
	Keep the setting and let $B\subset N\subset M$ be an intermediate von Neumann subalgebra. 
\begin{enumerate}
	\item If $x\in M$ satisfies $qE_{B\rtimes H}( x )q \neq 0$, then there exists $b\in qBq$ such that $ E(q bx q) \neq 0 $.

	\item If $g\in G$, $x_0\in N$, and a projection $p\in LH$ satisfies $qE_{B\rtimes H}( \lambda_g^* x_0p )q \neq 0$, then there exists $b\in B$ such that
	$$ E(q\lambda_g^* bx_0q)\in  \lambda_g^* N\quad \text{and}\quad E(q\lambda_g^* bx_0q)p\neq0.$$
If $p$ commutes with $N \cap (B^H)'$, then the partial isometry $v$ given by the polar decomposition of $E(q\lambda_g^* bx_0q)$ satisfies
	$$v \in (q\lambda_g^*Nq) \cap (qB^Hq)'\quad\text{and}\quad vp\neq0.$$

\end{enumerate}
\end{lem}
\begin{proof}
	(1) Put $X:=E_{B\rtimes H}( q x q )  \neq 0$. Since $\varphi$ is faithful and normal on $q(B\rtimes H)q = qBq\rtimes H$, there exists $b\in qBq$ and $f\in C_c(H)$ such that $0\neq \varphi(fb X )$. Then we have
\begin{align*}
	0
	\neq \varphi(fb X )
	&= \varphi(fbE_{q(B\rtimes H)q}(q xq ) )
	= \varphi(E_{q(B\rtimes H)q}(f bx  q))\\
	&= \varphi(E(f bx q ))
	= \varphi(qfE(qbxq)).
\end{align*}
We get $E(qbxq)\neq 0$.

	(2) Putting $x=\lambda_g^* x_0 p$, we apply (1) and find $b_0\in B$ such that 
	$$0\neq E(qb_0\lambda_g^* x_0 pq) = E(q\lambda_g^* \alpha_g(b_0)x_0 q)pq = 
E(q\lambda_g^* b x_0 q)p,\quad \text{where $b=\alpha_g(b_0)$}.$$
We get the second condition. For the first condition, for each $y\in N$ and $u\in \mathcal U(qB^Hq)$,
	$$\lambda_g \Ad(u)(q\lambda_g^* yq ) = \lambda_g u \lambda_g^* qy u^*q  =  \alpha_g(u) y u^* \in N.$$
Since $E\in \D(qB^Hq \subset qMq)$, this implies $\lambda_gE(q\lambda_g^* yq)\in N$. Putting $y= bx_0$, we get the first condition.

	Put $a := E(qu_g^* bx_0q) \in (\lambda_g^* \alpha_g(q)Nq)\cap (qB^Hq)'$ and let $a= v|a|$ be the polar decomposition. Then it is easy to see that $v\in (\lambda_g^* \alpha_g(q)Nq)\cap (qB^Hq)'$ and $|a|\in (qNq)\cap (qB^Hq)'= q(N'\cap B^H)q$. If $p$ commutes with $N'\cap B^H$, then it commutes with $|a|$. In particular, $0\neq ap = vp (p|a|p)$ is the polar decomposition and we get $vp\neq 0$. 
\end{proof}

\begin{proof}[\bf Proof of Theorem \ref{thmF}]
	We follow the proof of \cite[Theorem 1.2]{BB17}. The property that $B$ is semifinite is only used in Steps 1 and 2 of that proof. So to prove Theorem \ref{thmF}, we have only to show the following claim. For each compact subgroup $K \leq G$, let $p_K\in LK$ be the central and minimal projection defined by
	$$ p_K :=\frac{1}{\mu_G(K)}\int_{K} \lambda_g \, d\mu_G(g), $$
where $\mu_G$ is a fixed left Haar measure on $G$.

\begin{claim}
	Let $K\leq K_0 $ be a compact open subgroup and let $g\in G$. If there exists $x_0\in N$ such that
	$$  E_{B\rtimes K}(\lambda_g^* x_0)p_K  =E_{B\rtimes K}(\lambda_g^* x_0p_K)  \neq 0,$$
then we have $\lambda_gp_K \in (N)_1 p_K$.
\end{claim}
\begin{proof}[{\bf Proof of Claim}]
	Since $K$ is compact, the inclusion $B^K \subset B$ is with expectation $E_{B^K}$. Since $E_{B^K}$ commutes with the group action of $K$, we can extend it to $B^K\ovt LK\subset B\rtimes K$. Since $B^K \subset B^K\ovt LK$ is with expectation, it holds that $B^K \subset B\rtimes K$ is with expectation. In particular, we can use Lemma \ref{lem-Galois}. 
We note that the resulting expectation is not a canonical one, as the inclusion $B^K \subset B^K\ovt LK$ does not have any canonical expectation. 

Since $(B^K)'\cap B=\C$ by the minimality assumption, $B^K$ is a factor. If it is semifinite and properly infinite, then take a finite projection $q\in B^K$ such that 
	$$  qE_{B\rtimes K}(\lambda_g^* x_0)p_Kq  \neq 0.$$
In the other case, we simply put $q=1$. Then we apply Lemma \ref{lem-Galois}(2). Since  $(B^K)'\cap M \subset LK$ and since $p_K$ is in the center of $LK$, we can use the last statement of that lemma, and find a partial isometry $v \in (q\lambda_g^*Nq) \cap (LKq)$ with $vp_K\neq0$. 
Then since $p_K$ is central and minimal, there exists $c\in \C$ such that $vp_K = c \, qp_K$. Since $vp_K$ and $qp_K$ are partial isometries, we have $|c|=1$. Since $\lambda_g v\in N$, we conclude that 
	$$ \lambda_g qp_K =(c^{-1}\, \lambda_g v)p_K \in (N)_1p_K.$$
Finally if $B$ is semifinite and properly infinite, then since $\lambda_g qp_K\in (N)_1p_K$ for all finite projections $q\in B^K$ sufficiently large, we get $\lambda_gp_K \in (N)_1p_K$. 
\end{proof}
As explained above, using the claim, we can follow the proof of \cite[Theorem 1.2]{BB17}.
\end{proof}

\subsection{Examples}
\label{Examples_Galois}

The next proposition is a slight generalization of \cite[Proposition 4.10]{BB17}.

\begin{prop}\label{prop-Galois}
	Let $G$ be a totally disconnected group and $K\leq G$ a compact open subgroup such that $\bigcap_{g\in G} gKg^{-1} = \{e\}$. Let $(B_0,\varphi_0)$ be a diffuse factor with a faithful normal state. Consider the generalized Bernoulli action
	$$\alpha\colon G\actson \overline{\bigotimes}_{G/K}(B_0,\varphi_0) =:(B,\varphi) ,$$
induced from the left translation action $G\actson G/K$. Then, for each $\sigma$-finite factor $N$, the action $\beta:=\alpha\otimes \id_N \colon G \actson B\ovt N$ is  properly outer relative to $K$ and $\beta|_K$ is minimal.
\end{prop}
\begin{proof}
	In \cite[Proposition 4.10]{BB17}, it was proved that $\alpha$ is a faithful action and that $\alpha|_K$ is minimal. So, it holds that $\beta$ is faithful and that $\beta|_K$ is minimal. We have to show that $\beta$ is properly outer relative to $K$. For this, let $g\in G$ and $0\neq b\in B\ovt N$ be such that
	$$\beta_g(x)b=bx,\quad \text{for all }x\in (B\ovt N)^{\beta|_K}=B^{\alpha|_K}\ovt N.$$
We restrict this condition to one on $x\in B^{\alpha|_K}$. Then if we denote by $B^{(hK)}$ the tensor component of $B$ corresponding to $hK$ $(h\in G)$, since $\alpha_g$ is a $\ast$-isomorphism from $B^{(K)}$ onto $B^{(gK)}$, we get
	$$B^{(K)}\preceq_{B\ovt N} B^{(gK)} .$$
Suppose now that $g\not \in K$. Then this implies that $B^{(K)}\preceq_{B\ovt N} (B^{(K)})' \cap (B\ovt N)$, hence by \cite[Lemma 4.3]{Is17}, we get $B^{(K)}\preceq_{B^{(K)}}\C$. This means that $B^{(K)}$ has a minimal projection, which is a contradiction. We conclude $g\in K$, as desired.
\end{proof}

\begin{exa}\upshape
	Let $\alpha \colon G\actson (B,\varphi)$ be as in Proposition \ref{prop-Galois} and assume that $B_0$ is a $\rm II_1$ factor. Let $N$ be a $\sigma$-finite factor. Then, 
by Proposition \ref{prop-Galois} and Theorem \ref{thmF}, the crossed product $(B\ovt N)\rtimes_{\beta} G$ satisfies the intermediate subfactor property. In this case, the flow of weight of $B\ovt N$ coincides with the one of $N$, hence we can choose $B\ovt N$ as a type $\rm III_0$ factor with any prescribed flow of weight. In particular, this solves a question in \cite[QUESTION 1.5]{BB17}.
\end{exa}

\small{

}

\end{document}